\documentclass[11pt]{article}

\usepackage{adjustbox}
\usepackage{amssymb}
\usepackage{amsmath,amsfonts,amsthm}
\usepackage{latexsym}
\usepackage{mathrsfs}
\usepackage{color}
\usepackage{algorithm}
\usepackage{algorithmic}
\usepackage{subfigure}
\usepackage{dsfont}
\usepackage{multirow}
\usepackage[overload]{empheq}
\usepackage{enumitem}
\usepackage{setspace}
\usepackage{hyperref}       
\usepackage{cleveref} 
\usepackage{url}            
\usepackage{booktabs}       
\usepackage[a4paper]{geometry}
\usepackage{bm}
\usepackage{graphicx}
\usepackage[numbers]{natbib}
\usepackage{makecell}

\newcommand{\beqa}{\begin{eqnarray}}
\newcommand{\eeqa}{\end{eqnarray}}
\newcommand{\beqas}{\begin{eqnarray*}}
\newcommand{\eeqas}{\end{eqnarray*}}
\newcommand{\ba}{\begin{array}}
\newcommand{\ea}{\end{array}}
\newcommand{\bi}{\begin{itemize}}
\newcommand{\ei}{\end{itemize}}

\newcommand{\RN}[1]{%
  \textup{\uppercase\expandafter{\romannumeral#1}}%
}

\DeclareMathOperator{\argmax}{\arg\max}

\newtheorem{lemma}{Lemma}
\newtheorem{thm}{Theorem}
\newtheorem{coro}{Corollary}

\newtheorem{prop}{Proposition}

\newtheorem{assumption}{Assumption}
\newtheorem{fact}{Fact}
\newtheorem{remark}{Remark}

\newcounter{spb}
\newcommand{\email}[1]{\protect\href{mailto:#1}{#1}}

\def\bmw{\bm{W}}

\def\bU{\bm{U}}
\def\bV{\bm{V}}

\def\bW{\bm{W}}

\def\R{\mathbb{R}}

\begin{document}

\title{Error Bound Analysis for the Regularized Loss of \\ Deep Linear Neural Networks\footnote{The authors are listed in alphabetical order. {\bf Funding.} Rujun Jiang is supported by the National Key R\&D Program of China under grant 2023YFA1009300, the Major Program of the National Natural Science Foundation of China (72394360, 72394364). Peng Wang is supported in part by the University of Macau under grants SRG2025-00043-FST and UMDF-TISF-I/2026/013/FST, and in part by the Macau Science and Technology Development Fund (FDCT) 0091/2025/ITP2.}}
\author{ 
Po Chen\thanks{School of Data Science, Fudan University, Shanghai (\email{chenp24@m.fudan.edu.cn}).}
\and Rujun Jiang\thanks{Corresponding author. School of Data Science, Fudan University, Shanghai (\email{rjjiang@fudan.edu.cn}).}
\and Peng Wang\thanks{Department of Artificial Intelligence, University of Macau, China (\email{peng8wang@gmail.com}).}
}
\maketitle\medskip 
\abstract{%
The optimization foundations of deep linear networks have recently received significant attention. However, due to their inherent non-convexity and hierarchical structure, analyzing the loss functions of deep linear networks remains a challenging task. In this work, we study the local geometry of the regularized squared loss of deep linear networks around each critical point. Specifically, we obtain a closed-form characterization of the critical point set building on existing results and establish an error bound for the regularized loss under mild conditions on network width and regularization parameters.
Notably, this error bound quantifies the distance from a point to the critical point set in terms of the current gradient norm, which can be used to derive linear convergence of first-order methods.
To support our theoretical findings, we conduct numerical experiments and demonstrate that gradient descent converges linearly to a critical point when optimizing the regularized loss of deep linear networks. 
}%


\smallskip

\section{Introduction}\label{sec:intro}

Deep learning has achieved great empirical success in various fields, including computer vision \cite{he2016deep}, natural language processing \cite{vaswani2017attention}, and healthcare \cite{esteva2019guide}. Optimization is a key component of deep learning, playing a pivotal role in formulating learning objectives and training neural networks. In general, optimization problems arising in deep learning are highly non-convex and difficult to analyze due to the inherent nonlinearity and hierarchical structures of deep neural networks. Even in the context of linear neural networks, which are the most basic form of neural networks, our theoretical understanding remains far from complete and systematic.  In this work, we study the following problem based on deep linear networks: 
\begin{align}\label{eq:F}
\min_{\bm W}\ F(\bm W) := \left\|\bm W_L\cdots\bm W_1 -  \bm Y \right\|_F^2 + \sum_{l=1}^L \lambda_l \|\bm W_l\|_F^2, 
\end{align}
where $L\ge 2$ denotes the number of layers, $ \bm Y \in \R^{d_L\times d_0}$ denotes the target matrix, $\bm W_l \in \R^{d_l \times d_{l-1}}$ denotes the $l$-th weight matrix, $\bm W = \{\bm W_l\}_{l=1}^L$ denotes the collection of all weight matrices, and $\lambda_l > 0$ for all $l$ are regularization parameters. Notably, such a problem captures a wide range of deep learning problems, including deep matrix factorization \cite{arora2019implicit,de2021survey,trigeorgis2016deep}, neural collapse \cite{Han2021,zhou2022optimization}, and low-rank adaptation of learning models \cite{hu2022lora,yarascompressible}.

In the literature, a large amount of work has been dedicated to studying the following problem: $\min_{\bm W} \left\|\bm W_L\cdots\bm W_1\bm X -  \bm Z \right\|_F^2$, 
where $(\bm X, \bm Z) \in \R^{d_0\times N} \times \R^{d_L\times N}$ denotes the data input. A common assumption in many theoretical studies (see, e.g., \cite{arora2018convergence,bartlett2018gradient,domine2024lazy,wu2019global}) is that the data is whitened, i.e., $\bm X\bm X^T = \bm I_d$. Under this assumption, one can verify that $\min_{\bm W} \left\|\bm W_L\cdots\bm W_1\bm X -  \bm Z \right\|_F^2$ is equivalent to $\min_{\bm W} \left\|\bm W_L\cdots\bm W_1 -  \bm Z\bm X^T \right\|_F^2$. Obviously, this formulation is an unregularized counterpart of Problem \eqref{eq:F}. We should point out that regularization is commonly used during the training of neural networks to prevent overfitting, improve generalization, and accelerate convergence \cite{krizhevsky2012imagenet,zhang2021understanding}. However, the existing theoretical studies mainly focus on unregularized deep neural networks and cannot be directly applied to Problem \eqref{eq:F}, even when the data is whitened; see, e.g., \cite{arora2018convergence,bah2022learning,bartlett2018gradient,chitour2023geometric,Hu2020Provable,nguegnang2024convergence,wu2019global}. 

In practice, (stochastic) gradient descent (GD) and its variants are among the most widely used first-order methods for deep learning \cite{kingma2014adam}. Over the past few years, substantial progress has been made in studying the convergence behavior of GD for optimizing deep linear networks, particularly in the context of the unregularized problem $\min_{\bm W} \left\|\bm W_L\cdots\bm W_1 -  \bm Y\right\|_F^2$. For example, \citet{bartlett2018gradient} showed that GD with the identity initialization converges to an $\epsilon$-optimal solution within a polynomial number of iterations. Later, \citet{arora2018convergence} further improved the convergence result, showing that GD converges linearly to a global optimum when $\min\{d_1,\dots,d_{L-1}\} \ge \min\{d_0,d_L\}$, the initial weight matrices are approximately balanced, and the initial loss is smaller than a threshold. Meanwhile, \citet{wu2019global} showed that GD with zero-asymmetric initialization avoids saddle points and converges to an $\epsilon$-optimal solution in $O(L^3\log(1/\epsilon))$ iterations. Other works also proved similar global convergence and convergence rate results of GD under different assumptions; see, e.g., \citep{bah2022learning,chitour2023geometric,Hu2020Provable,nguegnang2024convergence,wu2019global}. Despite these inspiring results, the existing analyses suffer from some notable limitations. First, they are highly specific to a particular problem, relying on the analysis of gradient dynamics under carefully designed weight initialization schemes.
In addition, the existing analyses only apply to analyze the convergence to global optimal solutions. However, Problem \eqref{eq:F} and its unregularized counterpart may have other local solutions, to which first-order methods, such as GD, are likely to converge. To the best of our knowledge, the convergence behavior of first-order methods when they approach a critical point—whether a global minimum, a local minimum, or even a saddle point—remains an open question in the literature. To sum up, it remains an unsolved challenge in the literature to develop a unified framework to analyze the convergence behavior of first-order methods when solving Problem \eqref{eq:F}.

To address the above challenge, a promising approach is to study the local geometric structure of Problem \eqref{eq:F} associated with its objective function, such as the error bound condition \cite{liao2024error,rebjock2024fast,zhou2017unified}, the Polyak-\L ojasiewicz (P\L) inequality \cite{polyak1963gradient}, and quadratic growth \cite{drusvyatskiy2018error}. In this work, we mainly focus on the {\em error bound} condition associated with Problem \eqref{eq:F}. Formally, let $\mathcal{W}_F$ denote the set of critical points of Problem \eqref{eq:F}. We say that it possesses an error bound for $\mathcal{W}_F$ if there exist constants $\epsilon,\kappa > 0$ such that  for all $\bm W$ satisfying $\mathrm{dist}(\bm W, \mathcal{W}_F) \le \epsilon$, we have
\begin{align}\label{eq:eb}
\mathrm{dist}(\bm W, \mathcal{W}_F) \le \kappa\|\nabla F(\bm W)\|_F,
\end{align} 
where $\mathrm{dist}(\bm W, \mathcal{W}_F) := \min\{\|\bm W - \bm V\|_F: \bm V \in \mathcal{W}_F\}$ denotes the distance from $\bm W$ to the critical point set $\mathcal{W}_F$. Intuitively, the error bound inequality \eqref{eq:eb} requires the distance from a point to the set of critical points to be bounded by its gradient norm. A unified framework leveraging this condition, along with some algorithm-dependent conditions, has been widely studied to analyze linear convergence of first-order methods \cite{drusvyatskiy2018error,zhou2017unified} or superlinear convergence of second-order methods \cite{yue2019family} in convex optimization. \citet{luo1993} showed that even if the objective function is non-convex but smooth and satisfies the error bound, (projected) GD converges linearly to the solution set.
Moreover, \citet{li2018calculus,liao2024error,rebjock2024fast} demonstrated that the error bound is equivalent to other regularity conditions, such as the P\L\ inequality \cite{polyak1963gradient} and quadratic growth \cite{drusvyatskiy2018error} under mild conditions. These regularity conditions are widely used in the literature to prove linear convergence of GD for optimizing non-convex problems \cite{polyak1963gradient,luo1993,attouch2010proximal,frei2021proxy}. Notably, this type of convergence analysis framework is not limited to studying convergence to global optimal solutions. Instead, it can also be applied to local minima or even saddle points, as long as the error bound holds for the targeted solution sets. This makes such a framework particularly effective for analyzing non-convex optimization problems, whose landscape often contains various types of critical points. 

However, powerful as this approach may seem, proving error bounds for specific problems, particularly non-convex ones, is technically difficult, as it often requires a delicate analysis \cite{zhou2017unified,jiang2019novel,jiang2022holderian,Liu2019,wang2023linear,pang1997error}.
As far as we know, proving the error bound or other regularity conditions for deep neural networks remains relatively underexplored. Recently, \citet{wang2022linear} made progress in this direction by establishing the error bound for the set of global optimal solutions to a special instance of Problem \eqref{eq:F}: $\min  \|\bm W_2\bm W_1  - \bm Y\|_F^2 + \lambda_1\|\bm W_1\|_F^2 +  \lambda_2\|\bm W_2\|_F^2$, where $\lambda_1,\lambda_2 >0$ are regularization parameters and $\bm Y$ is a membership matrix. Moreover, \citet{liu2025} and \citet{tao2021error} extended the analysis to the setting of general data input $\bm Y$. However, these analyses are limited to global optimal solutions of 2-layer linear networks. Extending these results to deeper networks and critical points remains an open problem.

In view of the above discussion, our goal in this paper is to establish an error bound for the set of critical points of Problem \eqref{eq:F}.
    To this end, we first specialize the general critical-point characterization in \cite{chen2025complete}, under the condition that the width of each hidden layer is no smaller than that of the input and output layer (see Assumption~\ref{AS:1}). This yields a simplified and explicit form suitable for our error-bound analysis (see \Cref{thm:opti}).  Then, 
   leveraging this explicit characterization, we show that all critical points of Problem \eqref{eq:F} satisfy the error bound (see \Cref{thm:eb}) under mild conditions on the network width and regularization parameters. More specifically, we identify the sufficient and necessary conditions on the relationship between the regularization parameters and the spectrum of the input data that ensure the error bound holds. Such a result is significant, as it expands the currently limited repertoire of non-convex problems for which the local loss geometry is well understood. Moreover, it is important to note that our work develops new techniques in the process of establishing the error bound to handle the repeated singular values in $\bm Y$ and the complicated structure of the critical point set. These techniques could be of independent interest. The established error bound can be used to establish other regularity conditions, such as the P\L\ inequality and quadratic growth. Building on the error bound of Problem \eqref{eq:F},  we demonstrate that first-order methods can achieve linear convergence to a critical point of Problem \eqref{eq:F}, provided that they satisfy certain algorithm-dependent properties.
    Finally, we conduct numerical experiments in deep linear networks and in more general settings and observe that GD converges linearly to critical points. These results strongly support our theoretical findings.
 
{\bf Notation.} Given a positive integer $n$, let $[n] := \{1,\dots,n\}$. Given a set of numbers $\{\lambda_1,\dots,\lambda_n\}$, we define $\lambda_{\min} := \min\{\lambda_i: i \in [n]\}$ and $\lambda_{\max} := \max\{\lambda_i: i \in [n]\}$. Given a vector $\bm a$, let $\|\bm a\|$ denote its Euclidean norm, $\|\bm a\|_0$ its $\ell_0$ norm (i.e., number of non-zero entries), $a_i$ the $i$-th entry, and $\mathrm{diag}(\bm a)$ the diagonal matrix with $\bm a$ on the diagonal.
Given a matrix $\bm A$, let $\|\bm A\|$ denote the spectral norm, $\|\bm A\|_F$ the Frobenius norm, $a_{ij}$ the $(i,j)$-th element, and $\sigma_i(\bm A)$ the $i$-th largest singular value. We use $\bm 0_{m\times n}$ to denote an all-zero matrix of size $m\times n$, $\bm 0_m$ to denote an all-zero vector of dimension $m$, and simply write $\bm 0$ when the dimension can be inferred from the context. We denote the $i$-th standard basis vector by $\bm e_i$.
We simply write $\nabla_{\bm W_l} F(\bm W)$ as $\nabla_l F(\bm W)$. 

We use $\mathcal{O}^{n\times d}$ to denote the set of all $n\times d$ matrices that have orthonormal columns (in particular, $\mathcal{O}^{n}$ for $n \times n$ orthogonal matrices); $\mathcal{P}^{n}$ to denote the set of all $n\times n$ permutation matrices; $\mathrm{BlkD}(\bm X_1,\dots,\bm X_n)$ to denote the block diagonal matrix whose diagonal blocks are $\bm X_1,\dots,\bm X_n$. Given a matrix $\bm X\in \R^{m\times n}$ and a non-empty closed set $\mathcal{X} \subseteq \R^{m\times n}$, the Euclidean distance from $\bm X$ to $\cal X$ is defined as  $\mathrm{dist}(\bm X, \mathcal{X}) = \min_{\bm Z \in \mathcal{X}} \|\bm X - \bm Z\|_F$; the distance between $\mathcal{X}$ and a closed set $\mathcal{Y} \neq \emptyset$ is defined as $\mathrm{dist}(\mathcal{X},\mathcal{Y}) = \min_{\bm X \in \mathcal{X}, \bm Y \in \mathcal{Y}}\|\bm X - \bm Y\|_F$.  Let $\pi : [d] \rightarrow [d]$ denote a permutation of the elements in $[d]$ with the corresponding permutation matrix $\bm\Pi \in\mathcal{P}^{d}$ satisfying $\bm \Pi(i,j) = 1$ if $j = \pi(i)$ and $\bm \Pi(i,j) = 0$ otherwise for each $i \in [d]$.
Given weight matrices $\{\bm W_l\}_{l=1}^L$ with $\bm W_l \in \R^{d_l\times d_{l-1}}$ for each $l \in [L]$, let $\bm W_{i:1} := \bm W_i\bm W_{i-1}\cdots\bm W_1$ for each $i=2,\dots,L$ and $\bm W_{L:i} := \bm W_L\bm W_{L-1}\cdots \bm W_i$ for each $i=1,\dots,L-1$. In particular, we define $\bm W_{0:1} := \bm I$ and $\bm W_{L:L+1} := \bm I$. For all $i \ge j+1$, let $\bm W_{i:j}=\bm W_{i}\bm W_{i-1}\cdots\bm W_j$ and  $\bm W_{i:j}=\bm W_i$ when $i=j$. All other notations are standard.   

{\bf Organization.}  In \Cref{sec:main}, we present the main results of this paper. In \Cref{sec:proof}, we present the proofs of the main results. We report experimental results in \Cref{sec:expe} and conclude the paper in \Cref{sec:con}.

\section{Main Results}\label{sec:main}
To begin, we define the set of critical points of Problem \eqref{eq:F} as follows: 
\begin{align}
\mathcal{W}_F := \left\{ \bm W = (\bm W_1, \dots, \bm W_L) : \nabla F(\bm W) = \bm 0 \right\}.
\notag
\end{align}
Let $\bm Y = \bm U_Y\bm \Sigma_Y \bm V_Y^T$
be an SVD of $\bm Y \in \R^{d_L\times d_0}$, where $\bm U_Y \in \mathcal{O}^{d_L}$ and $\bm V_Y \in \mathcal{O}^{d_0}$. In addition, 
\begin{align}\label{eq:SY}
    \bm \Sigma_Y = \mathrm{BlkD}\left(\mathrm{diag}(y_1,y_2,\dots,y_{d_{\min}}), \bm 0_{(d_L-d_{\min})\times (d_0-d_{\min})}\right) \in \R^{d_L\times d_0},
\end{align}
    where $d_{\min} := \min\{d_0,d_L\}$ and $y_1\ge y_2 \ge \dots \ge y_{d_{\rm min}} \ge 0$ are singular values. Here, if $d_{\min} = d_0$ (resp., $d_{\min} =d_L$), the lower diagonal block vanishes, leaving a $\bm 0$ block of the appropriate size at the bottom (resp., on the right). Specifically, if $d_{\min} = d_0$, $\bm \Sigma_Y = \begin{bmatrix}
    \mathrm{diag}(y_1,y_2,\dots,y_{d_{\min}}) \\ \bm 0_{(d_L-d_{\min})\times d_0}  
\end{bmatrix}$; otherwise, $\bm \Sigma_Y = \begin{bmatrix}
    \mathrm{diag}(y_1,y_2,\dots,y_{d_{\min}}) & \bm 0_{d_L\times (d_0-d_{\min})}
\end{bmatrix}$. 
To handle the repeated singular values, we introduce the following setup. 
Let \(r_Y:=\operatorname{rank}(\bm Y)\). If \(r_Y=0\), then \(\bm Y=\bm0\), and Proposition~\ref{prop:eb zero} directly proves the desired error bound.
Thus, throughout the subsequent notation and analysis, we assume $r_Y>0$. Let $p_Y$ denote the
number of distinct positive singular values of $\bm Y$.
That is, there exist indices $s_0, s_1, \dots, s_{p_Y}$ such that $0= s_{0} < s_{1} < \dots < s_{p_Y} = r_Y$ and
\begin{align}\label{eq:SY1}
  y_{s_{0}+1} = \dots = y_{s_{1}} > y_{s_{1}+1} = \dots = y_{s_{2}} > \dots > y_{s_{p_Y-1}+1} = \dots = y_{s_{p_Y}}>0.
\end{align}
If \(r_Y<d_{\min}\), then \(y_{r_Y+1}=\cdots=y_{d_{\min}}=0\).
We set \(y_{d_{\min}+1}:=0\).  Let $h_i := s_i - s_{i-1}$ denote the multiplicity of the $i$-th largest positive singular value for each $i \in [p_Y]$.
Throughout the paper, we consistently use the above notation. 
\subsection{Characterization of the Critical Point Set}

To begin, we introduce the following assumption on the widths of network layers. 
\begin{assumption}\label{AS:1}
The network widths satisfy $\min\{d_1,\dots,d_{L-1}\} \ge \min\{d_0,d_L\}$. 
\end{assumption}
\noindent This assumption ensures that the width of each hidden layer is no less than that of the input or output layer. Notably, it aligns with common practices in deep learning, where hidden layers are typically designed to be wide to have sufficient capacity to capture complex data representations \cite{hanin2019universal,raghu2017expressive}. In addition, this assumption is widely used in the literature to analyze the convergence behavior of GD for optimizing deep networks; see, e.g., \cite{arora2018convergence,du2019width,du2018gradient,laurent2018deep}.  Then we obtain the closed-form expression of the critical point set \(\mathcal{W}_F\) under Assumption~\ref{AS:1} as follows:
\begin{thm}[Characterization of critical points]\label{thm:opti}
Suppose that Assumption~\ref{AS:1} holds.
It holds that $\bm W \in \mathcal{W}_F$ if and only if
\begin{subequations}\label{eq:W}
\begin{align}
&  \bm W_1 = \bm Q_2\bm \Sigma_1 \mathrm{BlkD}\left( \bm O_1,\dots,\bm O_{p_Y},\bm O_{p_Y+1} \right)\bm V_Y^T,\quad \bm W_l = \bm Q_{l+1} \bm \Sigma_l \bm Q_l^T,\ l=2,\dots,L-1,  \\
&  \bm W_L =  \bm U_Y\mathrm{BlkD}\left( \bm O_1^T,\dots,\bm O_{p_Y}^T,\widehat{\bm O}_{p_Y+1}^{T} \right) \bm \Sigma_L \bm Q_L^T,
\end{align}
\end{subequations}
where $\bm Q_l \in \mathcal{O}^{d_{l-1}}$ for all $l=2,\dots,L$, $\bm O_i \in \mathcal{O}^{h_i}$ for each $i \in [p_Y]$, $\bm O_{p_Y+1} \in \mathcal{O}^{d_0 - r_{Y}}$, $\widehat{\bm O}_{p_Y+1} \in \mathcal{O}^{d_L - r_{Y}}$, and $\bm \Sigma_l = \mathrm{BlkD}\left(\mathrm{diag}(\bm \sigma)/\sqrt{\lambda_l}, \bm 0 \right) \in \R^{d_l\times d_{l-1}}$ for each $l \in [L]$ with $\bm \sigma \in \R^{r_Y}$ satisfying  
\begin{align}\label{eq:sigma W}
\sigma_i^{2L-1} - \sqrt{\lambda}y_i \sigma_i^{L-1} +  \lambda \sigma_i = 0,\ \sigma_i \geq 0,\ \forall i \in [r_Y],\ \text{where}\ \lambda := \prod_{l=1}^L \lambda_l. 
\end{align}
\end{thm}
We remark that this theorem is a special case of \cite[Theorem 2.1]{chen2025complete}, which provides a closed-form characterization of the critical point set of Problem \eqref{eq:F} without any conditions on the network widths. In contrast, by leveraging Assumption \ref{AS:1}, we obtain a simplified and more explicit description of the critical point set in this theorem. This characterization serves as a cornerstone for establishing the error bound for Problem~\eqref{eq:F}. In general, these critical points are not isolated points but form a union of connected sets (see \Cref{prop:set W}). This brings significant difficulty in computing the distance $\mathrm{dist}(\bm W, \mathcal{W}_F)$ from a point $\bm W$ to the critical point set, as this distance does not admit a closed-form solution due to the complex form of $\mathcal{W}_F$. One of our key contributions is to develop a method that overcomes this difficulty.

\subsection{Error Bound for Deep Linear Networks}

Before we proceed, it is worth noting that there are some degenerate cases under which the error bound does not hold. For example, when $L=2$, we study 
\begin{align*}
\min_{\bm W_1, \bm W_2 \in \R^{d\times d}} F(\bm W) =  \|\bm W_2\bm W_1 - \lambda \bm I_d \|_F^2 +   \lambda \left( \|\bm W_1\|_F^2 + \|\bm W_2\|_F^2 \right).
\end{align*}
Applying \Cref{thm:opti} to the above problem yields $\mathcal{W}_F = \{(\bm 0, \bm 0)\}$. Now, consider \(\bm W_1 = \bm W_2 = x \bm I_d	\). We compute  
\begin{align*}
\nabla_1 F (\bm W) =  2\bm W_2^T \left(\bm W_2 \bm W_1 - \lambda \bm I_d\right) +  2\lambda \bm W_1 = 2 x^3 \bm I_d,\ \nabla_2 F (\bm W) =  2\left(\bm W_2 \bm W_1 - \lambda \bm I_d\right)\bm W_1^T + 2\lambda \bm W_2 = 2 x^3 \bm I_d. 
\end{align*}
Therefore, we have $\|\nabla F(\bm W)\|_F^2 = 8dx^6$. Moreover, we compute $\mathrm{dist}^2(\bm W,\mathcal{W}_F) = \|\bm W_1\|_F^2 + \|\bm W_2\|_F^2 = 2d x^2$. Consequently, we have \(d \|\nabla F(\bm W)\|_F =  \mathrm{dist}^3(\bm W,\mathcal{W}_F)\). Obviously, the error bound \eqref{eq:eb} does not hold as $x \to 0$. To avoid such degenerate cases, we impose the following conditions on the relationship between the regularization parameters and the data matrix $\bm Y$. 

\begin{assumption}\label{AS:2}
It holds for $L=2$ that
\begin{align}\label{eq:AS L=2}
\lambda := \lambda_1\lambda_2 \neq y_i^2,\ \forall i \in [r_Y].    
\end{align}
For all $L \ge 3$, it holds that 
\begin{align}\label{eq:AS y}
\lambda :=  \lambda_1\cdots\lambda_L \neq y_i^{2(L-1)}\left( \left( \frac{L-2}{L} \right)^{\frac{L}{2(L-1)}} + \left( \frac{L}{L-2} \right)^{\frac{L-2}{2(L-1)}} \right)^{-2(L-1)},\ \forall i \in [r_Y]. 
\end{align}
\end{assumption}

Notably, this assumption provides necessary and sufficient conditions under which the error bound holds for all critical points of Problem \eqref{eq:F}. Specifically, if Assumption \ref{AS:2} does not hold, the error bound fails to hold (see Appendix \ref{app sec:counter}). Conversely, if Assumption \ref{AS:2} holds, the error bound holds as follows:  

\begin{thm}[Error bound]\label{thm:eb}
Suppose that Assumptions \ref{AS:1} and \ref{AS:2} hold. There exist constants $\epsilon_1, \kappa_1 > 0$ such that for all $\bm W$ satisfying $\mathrm{dist}(\bm W, \mathcal{W}_F) \le \epsilon_1$,  we have
\begin{align}\label{eq:eb F}
\mathrm{dist}\left(\bm W, \mathcal{W}_F\right) \le \kappa_1\|\nabla F(\bm W)\|_F. 
\end{align}
\end{thm}
This theorem is significant as it shows that the error bound holds at any critical point of Problem \eqref{eq:F} under Assumptions \ref{AS:1} and \ref{AS:2}.  Importantly, the constants $\kappa_1, \epsilon_1$ can be explicitly derived through our proofs (see \Cref{prop:eb zero}, \Cref{thm:eb G}, and \eqref{eq:def_kappa_epsilon} and \eqref{eq:def_epsilon_kappa_1} in the proof of \Cref{thm:eb}). 
Now, we discuss some implications of our result. First, the established error bound can be used to derive other regularity conditions, such as P\L\ inequality and quadratic growth, for analyzing the convergence behavior of first-order methods.

\begin{coro}\label{coro:PL QG}
Suppose that Assumptions~\ref{AS:1} and~\ref{AS:2} hold. For any $\bm{W}^* \in \mathcal{W}_F$, there exists $\epsilon_2 > 0$ such that for all $\bm{W}$ in the neighborhood $\{\bm{W} : \|\bm{W} - \bm{W}^*\|_F \le \epsilon_2\}$, the following statements hold:\\
(i) {\bf (P\L\ inequality)} There exists a constant $\mu_1 > 0$ such that $\|\nabla F(\bm W)\|_F^2 \ge \mu_1 | F(\bm W) - F(\bm W^*)|.$ \\
(ii) {\bf (Quadratic growth)} If, in addition, $\bm W^*$ is a local minimizer, there exists a constant $\mu_2 > 0$ such that $\mathrm{dist}^2(\bm{W}, \mathcal{W}_F) \leq \mu_2 \left( F(\bm{W}) - F(\bm{W}^*) \right).$

\end{coro}

We remark that these regularity conditions (the error bound, the P\L\  inequality, and the quadratic growth condition) are shown to be equivalent in \cite{rebjock2024fast} under the additional assumption that the point is a local minimizer. For completeness, we provide a proof of \Cref{coro:PL QG}(i) in Appendix \ref{app:PL QG} when $\bm W^*$ is a critical point.
Moreover, in
Appendix~\ref{app:abs_PL_MSR}, we show the equivalence between the P\L\ inequality and the error bound around an arbitrary critical point for any finite-dimensional $C^2$ function.

Second, by combining the error bound with other algorithm-dependent conditions, such as the {\em sufficient decrease condition} and the {\em relative error condition}, local linear convergence of first-order methods can be obtained for solving Problem \eqref{eq:F} (see Appendix \ref{app:sec C}). In contrast to algorithm-specific convergence rate analyses in \citep{bah2022learning,chitour2023geometric,Hu2020Provable,nguegnang2024convergence,wu2019global}, the error-bound-based framework provides a unified analytical approach. In particular, it applies not only to GD but also to a broad class of first-order methods capable of optimizing Problem \eqref{eq:F}. Moreover, our experimental results in \Cref{sec:expe} demonstrate that GD achieves linear convergence to critical points, which further supports our theoretical findings. In addition, we discuss how the error bound can be used to analyze the convergence rates of second-order methods in \Cref{rmk:second} of Appendix \ref{app:sec C}.

Third, we emphasize our technical contributions to establishing the error bound. Unlike prior works \cite{Liu2019,jiang2019novel,jiang2022holderian,wang2023linear}, where the critical point set $\mathcal{W}_F$ has a relatively simple structure that allows for closed-form computation of the distance $\mathrm{dist}(\bm W, \mathcal{W}_F)$, the presence of orthogonal permutations and hierarchical structures in $\mathcal{W}_F$ (see \Cref{thm:opti}) prevents such computation. To address this challenge, we develop a suite of new techniques (see \Cref{subsubsec:pf eb}), which offer a new framework for establishing error bounds of non-convex problems with complex solution sets.
\section{Proofs of the Main Results}\label{sec:proof}

In this section, we present the proofs of our main results on the critical point set (i.e., \Cref{thm:opti}) in \Cref{subsec:pre} and the error bound (i.e., \Cref{thm:eb}) for Problem \eqref{eq:F} in \Cref{subsec:ana eb}. To facilitate readability, we summarize the notations used in this section in \Cref{tab:notations}.

\subsection{Proofs of the Critical Point Set}\label{subsec:pre}

To characterize the critical points and establish the error bound for Problem \eqref{eq:F}, we claim that it suffices to study the following problem: 
\begin{align}\label{eq:G}
    \min_{\bm W}\ G(\bm W) := \left\|\bm W_L\cdots\bm W_1 -  \sqrt{\lambda}\bm Y \right\|_F^2 + \lambda \sum_{l=1}^L \|\bm W_l\|_F^2,
\end{align}
where $\lambda := \prod_{l=1}^L \lambda_l$, as defined in \eqref{eq:AS L=2} and \eqref{eq:AS y}. We compute the gradient of $G(\bm W)$ as follows:
\begin{align}
 \nabla_{l} G(\bm W) &= 2\bm W_{L:l+1}^T\left(\bm W_{L:1} - \sqrt{\lambda}\bm Y\right)\bm W_{l-1:1}^T + 2\lambda \bm W_l\notag\\
 &={2 \left(\bm W_{L:l+1}^T \bm W_{L:l+1}
\right)\bmw_l \left(\bm W_{l-1:1}\bm W_{l-1:1}^T\right) -  2\sqrt{\lambda}\bm W_{L:l+1}^T\bm Y\bm W_{l-1:1}^T\label{eq:grad G}+2\lambda\bmw_l,\ \forall l \in [L]. }
 \end{align} 
The critical point set of Problem \eqref{eq:G} is defined as
\begin{align}
\mathcal{W}_G := \left\{ \bm W = (\bm W_1,\dots,\bm W_L): \nabla  G(\bm W) = \bm 0 \right\}. 
\notag
\end{align} 
\begin{table}[t]
\centering
{\caption{\centering The explicit expressions of the notations in the paper}\label{tab:notations}}
\resizebox{\textwidth}{!}{%
\begin{tabular}{c|c|c|c|c|c|c|c|c|c|c|c}
\hline
Notation & Eq. & Notation & Eq. & Notation & Eq. & Notation & Eq. & Notation & Eq. & Notation & Eq. \\
\hline
$p$&\eqref{eq:partsigma}&$p_{Y},r_Y$&\eqref{eq:SY1}&$s_i,h_i$&\eqref{eq:SY1}&$t_i,g_i$&\eqref{eq:partsigma}&$\varphi(x)$&\eqref{eq:phi}&$G(\bmw)$&\eqref{eq:G}\\
\hline
$\lambda$ & (\ref{eq:AS L=2},\ref{eq:AS y}) &
$\delta_y$ & \eqref{eq:delta y} &
$\delta_\sigma, \mathcal{Y}$ & \eqref{set:Y} &
$\mathcal{A}$ & \eqref{set:A} &
$\mathcal{B}$ & \eqref{set:B} &
$\mathcal{A}_{\text{sort}}$ & \eqref{set:W sigma} \\
\hline
$\eta_1$ & \eqref{eq:eta1} &
$\eta_2$ & \eqref{eq:eta_2def} &
$\eta_3,\eta_4$ & \eqref{eq:eta3}, \eqref{eq:eta4} &
$\eta_5$ & \eqref{eq:upper_v} &
$L_G$ & \eqref{eq:defLG}&
$c_1$ & \eqref{eq:c1} \\
\hline
$c_2$ & \eqref{eq:c2} &
$c_3$ & \eqref{eq:c3} &
$c_4$ & \eqref{eq:upper_v} &
$c_5$ & \eqref{eq:c5} &
$\delta_1,\delta_2$ & \eqref{eq:delta1},\eqref{eq:delta2} &
$\sigma^*_{\min},\sigma_{\max}^*,r_{\sigma}$ & \eqref{def:r sigma} \\
\hline
$\epsilon_{\bm 0},\kappa_{\bm 0}$&\eqref{def: epsilon_kappa1}&$\epsilon_{\bm \sigma^*},\kappa_{\bm \sigma^*}$&\eqref{def:epsilon_kappa_2}&$\mathcal{W}_{\bm \sigma^*}$& (\ref{eq:partition},\ref{set:W sigma})&$\kappa,\epsilon$&\eqref{eq:def_kappa_epsilon}&$\kappa_1,\epsilon_1$&\eqref{eq:def_epsilon_kappa_1}&$d_{\max}$&\eqref{eq:delta y}\\
\hline
\end{tabular}}
\vspace{-0.15in}
\end{table}
Now, we present the following lemma to prove our claim.
\begin{lemma}\label{lem:equi FG}
Consider Problems \eqref{eq:F} and \eqref{eq:G}. The following statements hold: \\
(i) $(\bm W_1,\dots,\bm W_L) \in \mathcal{W}_F$ if and only if $(\sqrt{\lambda_1}\bm W_1,\dots,\sqrt{\lambda_L}\bm W_L) \in \mathcal{W}_G$. \\
(ii)
Let $\bm{Y} = \bm{U}_Y \bm{\Sigma}_Y \bm{V}_Y^T$ be an SVD of $\bm{Y}$. Then $(\bm{W}_1, \ldots, \bm{W}_L)$ is a critical point of Problem~\eqref{eq:G} if and only if $(\bm{W}_1 \bm{V}_Y,\, \bm{W}_2,\, \ldots,\, \bm{W}_{L-1},\, \bm{U}_Y^T \bm{W}_L)$ is a critical point of
$
\min_{\bm{W}}\, \big\|\, \bm{W}_L \cdots \bm{W}_1 - \sqrt{\lambda} \bm{\Sigma}_Y \big\|_F^2 + \lambda \sum_{l=1}^L \|\bm{W}_l\|_F^2.
$\\
(iii) Suppose that there exist constants \(\epsilon, \kappa > 0\) such that for all \(\bm Z\) satisfying \(\mathrm{dist}(\bm Z, \mathcal{W}_G) \leq \epsilon\), we have 
\begin{align}\label{eq:ebG}
    \mathrm{dist}(\bm Z,  \mathcal{W}_G) \le \kappa \|\nabla G(\bm Z)\|_F.
\end{align} 
Then for all \(\bm W\) satisfying \(\mathrm{dist}(\bm W, \mathcal{W}_F) \le \epsilon/\sqrt{\lambda_{\max}}\), it holds that 
$
    \mathrm{dist}\left(\bm W, \mathcal{W}_F\right) \le ({\kappa \lambda}/{\lambda_{\min}}) \|\nabla F(\bm W)\|_F.
$
\end{lemma}
\begin{proof} 
(i) and (ii) have been proved in \cite[Lemma 3.1]{chen2025complete}. Now, we prove (iii). Using (i), we have $\mathcal{W}_G=\left\{(\sqrt{\lambda_1} \bm W_1,\dots,\sqrt{\lambda_L} \bm W_L): (\bm W_1, \dots, \bm W_L) \in \mathcal{W}_F\right\}$. Let $\bm W = \left( \bm W_1,\dots,\bm W_L\right)$ be arbitrary and $\hat{\bm W} = (\sqrt{\lambda_1}\bm W_1,\dots,\sqrt{\lambda_L}\bm W_L)$. Let $(\sqrt{\lambda_1}\hat{\bm W}_1^*,\dots,\sqrt{\lambda_L}\hat{\bm W}_L^*)$ with $(\hat{\bm W}_1^*,\dots,\hat{\bm W}_L^*) \in \mathcal{W}_F$ be such that $\mathrm{dist}^2(\hat{\bm W}, \mathcal{W}_G) = \sum_{l=1}^L \lambda_l\|\bm W_l - \hat{\bm W}^*_l\|_F^2$, and let $(\bm W_1^*,\dots,\bm W_L^*) \in \mathcal{W}_F$ be such that $\mathrm{dist}^2(\bm W,\mathcal{W}_F) = \sum_{l=1}^L \|\bm W_l - \bm W_l^*\|_F^2$. We have 
\begin{align}\label{eq:distconnect}
    \mathrm{dist}^2(\bm W,\mathcal{W}_F) = \sum_{l=1}^L \|\bm W_l - \bm W_l^*\|_F^2 \ge \frac{\sum_{l=1}^L \|\sqrt{\lambda_l}\bm W_l - \sqrt{\lambda_l}\bm W_l^*\|_F^2}{\lambda_{\max}}  \ge \frac{1}{\lambda_{\max}}\mathrm{dist}^2(\hat{\bm W}, \mathcal{W}_G). 
\end{align}
Using this and $\mathrm{dist}(\bm W,\mathcal{W}_F) \le  \epsilon/\sqrt{\lambda_{\max}}$ yields $\mathrm{dist}(\hat{\bm W}, \mathcal{W}_G) \le \epsilon$, and thus \eqref{eq:ebG} holds.  
Next, we have
\begin{align}
    & \mathrm{dist}^2(\hat{\bm W}, \mathcal{W}_G) = \sum_{l=1}^L \lambda_l\|\bm W_l - \hat{\bm W}^*_l\|_F^2 \ge \lambda_{\min} \sum_{l=1}^L \|\bm W_l - \hat{\bm W}^*_l\|_F^2 \ge \lambda_{\min} \mathrm{dist}^2(\bm W, \mathcal{W}_F),\label{eq2:lem equi}\\  
    & \nabla_{l} G(\hat{\bm W}) \overset{\eqref{eq:grad G}}{=} 2\sqrt{\lambda_L\cdots\lambda_{l+1}\lambda_{l-1}\cdots\lambda_1}\bm W_{L:l+1}^T\left(\sqrt{\lambda}\bm W_{L:1} - \sqrt{\lambda}\bm Y\right)\bm W_{l-1:1}^T  + 2 \lambda \sqrt{\lambda_l} \bm W_l \notag \\
    &\qquad\qquad\ =  \frac{\lambda}{\sqrt{\lambda_l}}  \nabla_{l} F(\bm W),\ \forall l.  \label{eq:gradientnorm}
\end{align} 
Therefore, we have 
$
    \mathrm{dist}(\bm W, \mathcal{W}_F) \overset{\eqref{eq2:lem equi}}{\le} \frac{1}{\sqrt{\lambda_{\min}}} \mathrm{dist}(\hat{\bm W}, \mathcal{W}_G) \overset{\eqref{eq:ebG}}{\le} \frac{\kappa}{\sqrt{\lambda_{\min}}} \|\nabla G(\hat{\bm W})\|_F  \overset{\eqref{eq:gradientnorm}}{\le} \frac{\kappa\lambda}{\lambda_{\min}} \|\nabla F(\bm W)\|_F. 
$
\end{proof}

Using (i) and (iii) of the above lemma, it suffices to establish an error bound for Problem \eqref{eq:G} in the rest of this section.
By \Cref{lem:equi FG}(ii), we assume without loss of generality that $\bm Y = \bm \Sigma_Y$ from now on, where $\bm \Sigma_Y$ is defined in \eqref{eq:SY}  
and its diagonal entries are partitioned as in \eqref{eq:SY1}. 
Moreover, we define
\begin{align}\label{eq:delta y}
\delta_y := \min \left\{|y_{s_i}-y_{s_{i}+1}|: i\in [p_Y] \right\},\quad d_{\max} := \max\{d_0,d_1,\ldots,d_L\}. 
\end{align} 
Using \cite[Proposition 3.3]{chen2025complete} and Assumption~\ref{AS:1}, it is straightforward to verify that the critical point set of Problem \eqref{eq:G} admits the following closed-form expression:

\begin{fact}\label{prop:opti G}
Suppose that Assumption~\ref{AS:1} holds. The critical point set of Problem~\eqref{eq:G} is 
\begin{equation}
\label{eq:partition}
    \mathcal{W}_G = \bigcup_{(\bm{\sigma}, \bm{\Pi}) \in \mathcal{B}} \mathcal{W}_{(\bm{\sigma}, \bm{\Pi})},
\end{equation}  
where  
\begin{align}\label{set:sigmapi}
& \mathcal{W}_{(\bm\sigma,\bm\Pi)} := \left\{
\bm W :
\begin{array}{l}
\bm \Sigma_l = \mathrm{BlkD}\left(\mathrm{diag}( \bm \sigma), \bm 0\right) \in \R^{d_l\times d_{l-1}},\ \forall l \in [L],\\ 
\bm W_1 = \bm Q_2\bm \Sigma_1 \mathrm{BlkD}\left(\bm \Pi, \bm I  \right)\mathrm{BlkD}\left( \bm O_1,\dots,\bm O_{p_{Y}},\bm O_{p_{Y}+1} \right), \\ 
\bm W_l = \bm Q_{l+1} \bm \Sigma_l \bm Q_l^T,\ l=2,\dots,L-1,  \\
\bm W_L =  \mathrm{BlkD}\left( \bm O_1^T,\dots,\bm O_{p_{Y}}^T,\widehat{\bm O}_{p_{Y}+1}^T \right)\mathrm{BlkD}\left(\bm \Pi^T, \bm I \right)\bm \Sigma_L \bm Q_L^T, \\
\bm O_i \in \mathcal{O}^{h_i},\ \forall i \in [p_{Y}],\  \bm O_{p_{Y}+1} \in \mathcal{O}^{d_0 - r_{Y}},\ \widehat{\bm O}_{p_{Y}+1}  \in \mathcal{O}^{d_L - r_{Y}},\\
\bm Q_l \in \mathcal{O}^{d_{l-1}},\ l=2,\dots,L,
\end{array} 
\right\},\\
&\mathcal{A} := \left\{ \bm a \in \R^{d_{\min}}: \ a_i^{2L-1} - \sqrt{\lambda}y_ia_i^{L-1} + \lambda a_i = 0,\ a_i \geq 0,\ \forall i \in [d_{\min}] \right\}, \label{set:A} \\
& \mathcal{B} := \left\{ (\bm \sigma, \bm \Pi)  \in \R^{d_{\min}} \times \mathcal{P}^{d_{\min}}:  \exists 
\bm a \in \mathcal{A},\  \bm \sigma =\bm \Pi \bm a,\ \sigma_1 \geq \sigma_2 \geq \dots \geq \sigma_{d_{\min}}  \right\}.\label{set:B} 
\end{align}
\end{fact}

Note that the permutation matrix $\bm \Pi$ can be absorbed into $\bm \Sigma_l$ and $\bm Q_l$ to simplify the above expression. With this observation, and by invoking \Cref{lem:equi FG}, we are now ready to prove \Cref{thm:opti} as follows: \smallskip 

\begin{proof}[Proof of \Cref{thm:opti}] 
By (ii) of \Cref{lem:equi FG}, without loss of generality, we may assume that $\bm Y = \bm \Sigma_Y$, i.e., $\bm U_Y = \bm I$ and $\bm V_Y = \bm I$.
Let $(\bm W_1,\dots,\bm W_L) \in \mathcal{W}_F$ be arbitrary. It follows from \Cref{prop:opti G} and \Cref{lem:equi FG} that there exists $(\bm \sigma,\bm \Pi) \in \cal B$ such that $(\bm W_1,\dots,\bm W_L)$ takes the form of \eqref{set:sigmapi} with $\bm \Sigma_l = \mathrm{BlkD}(\mathrm{diag}(\bm \sigma)/\sqrt{\lambda_l},\bm 0)$. According to \eqref{set:A} and \eqref{set:B}, there exists $\bm a \in \cal A$ such that $\bm \Pi \mathrm{diag}(\bm a) \bm \Pi^T = \mathrm{diag}(\bm \sigma)$. Then, let
\begin{align*}
    &\bm \Sigma_l^\prime := \mathrm{BlkD}(\bm \Pi^T , \bm I_{d_l - d_{\min}})\bm \Sigma_l \mathrm{BlkD}(\bm \Pi , \bm I_{d_{l-1} - d_{\min}}) = \mathrm{BlkD}\left(\mathrm{diag}(\bm a)/\sqrt{\lambda_l}, \bm 0 \right),\ \forall l \in [L],\\
    &\bm Q_l^\prime := \bm Q_l \mathrm{BlkD}(\bm \Pi , \bm I_{d_{l-1} - d_{\min}}) \in \mathcal{O}^{d_{l-1}},\ l=2,\dots,L.
\end{align*}
Therefore, we obtain that $(\bm W_1,\dots,\bm W_L)$ can be rewritten in the form of  \eqref{eq:W} with
$\bm \Sigma_l=\bm \Sigma_l^\prime$ for \(l\in[L]\), and $\bm Q_l =\bm Q_l^\prime $ for \(l=2,\dots,L\).

Conversely, let $(\bm W_1,\dots,\bm W_L)$ be of the form of \eqref{eq:W} for some $\bm \sigma=(\sigma_1,\ldots,\sigma_{r_Y},0,\ldots,0)\in \mathbb{R}^{d
_{\min}}$. Then, we choose a permutation matrix $\bm \Pi \in \mathcal{P}^{d_{\min}}$ such that
$\bm \sigma^\prime = \bm \Pi \bm \sigma\ \text{satisfies}\  \sigma^\prime_1 \geq \sigma^\prime_2 \ge \dots \ge \sigma^\prime_{d_{\min}}.$ Let 
\begin{align*}
    & \bm \Sigma^\prime_l := \mathrm{BlkD}(\bm \Pi , \bm I_{d_l - d_{\min}})\bm \Sigma_l \mathrm{BlkD}(\bm \Pi^T , \bm I_{d_{l-1} - d_{\min}}) = \mathrm{BlkD}\left(\mathrm{diag}(\bm \sigma^\prime)/\sqrt{\lambda_l}, \bm 0 \right),\ \forall l \in [L],\\
    &\bm Q^\prime_l := \bm Q_l \mathrm{BlkD}(\bm \Pi^T, \bm I_{d_{l-1} - d_{\min}}) ,\ l=2,\dots,L.
\end{align*}
Then, we obtain that $(\bm W_1,\dots,\bm W_L)$ can be rewritten in the form of \eqref{set:sigmapi} with
$\bm \Sigma_l=\bm \Sigma_l^\prime$ for \(l\in[L]\), and $\bm Q_l =\bm Q_l^\prime $ for \(l=2,\dots,L\).
This, together with \Cref{prop:opti G} and item (i) and (ii) of \Cref{lem:equi FG}, implies $(\bm W_1,\dots,\bm W_L) \in \mathcal{W}_F$. 
\end{proof}
\vspace{-0.1in}
\subsection{Analysis of the Error Bound}\label{subsec:ana eb} 

In this subsection, we establish the error bound for Problem \eqref{eq:G} based on 
\Cref{prop:opti G}. In \Cref{subsubsec:pre}, we first present preliminary results for our analysis. In \Cref{subsubsec:pf eb}, we provide the main proofs of the error bound. For ease of exposition, we define  
\begin{align}\label{set:Y}
&\mathcal{Y} := \bigcup_{ i \in [d_{\min}]}\left\{ \sigma \ge 0:  \sigma^{2L-1} + \lambda\sigma - \sqrt{\lambda} y_i\sigma^{L-1} = 0  \right\},\  \delta_{\sigma} := \min\left\{|x-y|: x \neq y \in \mathcal{Y}  \right\}.
\end{align} 
The following result elucidates the structure of the collection $\{\mathcal{W}_{(\bm \sigma, \bm \Pi)}\}_{(\bm \sigma, \bm \Pi) \in \mathcal{B}}$. 
\begin{prop}\label{prop:set W}
Let $(\bm \sigma, \bm \Pi) \in \mathcal{B}$ and $(\bm \sigma^\prime, \bm \Pi^\prime) \in \mathcal{B}$ be arbitrary. The following statements hold: \\
(i) It holds that $\mathcal{W}_{(\bm \sigma, \bm \Pi)} = \mathcal{W}_{(\bm \sigma^\prime, \bm \Pi^\prime)}$ if and only if $\bm \sigma = \bm \sigma^\prime$. \\
(ii) If $\bm \sigma \neq \bm \sigma^\prime$, it holds that 
$
\mathrm{dist}\left( \mathcal{W}_{(\bm \sigma, \bm \Pi)}, \mathcal{W}_{(\bm \sigma^\prime, \bm \Pi^\prime)} \right) \ge \delta_{\sigma}. 
$
\end{prop} 
We defer the proof of this proposition to Appendix \ref{subsec app:set W}. This proposition demonstrates that for any pair $(\bm \sigma, \bm \Pi) \in \mathcal{B}$ and $(\bm \sigma^\prime, \bm \Pi^\prime) \in \mathcal{B}$, if $\bm \sigma \neq \bm \sigma^\prime$, $\mathcal{W}_{(\bm\sigma, \bm\Pi)}$ is well separated from $\mathcal{W}_{(\bm\sigma^\prime, \bm\Pi^\prime)}$, and otherwise they are identical.
Therefore, with a slight abuse of notation, we write $\mathcal{W}_{\bm\sigma}:= \mathcal{W}_{(\bm \sigma, \bm \Pi)}$ for any $(\bm \sigma, \bm \Pi)\in \mathcal{B}$. Thus, we can express the set of critical points of Problem \eqref{eq:G} as follows:  
\begin{align}\label{set:W sigma}
\mathcal{W}_G := \bigcup_{\bm \sigma \in  \mathcal{A}_{\rm sort}} \mathcal{W}_{\bm \sigma },\ \text{where}\ \mathcal{A}_{\rm sort} := \left\{ \bm \sigma \in \R^{d_{\min}}: (\bm \sigma, \bm \Pi) \in  \mathcal{B} \right\}.  
\end{align}  
\subsubsection{Preliminary Results}\label{subsubsec:pre}

According to \Cref{prop:set W} and \eqref{set:W sigma}, for any $\bm W$ satisfying $\mathrm{dist}(\bm W, \mathcal{W}_G) < \delta_{\sigma}/2$, there exists only one $\bm \sigma^* \in \mathcal{A}_{\rm sort}$ such that 
$
\mathrm{dist}(\bm W, \mathcal{W}_G) = \mathrm{dist}(\bm W, \mathcal{W}_{\bm \sigma^*}). 
$
This observation simplifies our analysis, as it suffices to bound $\mathrm{dist}(\bm W, \mathcal{W}_{\bm \sigma^*})$ for each $\bm \sigma^* \in \mathcal{A}_{\rm sort}$. Note that $\bm 0 \in \mathcal{A}_{\rm sort}$. For this case, we directly show the error bound as follows:

\begin{prop}\label{prop:eb zero}
Suppose that Assumptions~\ref{AS:1} and \ref{AS:2} hold.  Then for all $\bm W$ satisfying $\mathrm{dist}(\bm W, \mathcal{W}_{\bm 0}) \le \epsilon_{\bm 0}$, we have
\begin{align}
\mathrm{dist}(\bm W,\mathcal{W}_{\bm 0}) 
\le \kappa_{\bm 0} \, \bigl\| \nabla G(\bm W) \bigr\|_F.
\label{eq:eb 0}
\end{align}
The constants $\epsilon_{\bm 0}$ and $\kappa_{\bm 0}$ are specified as follows. If $r_Y>0$, then
\begin{align}\label{def: epsilon_kappa1} 
(\epsilon_{\bm 0},\, \kappa_{\bm 0}) :=
\begin{cases}
\Biggl( 
\Bigl( \dfrac{\sqrt{\lambda}}{2(\sqrt{\lambda}+y_1)} 
        \min\!\Bigl\{ \min\limits_{i \in [s_{p_Y}]} \bigl|\lambda - y_i^2\bigr|,\, \lambda \Bigr\} \Bigr)^{\frac{1}{2}},\;
\dfrac{2(\sqrt{\lambda}+y_1)}
     {\sqrt{\lambda} \min\!\Bigl\{ \min\limits_{i \in [s_{p_Y}]} \bigl|\lambda - y_i^2\bigr|,\, \lambda \Bigr\}}
\Biggr), & \text{if}\ L=2, \\[2ex]
\Biggl(
\min\!\left\{ \left( \dfrac{\lambda}{3} \right)^{{1}/{(2L-2)}},\;
               \left( \dfrac{\sqrt{\lambda}}{3y_1} \right)^{{1}/{(L-2)}} \right\},\;
\dfrac{3\sqrt{L}}{2\lambda}
\Biggr), & \text{if}\ L \ge 3.
\end{cases}
\end{align}
If \(r_Y=0\), we set $ (\epsilon_{\bm0},\kappa_{\bm0})=(1,(2\lambda)^{-1})$.
\end{prop}

\begin{proof} 
If \(\bm Y=\bm0\), then \(r_Y=0\). By \Cref{prop:opti G} and \eqref{set:A}, we have \(\mathcal A_{\rm sort}=\{\bm0\}\), and hence \(\mathcal W_G=\mathcal W_{\bm0}=\{\bm0\}\). Moreover,
\[
\langle\nabla G(\bm W),\bm W\rangle
=\sum_{l=1}^L\langle\nabla_lG(\bm W),\bm W_l\rangle
=2L\|\bm W_{L:1}\|_F^2+2\lambda\|\bm W\|_F^2
\ge 2\lambda\|\bm W\|_F^2.
\]
Thus, we have 
$
\operatorname{dist}(\bm W,\mathcal W_{\bm0})
=\|\bm W\|_F
\le (2\lambda)^{-1}\|\nabla G(\bm W)\|_F .
$
Thus, below we assume \(r_Y>0\).
Let's begin by focusing on \( L = 2 \). Using \Cref{prop:opti G},
we have $\mathcal{W}_{\bm 0} = \{(\bm 0, \bm 0)\}$ and $\mathrm{dist}^2(\bm W,\mathcal{W}_{\bm 0}) = \|\bm W_1\|_F^2 + \|\bm W_2\|_F^2.$
Without loss of generality, assume that $\| \bm{W}_1 \|_F \ge \| \bm{W}_2 \|_F$. We have
\begin{align}
\nabla_{1} G(\bm{W}) =  2\bm{W}_2^T\left(\bm{W}_2 \bm{W}_1 - \sqrt{\lambda}\bm{Y}\right)  + 2\lambda  \bm{W}_1,\quad \nabla_{2} G(\bm{W}) = 2\left(\bm W_2\bm W_1 - \sqrt{\lambda}\bm Y\right)\bm W_1^T + 2\lambda \bm W_2\label{eq3:lem eb zero}.
\end{align}
Using the first equation in \eqref{eq3:lem eb zero}, we have 
\begin{align}\label{eq4:lem eb zero}
\frac{\sqrt{\lambda}}{2}\| \nabla_{1} G(\bm{W}) \|_F & \ge \left\|\lambda \bm W_2^T\bm Y - \lambda^{\frac{3}{2}}\bm W_1\right\|_F - \sqrt{\lambda}\|\bm W_2^T\bm W_2\bm W_1\|_F \notag\\
& \ge \left\|\lambda \bm W_2^T\bm Y - \lambda^{\frac{3}{2}}\bm W_1\right\|_F - \sqrt{\lambda}\|\bm W_2\|_F^2 \|\bm W_1\|_F.  
\end{align}
Multiplying $\bm Y^T$ on the both sides of the second equation in \eqref{eq3:lem eb zero}, together with the triangle inequality, yields \[\frac{1}{2}\left\|\bm Y^T\nabla_{2} G(\bm{W})\right\|_F  \ge \| \sqrt{\lambda}\bm Y^T\bm Y\bm W_1^T - \lambda\bm Y^T\bm W_2 \|_F - \|\bm Y^T\bm W_2\bm W_1\bm W_1^T\|_F,\]
which implies \[\frac{y_1}{2}  \|\nabla_{2} G(\bm{W}) \|_F \ge \| \sqrt{\lambda}\bm Y^T\bm Y\bm W_1^T - \lambda\bm Y^T\bm W_2 \|_F - y_1\|\bm W_2\|_F \|\bm W_1\|_F^2.\]
Summing up this inequality with \eqref{eq4:lem eb zero} and using the triangle inequality, we obtain 
\begin{align*}
\frac{\sqrt{\lambda}+y_1}{2} \|\nabla G(\bm W)\|_F
&\ge \frac{\sqrt{\lambda}}{2}\| \nabla_{1} G(\bm{W}) \|_F + \frac{y_1}{2} \left\|\nabla_{2} G(\bm{W})\right\|_F\\ & \ge  \sqrt{\lambda}\left\|\bm W_1\left(\lambda\bm I - \bm Y^T\bm Y \right)  \right\|_F - \|\bm W_1\|_F\|\bm W_2\|_F\left(y_1\|\bm W_1\|_F + \sqrt{\lambda}\|\bm W_2\|_F \right) \\
& \ge  \sqrt{\lambda} \min\left\{ \min_{i \in [s_{p_Y}]}\left| \lambda - y_i^2 \right|, \lambda \right\} \|\bm W_1\|_F - \left(\sqrt{\lambda} + y_1 \right)\|\bm W_1\|_F^3 \\
& \ge \frac{\sqrt{\lambda} }{2} \min\left\{ \min_{i \in [s_{p_Y}]}\left| \lambda - y_i^2 \right|, \lambda \right\}  \|\bm W_1\|_F, \\
& \ge \frac{\sqrt{\lambda} }{4} \min\left\{ \min_{i \in [s_{p_Y}]}\left| \lambda - y_i^2 \right|, \lambda \right\}  \mathrm{dist}(\bm W, \mathcal{W}_{\bm 0} )
\end{align*}
where the third inequality follows from \eqref{eq:AS L=2} and $\bm Y = \bm \Sigma_Y$ defined in \eqref{eq:SY}, $\|\bm A \bm B\|_F \ge \sigma_{\min}(\bm A)\|\bm B\|_F$, and 
$\|\bm W_1\|_F \ge \|\bm W_2\|_F$, and the fourth inequality uses $\|\bm W\|_F \le \epsilon_{\bm 0}$ and \eqref{def: epsilon_kappa1}. 
This directly implies \eqref{eq:eb 0}.

Next, we focus on $L\ge 3$. According to \Cref{prop:opti G} and \eqref{set:W sigma}, we have $\mathcal{W}_{\bm 0} = \{(\bm 0, \bm 0, \dots, \bm 0)\}$. Then, we have
$ \mathrm{dist}^2(\bm W,\mathcal{W}_{\bm 0}) = \sum_{l=1}^L \|\bm W_l\|_F^2\le \epsilon_{\bm 0}^2. $
Let $k \in [L]$ be such that $k \in \argmax_{l\in [L]} \|\bm W_l\|_F$. Then, we compute
\begin{align*}
\frac{1}{2}\left\|\nabla_{k} G(\bm W)\right\|_F & \overset{\eqref{eq:grad G}}{=} \left\| \bm W_{L:k+1}^T \left(\bm W_{L:1} - \sqrt{\lambda}\bm Y \right)\bm W_{k-1:1}^T + \lambda \bm W_k \right\|_F \ge \lambda \|\bm W_k\|_F - \left( \|\bm W_k\|_F^{2L-1} + \sqrt{\lambda}y_1  \|\bm W_k\|_F^{L-1} \right) \\
& = \left( \lambda - \|\bm W_k\|_F^{2L-2} - \sqrt{\lambda}y_1\|\bm W_k\|_F^{L-2} \right) \|\bm W_k\|_F  \overset{\eqref{def: epsilon_kappa1}}{\ge}  \frac{\lambda}{3}\|\bm W_k\|_F.  
\end{align*}
Thus, we have $ \mathrm{dist}^2(\bm W,\mathcal{W}_{\bm 0}) \le L\|\bm W_k\|_F^2 \le \frac{9L}{4\lambda^2}\left\|\nabla_{k} G(\bm W)\right\|_F^2,$
which directly implies the desired result. 
\end{proof}

According to this proposition, it suffices to consider $\bm \sigma^* \neq \bm 0$ from now on. Let  $\bm \sigma^* \in \mathcal{A}_{\rm sort}\setminus\{\bm 0 \}$ be arbitrary and define 
\begin{align}\label{def:r sigma}
\sigma_{\min}^* := \min\left\{\sigma_i^* \neq 0: i \in [d_{\min}] \right\},\ \sigma_{\max}^* := \max\left\{\sigma_i^*: i \in [d_{\min}] \right\},\ \text{and}\ \ r_{\sigma} := \|\bm \sigma^*\|_0. 
\end{align} 
By the definition of \(\delta_\sigma\) in \eqref{set:Y} and $\bm 0 \in \mathcal Y$, we have \(\delta_\sigma \leq \sigma^*_{\min}\) for each $\bm \sigma^* \in \mathcal{A}_{\rm sort}$. Then, we show the following results.
\begin{lemma}\label{lem:pre}
Suppose that Assumption~\ref{AS:1} holds. Let $\bm W = (\bm W_1,\dots,\bm W_L)$ and $\bm \sigma^* \in \mathcal{A}_{\rm sort}\setminus\{\bm 0 \}$ 
be arbitrary such that  
$\mathrm{dist}(\bm W, \mathcal{W}_{\bm \sigma^*}) < {\sigma^*_{\min}}/{2}$.
Then the following statements hold:
\begin{flalign}
& \text{(i)}\quad
\frac{\sigma_{\max}^*}{2}  \le \|\bm W_l\| \le \frac{3\sigma_{\max}^*}{2},\
\sigma_{i}(\bm W_l) \ge \frac{\sigma_{\min}^*}{2},\ 
\forall i \in [r_{\sigma}],\ l \in [L]. && \label{eq:sigmalowerbound}\\
& \text{(ii)}\quad
\|\bm W_{l+1}^T \bm W_{l+1} - \bm W_l\bm W_l^T\|_F 
\le \frac{3\sqrt{2}\sigma_{\max}^*}{4\lambda} \|\nabla G(\bm W)\|_F,\ 
\forall l \in [L-1]. && \label{eq:balance}\\
& \text{(iii)}\quad
|\sigma_i (\bm W_l) - \sigma_i(\bm W_{l+1})|
\le \frac{3\sqrt{2}\sigma_{\max}^*}{4\lambda\sigma_{\min}^*}\|\nabla G(\bm W)\|_F,\ 
\forall i \in [r_{\sigma}],\ l \in [L-1]. && \label{eq:wl+1}\\
& \text{(iv)}\quad 
\left\|\bm W_{j:i}^T \bm W_{j:i}  -  (\bm W^T_i\bm W_i)^{j-i+1}\right\|_F 
\le \frac{(j-i)(j-i+1)}{2\sqrt{2}\lambda}
   \left(\frac{3\sigma_{\max}^*}{2}\right)^{2j-2i+1}\|\nabla G(\bm W)\|_F,\ \forall i \le j, \label{eq:transwji}\\
&\qquad\ \ \left\|\bm W_{i:j} \bm W_{i:j}^T - (\bm W_i\bm W_i^T)^{i-j+1}\right\|_F
\le \frac{(i-j)(i-j+1)}{2\sqrt{2}\lambda}
   \left(\frac{3\sigma_{\max}^*}{2}\right)^{2i-2j+1}\|\nabla G(\bm W)\|_F,\ \forall i > j. && \label{eq:transwij}
\end{flalign}
\end{lemma}
 
\begin{proof} 
    (i) Let $\bm W^* = (\bm W_1^*,\dots,\bm W_L^*) \in \mathcal{W}_{\bm \sigma^*}$ be such that $\mathrm{dist}(\bm W, \mathcal{W}_{\bm \sigma^*}) = \|\bm W - \bm W^*\|_F$. Using the triangle inequality, we have for each $l \in [L]$, 
    \begin{align*}
    \|\bm W_l\| \ge \|\bm W_l^*\| - \|\bm W_l - \bm W_l^*\| \ge \|\bm W_l^*\| - \|\bm W_l - \bm W_l^*\|_F \ge \sigma^*_{\max} - \frac{\sigma^*_{\min}}{2} \ge  \frac{\sigma^*_{\max}}{2},
    \end{align*}
    where the third inequality follows from \Cref{prop:opti G} and $\mathrm{dist}(\bm W, \mathcal{W}_{\bm \sigma^*}) <  {\sigma^*_{\min}}/{2}$. Similarly, we have $\|\bm W_l\| \le 3 \sigma^*_{\max} /2$ for each $l\in [L]$. 
    Using Weyl's inequality  (see, e.g., \cite[Corollary 7.3.5]{horn2012matrix}), we have $\left|\sigma_i(\bm W_l) - \sigma_i^*\right| \leq \|\bm W_l - \bm W_l^*\| \leq  \sigma_{\min}^*/{2}$ for all $i \in [r_{\sigma}]$ and $l \in [L]$. This implies $ \sigma_i(\bm W_l) \geq \sigma_i^* - {\sigma_{\min}^*}/{2} \geq {\sigma_{\min}^*}/{2}.$\\
    (ii) According to \eqref{eq:grad G}, we compute 
    \begin{align*}
        &  \nabla_{l} G(\bm W) \bm W_l^T = 2\bm W_{L:l+1}^T \left(\bm W_{L:1} - \sqrt{\lambda}\bm Y\right) \bm W_{l:1}^T + 2\lambda \bm W_l \bm W_l^T, \\
        &  \bm W_{l+1}^T \nabla_{l+1} G(\bm W) = 2\bm W_{L:l+1}^T \left(\bm W_{L:1} - \sqrt{\lambda}\bm Y\right) \bm W_{l:1}^T + 2\lambda \bm W_{l+1}^T \bm W_{l+1}.
    \end{align*}
    Then, we have $\left(\nabla_{l} G(\bm W) \bm W_l^T - \bm W_{l+1}^T \nabla_{l+1} G(\bm W)\right)/2 = \lambda \left( \bm W_l \bm W_l^T - \bm W_{l+1}^T \bm W_{l+1} \right)$ for each $l\in [L-1]$. This implies
    \begin{align*}
        \left\| \bm W_l \bm W_l^T - \bm W_{l+1}^T \bm W_{l+1} \right\|_F & = \frac{1}{2\lambda}\left\|  \nabla_{l} G(\bm W) \bm W_l^T - \bm W_{l+1}^T \nabla_{l+1} G(\bm W) \right\|_F \\
        & \le  \frac{3\sigma_{\max}^*}{4\lambda}  \left( \|\nabla_{l} G(\bm W)\|_F +  \|\nabla_{l+1} G(\bm W)\|_F \right)  \le  \frac{3\sqrt{2}\sigma_{\max}^*}{4\lambda}  \|\nabla G(\bm W)\|_F,
    \end{align*}
    where the first inequality follows from the triangle inequality, $\|\bm A\bm B\|_F \le \|\bm A\| \|\bm B\|_F$, and \eqref{eq:sigmalowerbound}.  
    
    (iii) 
    Using Weyl's inequality and \eqref{eq:balance}, we obtain that for each $l \in [L-1]$ and $i \in [r_{\sigma}]$, 
        \[
        \left| \sigma_i^2(\bm W_l) -\sigma_i^2(\bm W_{l+1})\right| \leq \left\| \bm W_l \bm W_l^T - \bm W_{l+1}^T \bm W_{l+1} \right\|_F \leq \frac{3\sqrt{2}\sigma_{\max}^*}{4\lambda} \|\nabla G(\bm W)\|_F.
        \]
    This implies
        \begin{align*}
     \left|\sigma_i(\bm W_l) - \sigma_i(\bm W_{l+1}) \right|  \leq \frac{3\sqrt{2}\sigma_{\max}^*}{4\lambda\left( \sigma_i(\bm W_l) + \sigma_i(\bm W_{l+1}) \right)} \|\nabla G(\bm W)\|_F  \overset{\eqref{eq:sigmalowerbound}}{\leq} \frac{3\sqrt{2}\sigma_{\max}^*}{4\lambda\sigma_{\min}^*}\|\nabla G(\bm W)\|_F.
        \end{align*}

 (iv)  We first prove \eqref{eq:transwji}. Note that 
    \begin{align}\label{eq1:lem pushl}
     \bm W_{j:i}^T \bm W_{j:i} - (\bm W^T_i\bm W_i)^{j-i+1}   
     = \sum_{k=1}^{j-i} \bm W_{j-k:i}^T\left( (\bm W_{j-k+1}^T\bm W_{j-k+1})^k - (\bm W_{j-k}\bm W_{j-k}^T)^{k} \right) \bm W_{j-k:i}. 
    \end{align}
    For each $k \in [j-i]$, we compute
    \begin{align*}
        & \left\|\bm W_{j-k:i}^T\left( (\bm W_{j+1-k}^T\bm W_{j+1-k})^k - (\bm W_{j-k}\bm W_{j-k}^T)^k \right) \bm W_{j-k:i}\right\|_F \\
       \le\ & \left\| (\bm W_{j+1-k}^T\bm W_{j+1-k})^k - (\bm W_{j-k}\bm W_{j-k}^T)^k \right\|_F \prod_{l=i}^{j-k} \|\bm W_l\|^2 
       \le   k\left(\frac{3\sigma_{\max}^*}{2}\right)^{2(j-i)}\left\|\bm W_{j+1-k}^T\bm W_{j+1-k} - \bm W_{j-k}\bm W_{j-k}^T\right\|_F\\
       \overset{\eqref{eq:balance}}{\le} \ &  k\left(\frac{3\sigma_{\max}^*}{2}\right)^{2(j-i)} \frac{3\sqrt{2}\sigma_{\max}^*}{4\lambda}  \|\nabla G(\bm W)\|_F = \frac{\sqrt{2} k}{2\lambda}\left(\frac{3\sigma_{\max}^*}{2}\right)^{2(j-i)+1}\|\nabla G(\bm W)\|_F,
    \end{align*}
    where the second inequality follows from  
\begin{align*}
&\left\|(\bm W_{j+1-k}^T \bm W_{j+1-k})^k - (\bm W_{j-k}\bm W_{j-k}^T)^k\right\|_F \\
=\ &
 \left\|\sum_{l=1}^k (\bm W_{j+1-k}^T \bm W_{j+1-k})^{k-l} (\bm W_{j+1-k}^T \bm W_{j+1-k} - \bm W_{j-k} \bm W_{j-k}^T) (\bm W_{j-k}\bm W_{j-k}^T )^{l-1}\right\|_F\\
\le\ & \left\| \bm W_{j+1-k}^T\bm W_{j+1-k} - \bm W_{j-k}\bm W_{j-k}^T \right\|_F \sum_{l=1}^k \|\bm W_{j+1-k}\|^{2(k-l)}\|\bm W_{j-k}\|^{2(l-1)}\\
\overset{\eqref{eq:sigmalowerbound}}{\le}\ &  k\left( \frac{3\sigma_{\max}^*}{2} \right)^{2(k-1)}\left\| \bm W_{j+1-k}^T\bm W_{j+1-k} - \bm W_{j-k}\bm W_{j-k}^T \right\|_F.
\end{align*}
This, together with \eqref{eq1:lem pushl}, yields
\begin{align*}
   \left\|\bm W_{j:i}^T \bm W_{j:i} - (\bm W^T_i\bm W_i)^{j-i+1}\right\|_F & \leq \left(\sum_{k=1}^{j-i} k\right) \frac{\sqrt{2}}{2\lambda}\left(\frac{3\sigma_{\max}^*}{2}\right)^{2j-2i+1}\|\nabla G(\bm W)\|_F  \\
& \le \frac{(j-i)(j-i+1)}{2\sqrt{2}\lambda}\left(\frac{3\sigma_{\max}^*}{2}\right)^{2j-2i+1}\|\nabla G(\bm W)\|_F.
\end{align*}
Using the same argument, we can prove \eqref{eq:transwij}.  
\end{proof}

Next, we replace $\bm W_{L:l+1}^T \bm W_{L:l+1}$ in \eqref{eq:grad G} with $(\bm W_{l}\bm W_{l}^T)^{L-l}$ and $\bm W_{l-1:1}\bm W_{l-1:1}^T$ in \eqref{eq:grad G} with $(\bm W_{l}^T\bm W_l)^{l-1}$, respectively, and show that the resulting terms can be upper bounded by $\|\nabla G(\bm W)\|_F$
by leveraging \Cref{lem:pre}(iv). This is formally established in the following corollary, a key intermediate step toward establishing the error bound. 
\begin{coro}\label{coro1}
Under the setting of \Cref{lem:pre}, we have 
\begin{align}\label{eq:coro1}
\left\| (\bm W_l \bm W^T_l)^{L-1} \bm W_l - \sqrt{\lambda}\bm W_{L:l+1}^T \bm Y \bm W_{l-1:1}^T + \lambda \bm W_l\right\|_F \le c_1 \|\nabla G(\bm W)\|_F,\ \forall l \in [L], 
\end{align}
where 
\begin{align}\label{eq:c1} 
  c_1 := \max_{l\in [L]} \left\{ \left(\frac{3\sigma^*_{\max}}{2}\right)^{2L-2}\frac{(L-l)(L-l+1)+(l-1)l}{2\sqrt{2}\lambda}+\frac{1}{2} \right\}. 
\end{align} 
\end{coro}
\begin{proof}
 For ease of exposition, let 
     \begin{align*}
         & \bm R_1(l) := \bm W_{l-1:1} \bm W_{l-1:1}^T - (\bm W_l^T\bm W_l)^{l-1},\ l =2,3,\dots,L,& \bm R_1(1) := \bm 0, \\
         & \bm R_2(l) := \bm W_{L:l+1}^T \bm W_{L:l+1}  -  (\bm W_l\bm W_l^T)^{L-l},\ l=1,2,\dots,L-1,& \bm R_2(L) := \bm 0.  
     \end{align*}
For each $l \in [L]$, we compute
\begin{align*}
         & \left\|(\bm W_{l-1}\bm W_{l-1}^T)^{l-1}- (\bm W_l^T\bm W_l)^{l-1}\right\|_F \\
         \le \quad&\left\|(\bm W_{l-1}\bm W_{l-1}^T)^{l-2}(\bm W_{l-1}\bm W_{l-1}^T -\bm W_l^T\bm W_l ) \right\|_F + \cdots   +\left\| (\bm W_{l-1}\bm W_{l-1}^T -\bm W_l^T\bm W_l )(\bm W_l^T\bm W_l)^{l-2} \right\|_F \\
         \overset{(\ref{eq:sigmalowerbound}, \ref{eq:balance})}{\le} & \frac{3\sqrt{2}(l-1)\sigma_{\max}^*}{4\lambda} \left(\frac{3\sigma_{\max}^*}{2}\right)^{2l-4} \|\nabla G(\bm W)\|_F.
\end{align*} 
This, together with the triangle inequality and \eqref{eq:transwij}, yields that for $l =2,3,\dots,L$, 
\begin{align}\label{eq1:coro1}
\|\bm R_1(l)\|_F \le &\ \left\|\bm W_{l-1:1} \bm W_{l-1:1}^T -(\bm W_{l-1}\bm W_{l-1}^T)^{l-1} \right\|_F + \left\| (\bm W_{l-1}\bm W_{l-1}^T)^{l-1}- (\bm W_l^T\bm W_l)^{l-1}\right\|_F \notag \\
\le &\ \frac{l(l-1)}{2\sqrt{2}\lambda} \left(\frac{3\sigma_{\max}^*}{2}\right)^{2l-3} \|\nabla G(\bm W)\|_F.
\end{align}    
Using the same argument and \eqref{eq:transwji}, we have for $l = 1,2,\dots,L-1$, 
\begin{align}\label{eq2:coro1}
 \|\bm R_2(l)\|_F \le \frac{(L-l+1)(L-l)}{2\sqrt{2}\lambda} \left(\frac{3\sigma_{\max}^*}{2}\right)^{2(L-l)-1} \|\nabla G(\bm W)\|_F.
\end{align}
For each $l =1,2,\ldots L$, we compute  
     \begin{align*}
     \frac{1}{2}\nabla_{l} G(\bm W) \overset{\eqref{eq:grad G}}{=} &\bm W_{L:l+1}^T(\bm W_L\cdots\bm W_1- \sqrt{\lambda}\bm Y) \bm W_{l-1:1}^T + \lambda\bm W_l     \\
    =~ & (\bm W_l\bm W^T_l)^{L-1}\bm W_l- \sqrt{\lambda}\bm W_{L:l+1}^T \bm Y \bm W_{l-1:1}^T +  \lambda\bm W_l + 
     \bm R_2(l)\bm W_l\bm W_{l-1:1}\bm W_{l-1:1}^T + \left(\bm W_l\bm W_l^T\right)^{L-l}\bm W_l\bm R_1(l).
     \end{align*}
     This, together with the triangle inequality, yields for each $l \in [L]$, 
        \begin{align*}
            &\left\| (\bm W_l \bm W^T_l)^{L-1} \bm W_l - \sqrt{\lambda}\bm W_{L:l+1}^T \bm Y \bm W_{l-1:1}^T + \lambda \bm W_l\right\|_F \notag \\
            \leq\ &  \left\|\bm R_2(l)\bm W_l \bm W_{l-1:1}\bm W_{l-1:1}^T\right\|_F+\left\|(\bm W_{l}\bm W_{l}^T)^{L-l}\bm W_l\bm R_1(l)\right\|_F + \frac{1}{2}\|\nabla G(\bm W)\|_F \notag \\
            \leq\ & \|\bm R_2(l)\|_F \|\bmw_l\|\left(\prod_{i=1}^{l-1}\|\bmw_i\|^2\right)+\|\bm R_1(l)\|_F\|\bmw_l\|^{2L-2l+1}+ \frac{1}{2}\|\nabla G(\bm W)\|_F\\
        \overset{\eqref{eq:sigmalowerbound}}{\le}&\! \left(\frac{3\sigma^*_{\max}}{2}\right)^{2l-1}\|\bm R_2(l)\|_F + \left(\frac{3\sigma^*_{\max}}{2}\right)^{2L-2l+1}\|\bm R_1(l)\|_F+ \frac{1}{2}\|\nabla G(\bm W)\|_F\\
        \overset{(\ref{eq1:coro1},\ref{eq2:coro1})}{\leq}\!\!\!\!\!&~~~\left( \left(\frac{3\sigma^*_{\max}}{2}\right)^{2L-2}\frac{(L-l)(L-l+1)+(l-1)l}{2\sqrt{2}\lambda}+\frac{1}{2} \right) \|\nabla G(\bm W)\|_F,
        \end{align*}
    which directly implies \eqref{eq:coro1}. 
\end{proof}

To handle the repeated singular values in $\bm \sigma^* \in \mathcal{A}_{\rm sort}\setminus \{\bm 0\}$, let $p \ge 1$ denote the number of distinct positive singular values. Recall from \eqref{def:r sigma} that $r_{\sigma}$ denotes the number of positive singular values of $\bm \sigma^*$. Then there exist indices $t_0,t_1,\cdots,t_{p}$ such that 
$0 = t_0 < t_1<\cdots < t_p = r_{\sigma} $ and 
\begin{equation}\label{eq:partsigma}
\sigma_{t_0+1}^* =\dots = \sigma_{t_1}^* >\sigma_{t_1+1}^*=\dots =\sigma_{t_2}^* > \dots>\sigma_{t_{p-1}+1}^*=\dots =\sigma_{t_{p}}^*>0.
\end{equation} 
Let $g_i := t_i - t_{i-1}$ be the multiplicity of the $i$-th largest positive value of $\bm \sigma^*$ for each $i\in [p]$ and $g_{\max} := \max\{g_i:i \in [p]\}$. Obviously, we have $r_{\sigma}= \sum_{i=1}^p g_i$ . Moreover, let 
\begin{align}\label{eq:SVD Wl}
\bW_l = \bU_l\bm{\Sigma}_l\bV_l^T = \begin{bmatrix}
    \bm U_{l}^{(1)} & \dots &  \bm U_{l}^{(p)} & \bm U_{l}^{(p+1)} 
\end{bmatrix}\mathrm{BlkD}\left(\bm \Sigma_{l}^{(1)},\dots,\bm \Sigma_{l}^{(p)},\bm \Sigma_{l}^{(p+1)} \right)
{\begin{bmatrix}
    \bm V_{l}^{(1)} & \dots &  \bm V_{l}^{(p)} & \bm V_{l}^{(p+1)} 
\end{bmatrix}}^T
\end{align}
be an SVD of $\bm W_l$ for each $l \in [L]$, where $\bm U_l \in \mathcal{O}^{d_l}$ with $\bm U_{l}^{(i)} \in \R^{d_l \times g_i}$ for each $i \in [p]$ and $\bm U_{l}^{(p+1)} \in \R^{d_l \times (d_l-r_{\sigma})}$, $\bm \Sigma_l \in \mathbb{R}^{d_l\times d_{l-1}}$ with decreasing singular values, $\bm \Sigma_{l}^{(i)} \in \R^{g_i \times g_i}$ for each $i \in [p]$, and $\bm \Sigma_{l}^{(p+1)} \in \R^{(d_l-r_{\sigma})  \times (d_{l-1}-r_{\sigma})}$, and $\bm V_l \in \mathcal{O}^{d_{l-1}}$ with $\bm V_{l}^{(i)} \in \R^{d_{l-1} \times g_i}$ for each $i \in [p]$ and $\bm V_{l}^{(p+1)} \in \R^{d_{l-1} \times (d_{l-1} - r_{\sigma})}$. 

With the above setup, we are ready to show that, for any point in the neighborhood of the set of critical points, the product of $\bm U_{l-1}$ and $\bm V_{l}$ is close to a block diagonal matrix with orthogonal diagonal blocks. Recall that $\delta_{\sigma}$  and $\mathcal{A}_{\rm sort}$ are defined in \eqref{set:Y} and \eqref{set:W sigma}, respectively.
\begin{prop}\label{prop:Hlapproximate}
Suppose that Assumption~\ref{AS:1} holds. Let $\bm W = (\bm W_1,\dots,\bm W_L)$ and $ \bm \sigma^* \in \mathcal{A}_{\rm sort}\setminus\{\bm 0 \}$ be arbitrary such that 
\begin{align}\label{eq:dist 2}
    \mathrm{dist}(\bm W, \mathcal{W}_{\bm \sigma^*}) \le   \frac{\delta_{\sigma}}{{6}}. 
\end{align}  
For each $l \in \{2,\dots,L\}$, there exist matrices $\bm T_{l}^{(i)} \in \mathcal{O}^{g_i}$ for all $i \in [p]$ and $\bm T_{l}^{(p+1)} \in \mathcal{O}^{d_{l-1} - r_{\sigma}}$ such that
\begin{align}\label{eq:Hlapproximate}
  \left\| \bm U_{l-1}^T\bm V_l - \mathrm{BlkD}\left(\bm T_{l}^{(1)},\dots,\bm T_{l}^{(p)},\bm T_{l}^{(p+1)} \right)\right\|_F \le \frac{9\sigma_{\max}^{*}}{ 4\delta_{\sigma}\lambda\sigma^{*}_{\min}}\|\nabla G(\bm W)\|_F. 
\end{align}
where $\bm U_l$ and $\bm V_l$ for each $l \in [L]$ are defined in \eqref{eq:SVD Wl}. 
\end{prop} 
\begin{proof}  
Let $\bm W^* = (\bm W_1^*,\dots,\bm W_L^*) \in \mathcal{W}_{\bm \sigma^*}$ be such that $\mathrm{dist}(\bm W, \mathcal{W}_{\bm \sigma^*}) = \|\bm W - \bm W^* \|_F$.  For ease of exposition, let $\bm H_l := \bm U_{l-1}^T\bm V_l$ for each $l \in \{2,\dots,L\}$. According to \eqref{eq:SVD Wl}, the $(i,j)$-th block of $\bm H_l$, denoted by $\bm H_{l}^{(i,j)}$, is  
\begin{align}\label{eq:Hlij}
    \bm H_{l}^{(i,j)} := \bm U_{l-1}^{(i)^T}\bm V_{l}^{(j)},\ \forall (i,j) \in [p+1] \times [p+1]. 
\end{align} 
By \eqref{eq:dist 2} with $\delta_{\sigma} \le \sigma_{\min}^*$ and \Cref{lem:pre}, we have \eqref{eq:balance}. Substituting \eqref{eq:SVD Wl} and $\bm H_l = \bm U_{l-1}^T \bm V_l$ into \eqref{eq:balance} yields
\begin{align*}
 \left\| \bm H_l\bm{\Sigma}_l^T\bm{\Sigma}_l -  \bm{\Sigma}_{l-1}\bm{\Sigma}_{l-1}^T\bm H_l  \right\|_F \leq \frac{3\sqrt{2}\sigma^*_{\max}}{4\lambda} \|\nabla G(\bm W)\|_F,\ \forall l = 2,\dots,L.
\end{align*} 
Using this and the block diagonal structures of $\bm \Sigma_l$ in \eqref{eq:SVD Wl} and $\bm H_l$ in \eqref{eq:Hlij}, we obtain
\begin{align}\label{eq1:prop Hla}
    &\ \sum_{j=1}^{p+1} \sum_{i=1, i \neq j}^{p+1} \left\| \bm H_{l}^{(i,j)}\bm \Sigma_{l}^{(j)^T}\bm \Sigma_{l}^{(j)} - \bm \Sigma_{l-1}^{(i)}\bm \Sigma_{l-1}^{(i)^T} \bm H_{l}^{(i,j)} \right\|_F^2 \notag \\
\le &\    \sum_{j=1}^{p+1} \sum_{i=1}^{p+1} \left\| \bm H_{l}^{(i,j)}\bm \Sigma_{l}^{(j)^T}\bm \Sigma_{l}^{(j)} - \bm \Sigma_{l-1}^{(i)}\bm \Sigma_{l-1}^{(i)^T} \bm H_{l}^{(i,j)} \right\|_F^2 \le \left( \frac{3\sqrt{2}\sigma^*_{\max}}{4\lambda}\right)^2  \|\nabla G(\bm W)\|_F^2. 
\end{align} 
For notational convenience, set $\sigma^*_{t_{p+1}}:=0$.
For each  $l \in [L]$ and $i \in [p+1]$, we have $\|{\bm \Sigma_{l}^{(i)}} - \sigma^*_{t_i} \bm I\| \le \|\bm W_l - \bm W_l^*\|  \leq \mathrm{dist}(\bm W, \mathcal{W}_{\bm \sigma^*}) \overset{\eqref{eq:dist 2}}{\le} {\delta_{\sigma}}/{{6}},$ where the first inequality follows from Mirsky's inequality (see \Cref{lem:mirsky} in Appendix~\ref{app:sec auxi}). Using Weyl's inequality, we have $\sigma^*_{t_i} - \frac{\delta_{\sigma}}{{6}} \le {\sigma_{\min}(\bm \Sigma_{l}^{(i)}}) \le {\sigma_{\max}(\bm \Sigma_{l}^{(i)})} \le \sigma^*_{t_i} + \frac{\delta_{\sigma}}{{6}}$. For each $l \in [L]$ and $i, j \in [p+1]$ with $i > j$, we compute
\begin{align}
    & \left\|\bm H_{l}^{(i,j)}\bm \Sigma_{l}^{(j)^T}\bm \Sigma_{l}^{(j)}\right\|_F \ge \sigma_{\min}^2(\bm \Sigma_{l}^{(j)})\|\bm H_{l}^{(i,j)}\|_F \ge  \left(  \sigma^*_{t_j} - \frac{\delta_{\sigma}}{{6}} \right)^2 \|\bm H_{l}^{(i,j)}\|_F, \label{eq2:prop Hla}\\
    & \left\|\bm \Sigma_{l-1}^{(i)}\bm \Sigma_{l-1}^{(i)^T}\bm H_{l}^{(i,j)}\right\|_F \le \sigma_{\max}\left( \bm \Sigma_{l-1}^{(i)}\bm \Sigma_{l-1}^{(i)^T} \right)\|\bm H_{l}^{(i,j)}\|_F \le \left(  \sigma^*_{t_i} + \frac{\delta_{\sigma}}{{6}} \right)^2\|\bm H_{l}^{(i,j)}\|_F. \label{eq3:prop Hla}
\end{align}
For each $l \in [L]$ and  $i > j \in [p+1]$, we bound
\begin{align*}
&\quad \left\| \bm H_{l}^{(i,j)}\bm \Sigma_{l}^{(j)^T}\bm \Sigma_{l}^{(j)} - \bm \Sigma_{l-1}^{(i)}\bm \Sigma_{l-1}^{(i)^T} \bm H_{l}^{(i,j)} \right\|_F \ge \left\| \bm H_{l}^{(i,j)}\bm \Sigma_{l}^{(j)^T}\bm \Sigma_{l}^{(j)}\right\|_F - \left\|\bm \Sigma_{l-1}^{(i)}\bm \Sigma_{l-1}^{(i)^T} \bm H_{l}^{(i,j)}\right\|_F \\
& \overset{(\ref{eq2:prop Hla},\ref{eq3:prop Hla})}{\ge}  \left(  \left(  \sigma^*_{t_j} - \frac{\delta_{\sigma}}{{6}} \right)^2 - \left(  \sigma^*_{t_i} + \frac{\delta_{\sigma}}{{6}} \right)^2 \right) \|\bm H_{l}^{(i,j)}\|_F 
\overset{\eqref{set:Y}}{\ge} \frac{2\delta_{\sigma} \sigma^*_{\min}}{3}\|\bm H_{l}^{(i,j)}\|_F. 
\end{align*}
Applying the same argument to the case where $i < j$, we obtain the same result as above. This, together with \eqref{eq1:prop Hla}, yields the following bound for off-diagonal blocks of $\bm H_{l}$:
\begin{align}\label{eq4:prop Hla}  
\sum_{j=1}^{p+1} \sum_{ i=1, i \neq j}^{p+1} \|\bm H_{l}^{(i,j)}\|_F^2  \le \frac{81\sigma_{\max}^{*2}}{32\delta_{\sigma}^2\lambda^2\sigma^{*2}_{\min}}\|\nabla G(\bm W)\|_F^2. 
\end{align}
For each $l \in [L]$ and $i\in [p+1]$, let $\bm H_{l}^{(i,i)} = \bm P_{l}^{(i)}\bm \Lambda_{l}^{(i)}\bm Q_{l}^{(i)^T}$ be a full SVD of $\bm H_{l}^{(i,i)}$, where $\bm P_{l}^{(i)},\ \bm Q_{l}^{(i)}$ are square orthogonal matrix. For notational simplicity, for the fixed \(l\), set \(g_{p+1}:=d_{l-1}-r_\sigma\). Then, we compute
\begin{align*}
 \left\|\bm H_{l}^{(i,i)}- \bm P_{l}^{(i)}\bm Q_{l}^{(i)^T}\right\|_F^2 & = \sum_{k=1}^{g_i} \left( 1 - \sigma_k(\bm H_{l}^{(i,i)}) \right)^2 \le \sum_{k=1}^{g_i} \left( 1 - \sigma_k^2(\bm H_{l}^{(i,i)}) \right)^2   \\
 &  = \left\| \bm I -  \bm H_{l}^{(i,i)}\bm H_{l}^{(i,i)^T} \right\|_F^2 \le \mathrm{tr}
 \left(\bm I -  \bm H_{l}^{(i,i)}\bm H_{l}^{(i,i)^T} \right) =  \sum_{j=1,j\neq i}^{p + 1} \|\bm H_{l}^{(i,j)}\|_F^2,
\end{align*}  
 where the second inequality and the last equality is due to 
$[\bm H_{l}^{(i,1)},\dots,\bm H_{l}^{(i,p+1)}] = \bm U_{l-1}^{(i)^T}\bm V_l$ with $\bm U_{l-1},\ \bm V_l \in \mathcal{O}^{d_{l-1}}$ for all $l\in [L]$ and $ i\in [p+1]$. Therefore, we obtain 
\begin{align*}
\sum_{i=1}^{p+1} \left\|\bm H_{l}^{(i,i)}- \bm P_{l}^{(i)}\bm Q_{l}^{(i)^T}\right\|_F^2 \le \sum_{i=1}^{p+1} \sum_{j=1, j\neq i}^{p + 1} \|\bm H_{l}^{(i,j)}\|_F^2 \overset{\eqref{eq4:prop Hla}}{\le}  \frac{81\sigma_{\max}^{*2}}{32\delta_{\sigma}^2\lambda^2\sigma^{*2}_{\min}}\|\nabla G(\bm W)\|_F^2. 
\end{align*}
This, together with the definition of $\bm H_l$, \eqref{eq4:prop Hla} and $\bm T_{l}^{(i)} =\bm P_{l}^{(i)}\bm Q_{l}^{(i)^T}$, yields  \eqref{eq:Hlapproximate}. 
\end{proof}
      
Next, using the above proposition, we show that for weight matrices in the neighborhood of the set of critical points, the singular matrices $\bm \Sigma_1$ and $\bm \Sigma_L$ satisfy the following spectral inequalities. 
\begin{lemma}\label{lem:twoinequality}
Consider the setting of \Cref{lem:pre}. Suppose that there exist matrices $\bm T_{l}^{(i)} \in \mathcal{O}^{g_i}$ for all $i \in [p]$ and $\bm T_{l}^{(p+1)} \in \mathcal{O}^{d_{l-1} - r_{\sigma}}$ such that \eqref{eq:Hlapproximate} holds for all $l=2,\dots,L$. Let $c_1$ be defined in \eqref{eq:c1} and 
\begin{align} 
& \bm A_i := \left(\prod_{l=1}^{L-1} \bm T_{l+1}^{(i)}\right) \left(\bm \Sigma_{1}^{(i)}\right)^{L-1},\ \forall i \in [p],\ \bm A_{p+1} :=  \prod_{l=1}^{L-1}  \bm T_{l+1}^{(p+1)}\bm \Sigma_{l+1}^{(p+1)^T},\label{eq:A} \\
&  \bm B_i := \left( \prod_{l=1}^{L-1} \bm T_{l+1}^{(i)} \right) \left(\bm \Sigma_{L}^{(i)}\right)^{L-1},\ \forall i \in [p],\ \bm B_{p+1} :=  \prod_{l=1}^{L-1}  \bm \Sigma_{l}^{(p+1)^T} \bm T_{l+1}^{(p+1)}, \label{eq:B}\\
& \eta_1 := \frac{\sigma_{\max}^*}{\sigma_{\min}^*}\left(  \frac{3\sqrt{2}\sigma_{\min}^*}{4\lambda} +   \frac{81\sigma_{\max}^{*{2}}}{8\delta_{\sigma}\lambda} +    \frac{9\sqrt{2g_{\max}} L\sigma_{\max}^{*}}{4\lambda} \right), \label{eq:eta1}\\
& c_2 := \left( \frac{3}{2}\sigma_{\max}^* \right)^{L}\frac{3y_1L}{2\sqrt{\lambda}\delta_{\sigma} \sigma^{*}_{\min}} + c_1+ y_1 p \sqrt{\lambda} \left(\frac{3\sigma_{\max}^*}{2}\right)^{L-2}\left(\frac{L^2\eta_1 }{2\sigma^*_{\min}}+  \frac{3\sqrt{2g_{\max}}L^2\sigma_{\max}^*}{2\lambda\sigma_{\min}^*} \right). \label{eq:c2}
\end{align}
Then, we have 
\begin{align}
& \left\|(\bm \Sigma_1\bm \Sigma_1^T)^{L-1}\bm \Sigma_1 + \lambda \bm \Sigma_1 - \sqrt{\lambda}\mathrm{BlkD}\left(\bm A_1, \dots, \bm A_{p}, \bm A_{p+1}\right)\bm U_L^T \bm Y \bm V_1  \right\|_F \leq c_2 \|\nabla G(\bm W)\|_F, \label{eq:twoinequality2} \\
& \left\|(\bm \Sigma_L\bm \Sigma_L^T)^{L-1}\bm \Sigma_L + \lambda \bm \Sigma_L - \sqrt{\lambda}\bm U_L^T \bm Y \bm V_1 \mathrm{BlkD}\left(\bm B_1, \dots, \bm B_{p}, \bm B_{p+1}\right) \right\|_F \leq c_2 \|\nabla G(\bm W)\|_F.\label{eq:twoinequality1} 
\end{align} 
\end{lemma} 
\begin{proof}  Let $\bm W^* = (\bm W_1^*,\dots,\bm W_L^*) \in \mathcal{W}_{\bm \sigma^*}$ be such that $\mathrm{dist}(\bm W, \mathcal{W}_{\bm \sigma^*}) = \|\bm W - \bm W^*\|_F$. It follows from \Cref{lem:pre} and \Cref{coro1} that \eqref{eq:sigmalowerbound}-\eqref{eq:wl+1} and \eqref{eq:coro1} hold. Recall that \eqref{eq:SVD Wl} denotes an SVD of $\bm W_l$ for each $l \in [L]$.  For ease of exposition, let $\bm H_l:=\bm U_{l-1}^T\bm V_l$  and $\bm T_l := \mathrm{BlkD}\left(\bm T_{l}^{(1)},\dots,\bm T_{l}^{(p)},\bm T_{l}^{(p+1)}\right)$ for each $l = 2,\dots,L$. This, together with \eqref{eq:Hlapproximate}, implies 
\begin{align}\label{eq0:lem two}
\left\|\bm H_l - \bm T_l \right\|_F \le \frac{9\sigma_{\max}^{*}}{ 4\delta_{\sigma}\lambda\sigma^{*}_{\min}}\|\nabla G(\bm W)\|_F,\ l=2,\dots,L. 
\end{align}

First, we are devoted to proving \eqref{eq:twoinequality1}. It follows from \eqref{eq:coro1} with $l=L$ that  
\begin{align}\label{eq1:lem two}
\left\| (\bm W_L \bm W^T_L)^{L-1} \bm W_L - \sqrt{\lambda} \bm Y \bm W_{L-1:1}^T + \lambda \bm W_L \right\|_F \le c_1\|\nabla G(\bm W)\|_F,
\end{align}
where $c_1$ is defined in \eqref{eq:c1}. Using the SVD in \eqref{eq:SVD Wl}, we have
\begin{align*}
&\left\| (\bm W_L \bm W^T_L)^{L-1} \bm W_L - \sqrt{\lambda} \bm Y \bm W_{L-1:1}^T + \lambda \bm W_L \right\|_F  
= \left\| (\bm \Sigma_L\bm \Sigma_L^T)^{L-1}\bm \Sigma_L - \sqrt{\lambda}\bm U_L^T\bm Y \bm V_1 \prod_{l=1}^{L-1} \bm \Sigma_l^T \bm H_{l+1} + \lambda \bm \Sigma_L \right\|_F \\
\ge &\left\| (\bm \Sigma_L\bm \Sigma_L^T)^{L-1}\bm \Sigma_L + \lambda \bm \Sigma_L  - \sqrt{\lambda}\bm U_L^T\bm Y \bm V_1 \prod_{l=1}^{L-1} \bm \Sigma_l^T \bm T_{l+1} \right\|_F - \sqrt{\lambda}\left\| \bm Y \bm V_1\bm \Sigma_1^T(\bm H_2 - \bm T_2)\prod_{l=2}^{L-1} \bm \Sigma_l^T \bm H_{l+1} \right\|_F \\
&  -  \sqrt{\lambda}\left\| \bm Y \bm V_1\bm \Sigma_1^T\bm T_2 \bm \Sigma_2^T(\bm H_3-\bm T_3)\prod_{l=3}^{L-1} \bm \Sigma_l^T \bm H_{l+1} \right\|_F -\dots - \sqrt{\lambda}\left\| \bm Y \bm V_1\left( \prod_{l=1}^{L-2}\bm \Sigma_l^T\bm T_{l+1} \right)\bm \Sigma_{L-1}^T(\bm H_{L} - \bm T_L) \right\|_F \\
\ge &  \left\| (\bm \Sigma_L\bm \Sigma_L^T)^{L-1}\bm \Sigma_L + \lambda \bm \Sigma_L  - \sqrt{\lambda}\bm U_L^T\bm Y \bm V_1 \prod_{l=1}^{L-1} \bm \Sigma_l^T \bm T_{l+1} \right\|_F - \left( \frac{3}{2}\sigma_{\max}^* \right)^{L}\frac{3 y_1 L}{2\sqrt{\lambda} \delta_{\sigma}\sigma^{*}_{\min}}\|\nabla G(\bm W)\|_F, 
\end{align*}
where the first inequality uses the triangle inequality and the last inequality follows from \eqref{eq:sigmalowerbound}, $\|\bm H_l\| =\|\bm T_l\| = 1$ for all $l \in [L]$, and \eqref{eq0:lem two}. This, together with \eqref{eq1:lem two}, implies 
\begin{align}\label{eq2:lem two}
&\ \left\| (\bm \Sigma_L\bm \Sigma_L^T)^{L-1}\bm \Sigma_L + \lambda \bm \Sigma_L  - \sqrt{\lambda}\bm U_L^T\bm Y \bm V_1 \prod_{l=1}^{L-1} \bm \Sigma_l^T \bm T_{l+1} \right\|_F \notag\\
\le &\ \left( \left( \frac{3}{2}\sigma_{\max}^* \right)^{L}\frac{3 y_1 L}{2 \sqrt{\lambda}\delta_{\sigma}\sigma^{*}_{\min}} + c_1\right) \|\nabla G(\bm W)\|_F.  
\end{align}
Using the diagonal structure of $\bm \Sigma_l$ in \eqref{eq:SVD Wl} and $\bm T_l = \mathrm{BlkD}\left(\bm T_{l}^{(1)},\dots,\bm T_{l}^{(p)},\bm T_{l}^{(p+1)} \right)$, we compute 
\begin{align*}
      \prod_{l=1}^{L-1}\bm\Sigma_l^T \bm T_{l+1} = \mathrm{BlkD}\left(\prod_{l=1}^{L-1}\bm\Sigma_{l}^{(1)}\bm T_{l+1}^{(1)},\dots,\prod_{l=1}^{L-1}\bm\Sigma_{l}^{(p)}\bm T_{l+1}^{(p)},\prod_{l=1}^{L-1}\bm\Sigma_{l}^{(p+1){T}}\bm T_{l+1}^{(p+1)}\right). 
\end{align*}
Using \eqref{eq:wl+1} in \Cref{lem:pre}, we have for each $l \in [L]$ and $i \in [p]$, 
    \begin{align*}
    \left\|\bm\Sigma_{l}^{(i)} -\bm \Sigma_{L}^{(i)}\right\|_F^2 &  = \sum_{j=1}^{g_i} \left( \sigma_j(\bm\Sigma_{l}^{(i)}) - \sigma_j(\bm \Sigma_{L}^{(i)})  \right)^2 = \sum_{j=1}^{g_i} \left( \sum_{k=l}^{L-1} \left( \sigma_j(\bm\Sigma_{k}^{(i)}) - \sigma_j(\bm\Sigma_{k+1}^{(i)})\right)  \right)^2\\
    &  \le g_i (L-1)^2 \left(\frac{3\sqrt{2}\sigma_{\max}^*}{4\lambda\sigma_{\min}^*} \|\nabla G(\bm W)\|_F\right)^2.
    \end{align*}
    This further implies
    \begin{align}\label{eq3:lem two}
  \left\|\bm\Sigma_{l}^{(i)} -\bm \Sigma_{L}^{(i)}\right\|_F \le  \frac{3\sqrt{2g_{\max}} L\sigma_{\max}^*}{4\lambda\sigma_{\min}^*}  \|\nabla G(\bm W)\|_F. 
    \end{align}

    Now we claim that it holds for each $l = 2,\ldots L$ and $i \in [p]$ that 
    \begin{align}\label{eq4:lem two}
         \left\|\bm T_{l}^{(i)} \bm \Sigma_{L}^{(i)} - \bm \Sigma_{L}^{(i)} \bm T_{l}^{(i)}\right\|_F \leq \frac{\eta_1}{\sigma_{\min}^*} \|\nabla G(\bm W)\|_F,
    \end{align}
    where $\eta_1$	is defined in \eqref{eq:eta1}. To maintain the flow of the main proof, the proof of this claim is deferred to the end. Using the triangle inequality, we obtain    
    \begin{align}\label{eq5:lem two}
    & \left\|\prod_{l=1}^{L-1} \bm{\Sigma}_{L}^{(i)} \bm{T}_{l+1}^{(i)} - \left( \prod_{l=1}^{L-1} \bm{T}_{l+1}^{(i)} \right) \left(\bm{\Sigma}_{L}^{(i)}\right)^{L-1} \right\|_F  \notag \\
    \le & \  \sum_{j=2}^L \left\|\underbrace{\left(\prod_{l=1}^{L-j} \bm{\Sigma}_{L}^{(i)} \bm{T}_{l+1}^{(i)}\right)\left(\bm \Sigma_{L}^{(i)} \prod_{k=L-j+2}^L \bm T_{k}^{(i
    )}  -\left(\prod_{k=L-j+2}^L \bm T_{k}^{(i)} \right)\bm\Sigma_{L}^{(i)}\right)\left(\bm\Sigma_{L}^{(i)}\right)^{j-2}}_{\bm 
R_j} \right\|_F. 
    \end{align}
For each $\|\bm R_j\|_F$, we bound
\begin{align*}
    \|\bm R_j\|_F& \le \left(\frac{3\sigma^*_{\max}}{2}\right)^{L-2}\left\|\bm \Sigma_{L}^{(i)} \prod_{k=L-j+2}^L \bm T_{k}^{(i)} - \left(\prod_{k=L-j+2}^L \bm T_{k}^{(i)}\right) \bm\Sigma_{L}^{(i)}\right\|_F\\
    &\le \left(\frac{3\sigma^*_{\max}}{2}\right)^{L-2}\left\|\left(\bm \Sigma_{L}^{(i)}\bm T_{L-j+2}^{(i)}-\bm T_{L-j+2}^{(i)}\bm \Sigma_{L}^{(i)}\right)\prod_{k=L-j+3}^L \bm T_{k}^{(i)}\right\|_F \\
    & \quad+  \left(\frac{3\sigma^*_{\max}}{2}\right)^{L-2}\left\|\bm T_{L-j+2}^{(i)}\left(\bm \Sigma_{L}^{(i)} \bm T_{L-j+3}^{(i)}- \bm T_{L-j+3}^{(i)}\bm \Sigma_{L}^{(i)}\right)\prod_{k=L-j+4}^L\bm T_{k}^{(i)} \right \|_F + \cdots\\
    & \quad+ \left( \frac{3\sigma^*_{\max}}{2} \right)^{L-2} \left\| \prod_{k=L-j+2}^{L-1} \bm T_{k}^{(i)} (\bm \Sigma_{L}^{(i)}\bm T_{L}^{(i)}-\bm T_{L}^{(i)}\bm \Sigma_{L}^{(i)})\right\|_F \\
    &\overset{\eqref{eq4:lem two}}{\le} (j-1)\left(\frac{3\sigma^*_{\max}}{2}\right)^{L-2} \frac{\eta_1}{\sigma^*_{\min}}\|\nabla G(\bm W)\|_F,  
\end{align*}
where the first inequality uses (i) of \Cref{lem:pre} and $\|\bm T_{l+1}^{(i)}\| = 1$ for all $l$ and $i$, and the second one is due the triangle inequality. This, together with \eqref{eq5:lem two}, implies 
\begin{align}\label{eq6:lem two}
\left\|\prod_{l=1}^{L-1} \bm{\Sigma}_{L}^{(i)} \bm{T}_{l+1}^{(i)} - \left( \prod_{l=1}^{L-1} \bm{T}_{l+1}^{(i)} \right) \left(\bm{\Sigma}_{L}^{(i)}\right)^{L-1} \right\|_F \leq  \left(\frac{3\sigma_{\max}^*}{2}\right)^{L-2}\frac{L^2\eta_1}{2\sigma^*_{\min}}\|\nabla G(\bm W)\|_F. 
\end{align} 
For each $i \in [p]$, we have 
    \begin{align*}
    &  \left\|\prod_{l=1}^{L-1} \bm{\Sigma}_{l}^{(i)} \bm{T}_{l+1}^{(i)} - \left( \prod_{l=1}^{L-1} \bm{T}_{l+1}^{(i)} \right) \left(\bm{\Sigma}_{L}^{(i)}\right)^{L-1} \right\|_F \notag \\
    \leq & \ \left\| \prod_{l=1}^{L-1} \bm{\Sigma}_{L}^{(i)}\bm{T}_{l+1}^{(i)} 
    - \left( \prod_{l=1}^{L-1} \bm{T}_{l+1}^{(i)} \right) (\bm{\Sigma}_{L}^{(i)})^{L-1} \right\|_F 
    + \left\|(\bm{\Sigma}_{1}^{(i)} - \bm{\Sigma}_{L}^{(i)}) \bm{T}_{2}^{(i)} \prod_{l=2}^{L-1} \bm{\Sigma}_{l}^{(i)} \bm{T}_{l+1}^{(i)} \right\|_F +  \notag \\
    &\ \left\| \bm{\Sigma}_{L}^{(i)} \bm{T}_{2}^{(i)}(\bm{\Sigma}_{2}^{(i)} - \bm{\Sigma}_{L}^{(i)}) \bm{T}_{3}^{(i)} \prod_{l=3}^{L-1} \bm{\Sigma}_{l}^{(i)} \bm{T}_{l+1}^{(i)} \right\|_F + \cdots + \left\|\prod_{l=1}^{L-2} \bm{\Sigma}_{l}^{(i)} \bm{T}_{l+1}^{(i)} (\bm{\Sigma}_{L-1}^{(i)} - \bm{\Sigma}_{L}^{(i)}) \bm{T}_{L}^{(i)}\right\|_F \\
    \leq &\ \left(\frac{3\sigma_{\max}^*}{2}\right)^{L-2}\left(\frac{L^2\eta_1 }{2\sigma^*_{\min}}+  \frac{3\sqrt{2g_{\max}}L^2\sigma_{\max}^*}{4\lambda\sigma_{\min}^*} \right)\|\nabla G(\bm W)\|_F,
    \end{align*}
    where the first inequality uses the triangle inequality, and the last inequality follows from \eqref{eq6:lem two}, \eqref{eq:sigmalowerbound} in \Cref{lem:pre}, \( \|\bm T_l\| = 1 \) for all \( l \in [L] \), and \eqref{eq3:lem two}. This, together with \eqref{eq:B}, yields 
\begin{align}\label{eq7:lem two}
&  \left\|\prod_{l=1}^{L-1}\bm\Sigma_l^T \bm T_{l+1} - \mathrm{BlkD}\left(\bm B_1, \dots, \bm B_{p}, \bm B_{p+1}\right)\right\|_F 
\le  \sum_{i=1}^p \left\|\prod_{l=1}^{L-1} \bm{\Sigma}_{l}^{(i)} \bm{T}_{l+1}^{(i)} - \left( \prod_{l=1}^{L-1} \bm{T}_{l+1}^{(i)} \right) \left(\bm{\Sigma}_{L}^{(i)}\right)^{L-1} \right\|_F \notag \\
\le & \ p \left(\frac{3\sigma_{\max}^*}{2}\right)^{L-2}\left(\frac{L^2\eta_1 }{2\sigma^*_{\min}}+  \frac{3\sqrt{2g_{\max}}L^2\sigma_{\max}^*}{4\lambda\sigma_{\min}^*} \right) \|\nabla G(\bm W)\|_F. 
\end{align}

Now, we are ready to prove \eqref{eq:twoinequality1}. Specifically, we have
	\begin{align*}
	&\ \left\|(\bm \Sigma_L\bm \Sigma_L^T)^{L-1}\bm \Sigma_L + \lambda \bm \Sigma_L - \sqrt{\lambda}\bm U_L^T \bm Y \bm V_1 \mathrm{BlkD}\left(\bm B_1, \dots, \bm B_{p}, \bm B_{p+1}\right) \right\|_F \\ 
	\le& \left\|(\bm \Sigma_L\bm \Sigma_L^T)^{L-1}\bm \Sigma_L + \lambda \bm \Sigma_L - \sqrt{\lambda}\bm U_L^T \bm Y \bm V_1 \prod_{l=1}^{L-1}\bm\Sigma_l^T \bm T_{l+1} \right\|_F +\sqrt{\lambda}y_1\left\|\prod_{l=1}^{L-1}\bm\Sigma_l^T \bm T_{l+1} - \mathrm{BlkD}\left(\bm B_1, \dots, \bm B_{p}, \bm B_{p+1}\right)\right\|_F\\
    \overset{(\ref{eq2:lem two}, \ref{eq7:lem two})}{\le}  &  c_2 \|\nabla G(\bm W)\|_F, 
\end{align*}	 
where $c_2$ is defined in \eqref{eq:c2}. In \eqref{eq1:lem two}, setting $l=1$ and applying the same argument, we obtain \eqref{eq:twoinequality2}.\medskip

The rest of the proof is devoted to proving the claim \eqref{eq4:lem two}. According to \eqref{eq:SVD Wl} and $\bm H_l=\bm U_{l-1}^T\bm V_l$  for each $l=2,\dots,L$, we have for each $i \in [p]$, 
\begin{align*}
   \quad& \left\|\bm W_l^T\bm W_{l}- \bm W_{l-1}\bm W_{l-1}^T\right\|_F = \left\|\bm H_l \bm \Sigma_l^T \bm \Sigma_l \bm H_l^T - \bm \Sigma_{l-1}\bm \Sigma_{l-1}^T\right\|_F \\
    \ge & \left\|\bm T_l \bm \Sigma_l^T \bm \Sigma_l \bm T_l^T - \bm \Sigma_{l-1}\bm \Sigma_{l-1}^T\right\|_F - 2\left\|\bm \Sigma_l\right\|^2 \left\|\bm H_l - \bm T_l\right\|_F \\
     \ge & \left\|\bm T_{l}^{(i)} (\bm \Sigma_{l}^{(i)})^2 - (\bm \Sigma_{l-1}^{(i)})^2\bm T_{l}^{(i)}\right\|_F - 2\left\|\bm \Sigma_l\|^2 \|\bm H_l - \bm T_l\right\|_F \\
      \ge & \left\|\bm T_{l}^{(i)} (\bm\Sigma_{L}^{(i)})^2 -(\bm\Sigma_{L}^{(i)})^2\bm T_{l}^{(i)}\right\|_F-2\left\|\bm \Sigma_l\right\|^2 \left\|\bm H_l - \bm T_l\right\|_F - \left\|(\bm\Sigma_{L}^{(i)})^2-(\bm\Sigma_{l}^{(i)})^2\right\|_F - \left\|(\bm\Sigma_{L}^{(i)})^2-(\bm\Sigma_{l-1}^{(i)})^2\right\|_F\\
      \overset{\eqref{eq:sigmalowerbound}}{\ge} & \left\|\bm T_{l}^{(i)} (\bm\Sigma_{L}^{(i)})^2 -(\bm\Sigma_{L}^{(i)})^2\bm T_{l}^{(i)}\right\|_F - 2\|\bm \Sigma_l\|^2\|\bm H_l - \bm T_l\|_F -3\sigma^*_{\max}\left(\left\|\bm\Sigma_{L}^{(i)}-\bm\Sigma_{l}^{(i)}\right\|_F + \left\|\bm\Sigma_{L}^{(i)}-\bm\Sigma_{l-1}^{(i)}\right\|_F\right),
    \end{align*}
    where the second inequality follows from $\bm T_{l}^{(i)} \in \mathcal{O}^{g_i}$ for each $i \in [p]$. This together with  \eqref{eq:balance}, \eqref{eq0:lem two}, and \eqref{eq3:lem two}, yields that for each $l=2,\dots,L$ and $i \in [p]$, we have 
    \begin{align}\label{eq9:lem two}
    \left\|\bm T_{l}^{(i)} (\bm\Sigma_{L}^{(i)})^2 -(\bm\Sigma_{L}^{(i)})^2\bm T_{l}^{(i)}\right\|_F \le \eta_1 \left\|\nabla G(\bm W)\right\|_F.  
    \end{align}
For each $i \in [p]$, using Weyl's inequality yields 
        $
        |\sigma_j(\bm\Sigma_{L}^{(i)}) - \sigma_{t_i}^*| \leq \|\bm W_L - \bm W_L^*\| \leq {\sigma_{\min}^*}/{2} 
        $
        for each $j \in [g_i]$, which implies \(\| \bm\Sigma_{L}^{(i)} -\sigma_{t_i}^* \bm I\| \le \sigma_{\min}^*/2 \). This, together with \Cref{lem:commute} in Appendix~\ref{app:lem commute} and $\bm T_{l}^{(i)} \in \mathcal{O}^{g_i}$, yields 
$ \|\bm T_{l}^{(i)} (\bm\Sigma_{L}^{(i)})^2 -(\bm\Sigma_{L}^{(i)})^2\bm T_{l}^{(i)}\|_F \ge \sigma_{\min}^* \|\bm T_{l}^{(i)} \bm\Sigma_{L}^{(i)} -\bm\Sigma_{L}^{(i)}\bm T_{l}^{(i)}\|_F. $
    Using this and \eqref{eq9:lem two}, we obtain \eqref{eq4:lem two}. 
\end{proof}

\subsubsection{Proof of the Error Bound}\label{subsubsec:pf eb}

Armed with the above results, we now proceed to prove the error bound stated in \Cref{thm:eb}. This is achieved in three steps. First, we show that, in a neighborhood of the critical point set, the singular values and corresponding singular vectors of the weight matrices can be controlled by $\|\nabla G(\bm W)\|_F$ (see Propositions~\ref{prop:singular value 2} and \ref{prop:singular control}). Second, we establish the error bound for Problem \eqref{eq:G} (see \Cref{thm:eb G}). To this end, we construct an intermediate point $\hat{\bm W}$ using left and right singular matrices (or suitable orthogonal matrices derived from them) of $\bm W$ (see \eqref{eq0:thm eb} in the proof of \Cref{thm:eb G}) and show that both $\mathrm{dist}(\hat{\bm W}, \mathcal{W}_G)$ and $\|\bm W - \hat{\bm W}\|_F$ are upper bounded by $\|\nabla G(\bm W)\|_F$. Finally, using \Cref{lem:equi FG} and \Cref{thm:eb G}, we prove the desired error bound of Problem \eqref{eq:F} in \Cref{thm:eb}.

We next bound the singular values and the associated singular vectors of the weight matrices in the following two propositions.

\begin{prop}\label{prop:singular value 2}
Let $\bm \sigma^* \in \mathcal{A}_{\rm sort}\setminus\{\bm 0 \}$ be arbitrary. The following statements hold: \\
(i) Suppose that $L=2$ and \eqref{eq:AS L=2} holds. Then for all $\bm W$ satisfying 
\begin{subequations}\label{eq:dist grad 1}
\begin{align}
    & \mathrm{dist}(\bm W, \mathcal{W}_{\bm \sigma^*}) \le \min\left\{ \frac{\delta_{\sigma}}{{6}},  \frac{\sqrt{\lambda}}{\sqrt{3(\sqrt{\lambda}+y_1)}}\left(
    \min\left\{ \min_{i \in [s_{p_Y}]}\left| \sqrt{\lambda} - y_i \right|, \sqrt{\lambda} \right\} \right)^{{1}/{2}}\right\}, \label{eq:dist 7}\\
    & \|\nabla G(\bm W)\|_F \le \frac{ \sqrt{2\lambda} \min\left\{ \min_{i \in [s_{p_Y}]}\left| \sqrt{\lambda} - y_i \right|, \sqrt{\lambda} \right\} \sigma^*_{\min} }{12c_2} , \label{eq:grad 2}
\end{align}
\end{subequations}
it holds for $l=1,2$ that
\begin{align}\label{eq:prop singular}
\sigma_i(\bm W_l) \le c_3   \|\nabla G(\bm W)\|_F,\  \forall i =  r_{\sigma}+1,\dots,\min\{d_l,d_{l-1}\},  
\end{align}
where
$ c_3 := {6c_2(y_1+\sqrt{\lambda})}/({\lambda \min\{ \min_{i \in [s_{p_Y}]}| \sqrt{\lambda} - y_i |, \sqrt{\lambda} \}}).$\\
(ii) Suppose that $L \ge 3$ and \eqref{eq:AS y} holds. Then for all $\bm W$ satisfying
\begin{align}\label{eq:dist 5}
  \mathrm{dist}(\bm W, \mathcal{W}_{\bm \sigma^*}) \le \min\left\{ \frac{\delta_{\sigma}}{{6}}, \left( \frac{\sqrt{\lambda}}{2y_1} \right)^{1/(L-2)} \right\}, 
\end{align}  
it holds for all $l \in [L]$ that  
\begin{align}\label{eq:singular 2}
\sigma_i(\bm W_l) \le c_3\|\nabla G(\bm W)\|_F,\ \forall i =  r_{\sigma}+1,\dots,\min\{d_l,d_{l-1}\}, 
\end{align}
where  
\begin{align}\label{eq:c3}  
c_3 := \frac{2 }{\lambda}\left( c_1 +  \left( \frac{3}{2}\sigma_{\max}^* \right)^{L}\frac{3(L-1)y_1}{ 2\delta_{\sigma}\sqrt{\lambda} \sigma^{*}_{\min}} \right).
\end{align}
\end{prop}
\begin{proof}  Let $\bm W^* = (\bm W_1^*,\dots,\bm W_L^*) \in \mathcal{W}_{\bm \sigma^*}$ be such that $\mathrm{dist}(\bm W, \mathcal{W}_{\bm \sigma^*}) = \|\bm W - \bm W^*\|_F$. From \eqref{eq:dist 7}, we have $\mathrm{dist}(\bm W,  \mathcal{W}_{\bm \sigma^*}) \le \delta_{\sigma}/{6} < \sigma_{\min}^*/2$ in case (i). From \eqref{eq:dist 5}, the same bound holds in case (ii). This, together with \Cref{lem:pre} and \Cref{coro1}, implies that \eqref{eq:sigmalowerbound}-\eqref{eq:wl+1} and \eqref{eq:coro1} hold. According to $\mathrm{dist}(\bm W, \mathcal{W}_{\bm \sigma^*}) \le \delta_{\sigma}/{6}$, \Cref{prop:Hlapproximate}, and \Cref{lem:twoinequality}, there exist matrices $\bm T_{l}^{(i)} \in \mathcal{O}^{g_i}$ for all $i \in [p]$ and $\bm T_{l}^{(p+1)} \in \mathcal{O}^{d_{l-1} - r_{\sigma}}$  such that \eqref{eq:Hlapproximate}, \eqref{eq:twoinequality2}, and \eqref{eq:twoinequality1} hold for all $l=2,\dots,L$, where $\bm A_i$ and $\bm B_i$ for each $i \in [p+1]$ are respectively defined in \eqref{eq:A} and \eqref{eq:B}. Recall that \eqref{eq:SVD Wl} denotes an SVD of $\bm W_l$ and let $\bm H_l := \bm U_{l-1}^T\bm V_l$ and $\bm T_l := \mathrm{BlkD}(\bm T_{l}^{(1)},\dots,\bm T_{l}^{(p)},\bm T_{l}^{(p+1)})$ for each $l=2,\dots,L$. This, together with \eqref{eq:Hlapproximate}, implies 
\begin{align}\label{eq0:prop singular 2}
\left\|\bm H_l - \bm T_l \right\|_F \le \frac{9\sigma_{\max}^{*}}{ 4\delta_{\sigma}\lambda\sigma^{*}_{\min}}\|\nabla G(\bm W)\|_F,\ l=2,\dots,L. 
\end{align}

(i) Suppose that $L=2$ and \eqref{eq:AS L=2} holds. For ease of exposition, define $\bm \Psi := \bm U_2^T\bm Y\bm V_1 \in \R^{d_2\times d_0}$ and we partition $\bm \Psi$ as  
\begin{align}\label{eq2:prop singular}
\bm \Psi = 
\begin{bmatrix}
\bm \Psi^{(1,1)} & \cdots & \bm \Psi^{(1,p+1)} \\
\vdots & \ddots & \vdots \\
\bm \Psi^{(p+1,1)} & \cdots & \bm \Psi^{(p+1,p+1)}
\end{bmatrix},
\end{align}
where $\bm \Psi^{(i,j)} \in \R^{g_i \times g_j} $ for each $i,j \in [p]$, $\bm \Psi^{(p+1, j)} \in \R^{(d_2 - r_{\sigma})\times g_j},\ \bm \Psi^{(i, p+1)} \in \R^{g_i \times (d_0 - r_{\sigma})}$ for each $i, j \in [p]$, and $\bm \Psi^{(p+1,p+1)} \in \R^{(d_2-r_\sigma)\times(d_0 - r_\sigma)}$. 
Using $L=2$, \eqref{eq:A}, \eqref{eq:B}, \eqref{eq:twoinequality2} and \eqref{eq:twoinequality1}, we have 
\begin{align}
    &c_2\|\nabla G(\bm W)\|_F \ge \left\|(\bm \Sigma_{1}^{(p+1)}\bm \Sigma_{1}^{(p+1)^T})\bm \Sigma_{1}^{(p+1)} +\lambda \bm
    \Sigma_{1}^{(p+1)} - \sqrt{\lambda} \bm T_{2}^{(p+1)} \bm \Sigma_{2}^{(p+1)^T} \bm \Psi^{(p+1,p+1)}\right\|_F,\label{eq7:prop singular}\\
    &c_2\|\nabla G(\bm W)\|_F \ge \left\|(\bm \Sigma_{2}^{(p+1)}\bm \Sigma_{2}^{(p+1)^T})\bm \Sigma_{2}^{(p+1)} +\lambda \bm\Sigma_{2}^{(p+1)} - \sqrt{\lambda}\bm\Psi^{(p+1,p+1)} \bm \Sigma_1^{(p+1)^T}\bm T_{2}^{(p+1)}\right\|_F.\label{eq8:prop singular}
\end{align}

Using \eqref{eq:twoinequality2} again, we have 
\begin{align}
\notag
    c_2^2 \|\nabla G(\bm W)\|_F^2 \ge \left\|(\bm \Sigma_1\bm \Sigma_1^T)\bm \Sigma_1 + \lambda \bm \Sigma_1 - \sqrt{\lambda}\mathrm{BlkD}\left(\bm A_1, \dots, \bm A_{p}, \bm A_{p+1}\right)\bm \Psi \right\|_F^2.
\end{align}
    By dropping the diagonal blocks and the $(p+1)$-th row blocks of the matrix on above the right-hand side, we have 
\begin{equation}\label{eq4:prop singular}
    c_2^2 \|\nabla G(\bm W)\|_F^2 \ge \lambda\sum_{i=1}^p \sum_{j=1,j\neq i}^{p+1}\|\bm A_i \bm\Psi^{(i,j)}\|_F^2 
    \overset{(\ref{eq:sigmalowerbound}, \ref{eq:A})}{\ge} \frac{\lambda\sigma^{*2}_{\min}}{4}  \sum_{i=1}^p \sum_{j=1,j\neq i}^{p+1}\| \bm\Psi^{(i,j)}\|_F^2.
\end{equation}
Applying the same argument to \eqref{eq:twoinequality1} yields
$ 
    c_2^2 \|\nabla G(\bm W)\|_F^2 \ge  \frac{\lambda\sigma^{*2}_{\min}}{4}  \sum_{j=1}^p \sum_{i=1,i\neq j}^{p+1}\| \bm\Psi^{(i,j)}\|_F^2.   
$
This, together with \eqref{eq4:prop singular}, yields
\begin{equation}\label{eq5:prop singular}
    \frac{8c_2^2}{\lambda\sigma_{\min}^{*2}}\|\nabla G(\bm W)\|_F^2 \ge \sum_{i=1}^{p+1} \sum_{j=1,j\neq i}^{p+1}\| \bm\Psi^{(i,j)}\|_F^2.
\end{equation}
Using \eqref{eq:AS L=2}, we define
$\delta := \min\left\{ \min_{i \in [s_{p_Y}]}\left| \sqrt{\lambda} - y_i \right|, \sqrt{\lambda} \right\} > 0.$ Substituting this into \eqref{eq:grad 2} yields
\begin{align*}
\frac{\delta}{3} \ge \frac{2\sqrt{2}c_2}{\sqrt{\lambda} \sigma_{\min}^*} \|\nabla G(\bm W)\|_F \overset{(\ref{eq2:prop singular}, \ref{eq5:prop singular})}{\ge} \left\| \bm \Psi - \mathrm{BlkD}\left(\bm \Psi^{(1,1)},\dots, \bm \Psi^{(p+1,p+1)}\right) \right\|_F. 
\end{align*}
This, together with $\bm \Psi = \bm U_2^T\bm Y\bm V_1$ and Weyl's inequality, yields for all $i \in [d_{\min}]$, 
\begin{align*}
\frac{\delta}{3} & \ge\left|y_i - \sigma_i\left(\mathrm{BlkD} (\bm \Psi^{(1,1)},\dots, \bm \Psi^{(p+1,p+1)})\right) \right| \ge \left|y_i - \sqrt{\lambda}\right|  - \left| \sigma_i\left(\mathrm{BlkD} (\bm \Psi^{(1,1)},\dots, \bm \Psi^{(p+1,p+1)})\right) - \sqrt{\lambda} \right| \\
& \ge \delta - \left| \sigma_i\left(\mathrm{BlkD} (\bm \Psi^{(1,1)},\dots, \bm \Psi^{(p+1,p+1)})\right) - \sqrt{\lambda} \right|, 
\end{align*}  
which implies $\left| \sigma_i( \bm \Psi^{(p+1,p+1)} ) - \sqrt{\lambda} \right| \ge {2\delta}/{3}$ for all
$i \in [d_{\min}- r_{\sigma}]$. Accordingly, for all $i \in [d_{\min} - r_{\sigma}]$, we obtain
\begin{align}\label{eq6:prop singular}
\left| \sigma_i^2( \bm \Psi^{(p+1,p+1)}) - \lambda \right| & = \left| \sigma_i( \bm \Psi^{(p+1,p+1)}) - \sqrt{\lambda} \right|  \left| \sigma_i( \bm \Psi^{(p+1,p+1)}) + \sqrt{\lambda} \right| \ge \frac{2\delta\sqrt{\lambda}}{3}.  
\end{align}

According to $\bm \Psi = \bm U_2^T\bm Y\bm V_1$ and \eqref{eq2:prop singular}, we have  
\begin{align}\label{eq9:prop singular}
    \| \bm{\Psi}^{(p+1,p+1)} \| \leq \| \bm{Y} \| = y_1.
\end{align}
Using this and \eqref{eq7:prop singular}, we obtain 
\begin{align*}
c_2y_1\|\nabla G(\bm W)\|_F &\ge \left\|\left((\bm \Sigma_{1}^{(p+1)}\bm \Sigma_{1}^{(p+1)^T})\bm \Sigma_{1}^{(p+1)} +\lambda \bm
    \Sigma_{1}^{(p+1)} - \sqrt{\lambda} \bm T_{2}^{(p+1)} \bm \Sigma_{2}^{(p+1)^T} \bm \Psi^{(p+1,p+1)}\right) \bm{\Psi}^{(p+1,p+1)^T }\right\|_F \\
& \ge  \left\|\lambda \bm \Sigma_{1}^{(p+1)} \bm\Psi^{(p+1,p+1)^T} - \sqrt{\lambda} \bm T_{2}^{(p+1)} \bm \Sigma_{2}^{(p+1)^T} \bm\Psi^{(p+1,p+1)}\bm\Psi^{(p+1,p+1)^T} \right\|_F \\
    & \quad\ - \left\|\bm \Sigma_{1}^{(p+1)}\bm \Sigma_{1}^{(p+1)^T}\bm \Sigma_{1}^{(p+1)}\bm\Psi^{(p+1,p+1)^T} \right\|_F \\
    & \ge \left\|\lambda\bm T_{2}^{(p+1)^T} \bm \Sigma_{1}^{(p+1)} \bm\Psi^{(p+1,p+1)^T} - \sqrt{\lambda}  \bm \Sigma_{2}^{(p+1)^T} \bm\Psi^{(p+1,p+1)}{\bm\Psi^{(p+1,p+1)}}^T\right\|_F - y_1 \|\bm \Sigma_{1}^{(p+1)}\|_F^3, 
\end{align*}
where the last inequality follows from $\bm T_{2}^{(p+1)} \in \mathcal{O}^{d_2-r_{\sigma}}$ and \eqref{eq9:prop singular}. Using \eqref{eq8:prop singular}, we have
\begin{align*}
    c_2\sqrt{\lambda}\|\nabla G(\bm W)\|_F &\ge \sqrt{\lambda} \left\|\lambda \bm\Sigma_{2}^{(p+1)} - \sqrt{\lambda}\bm \Psi^{(p+1,p+1)} \bm \Sigma_{1}^{(p+1)^T}\bm T_{2}^{(p+1)} \right\|_F - \sqrt{\lambda} \left\|\bm \Sigma_{2}^{(p+1)}\bm \Sigma_{2}^{(p+1)^T} \bm \Sigma_{2}^{(p+1)} \right\|_F \\
    &\ge \left\| \lambda^{\frac{3}{2}}\bm\Sigma_{2}^{(p+1)^T} - \lambda \bm T_{2}^{(p+1)^T} \bm \Sigma_{1}^{(p+1)}\bm \Psi^{(p+1,p+1)^T} \right\|_F - \sqrt{\lambda} \left\|\bm \Sigma_{2}^{(p+1)} \right\|_F^3. 
\end{align*} 
Summing up the above two inequalities, together with the triangle inequality, yields
\begin{align}\label{eq10:prop singular}
    c_2(y_1+\sqrt{\lambda})\|\nabla G(\bm W)\|_F & \ge \left\|\lambda^{\frac{3}{2}}\bm\Sigma_{2}^{(p+1)^T} - \sqrt{\lambda}  \bm \Sigma_{2}^{(p+1)^T} \bm\Psi^{(p+1,p+1)}\bm\Psi^{(p+1,p+1)^T}\right\|_F \notag \\
    &\quad - \sqrt{\lambda} \left\|\bm \Sigma_{2}^{(p+1)}\right\|_F^3-y_1\left\|\bm \Sigma_{1}^{(p+1)}\right\|_F^3. 
\end{align}
Repeating the argument from \eqref{eq9:prop singular} to \eqref{eq10:prop singular}, but this time multiplying \eqref{eq7:prop singular} by $\sqrt{\lambda}$ and \eqref{eq8:prop singular} by $y_1$, we obtain
\begin{align*}
    c_2(y_1+\sqrt{\lambda})\|\nabla G(\bm W)\|_F  \ge & \left\|\lambda^{\frac{3}{2}}\bm\Sigma_{1}^{(p+1)^T} - \sqrt{\lambda}\bm\Psi^{(p+1,p+1)^T}\bm\Psi^{(p+1,p+1)}  \bm \Sigma_{1}^{(p+1)^T} \right\|_F - \sqrt{\lambda} \left\|\bm \Sigma_{1}^{(p+1)}\right\|_F^3-y_1 \left\|\bm \Sigma_{2}^{(p+1)}\right\|_F^3. 
\end{align*}
Summing up the above two inequalities yields 
\begin{align*}
   & 2c_2(y_1+\sqrt{\lambda})\left\|\nabla G(\bm W)\right\|_F \\
   \ge &\ \sqrt{\lambda}  \left\|(\bm\Psi^{(p+1,p+1)}\bm\Psi^{(p+1,p+1)^T} - \lambda \bm I )\bm \Sigma_{2}^{(p+1)}\right\|_F +\sqrt{\lambda}\left\|(\bm\Psi^{(p+1,p+1)^T}\bm\Psi^{(p+1,p+1)} - \lambda \bm I)\bm \Sigma_{1}^{(p+1)T}\right\|_F \\
   &\quad -(\sqrt{\lambda}+y_1)\left(\left\|\bm \Sigma_{2}^{(p+1)}\right\|_F^3 + \left\|\bm \Sigma_{1}^{(p+1)}\right\|_F^3\right)\\
   \ge &\ \sqrt{\lambda}\sigma_{\min}\left(\bm\Psi^{(p+1,p+1)}\bm\Psi^{(p+1,p+1)^T} - \lambda \bm I\right)\left\|\bm \Sigma_{2}^{(p+1)}\right\|_F+\sqrt{\lambda}\sigma_{\min}\left(\bm\Psi^{(p+1,p+1)^T}\bm\Psi^{(p+1,p+1)} - \lambda \bm I\right)\left\|\bm \Sigma_{1}^{(p+1)}\right\|_F\\
   &\quad  - (\sqrt{\lambda}+y_1)\left(\left\|\bm \Sigma_{2}^{(p+1)}\right\|_F^3 + \left\|\bm \Sigma_{1}^{(p+1)}\right\|_F^3\right)\\
   \overset{\eqref{eq6:prop singular}}{\ge} &\ \left(\frac{2\lambda\delta }{3} -(\sqrt{\lambda}+y_1)\left\|\bm \Sigma_{1}^{(p+1)}\right\|_F^2\right)\left\|\bm \Sigma_{1}^{(p+1)}\right\|_F +  \left(\frac{2\lambda\delta}{3} -(\sqrt{\lambda}+y_1)\left\|\bm \Sigma_{2}^{(p+1)}\right\|_F^2\right)\left\|\bm \Sigma_{2}^{(p+1)}\right\|_F \\
    \geq &\ \frac{\lambda\delta}{3}\left(\left\|\bm \Sigma_{1}^{(p+1)}\right\|_F+\left\|\bm \Sigma_{2}^{(p+1)}\right\|_F\right),
\end{align*}
where the last inequality follows from the definition of $\delta$, \eqref{eq:dist 7}, and $\| \bm{\Sigma}_1^{(p+1)}\|_F = \| \bm{\Sigma}_1^{(p+1)} - \bm 0\|_F \le \| \bm{\Sigma}_1 - \bm\Sigma_1^*\|_F\le \|\bm W_1-\bm W_1^*\|_F \le  \mathrm{dist}(\bm{W}, \mathcal{W}_{\bm{\sigma}^*})$ and $\| \bm{\Sigma}_2^{(p+1)} \|_F \le  \mathrm{dist}(\bm{W}, \mathcal{W}_{\bm{\sigma}^*})$ due to Mirsky's inequality (see \Cref{lem:mirsky} in Appendix~\ref{app:sec auxi}). This directly implies \eqref{eq:prop singular}. 

(ii)  Suppose that $L \ge 3$ and \eqref{eq:AS y} holds. 
For ease of exposition, let $\bm \Psi := \bm U_L^T\bm Y\bm V_1 \in \R^{d_L\times d_0}$ and we still partition $\bm \Psi$ into the form as in \eqref{eq2:prop singular}. Let $l \in [L]$ be such that $l \in \argmax\{\sigma_{r_{\sigma}+1}(\bm W_k):{k \in [L]}\}$. Define 
\begin{align*}
     \bm \Delta&:=  \left(\prod_{k=l+1}^{L}\bm H_{k}\bm \Sigma_k^T \right) \bm \Psi \left(\prod_{k=1}^{l-1}\bm \Sigma_k^T \bm H_{k+1} \right)-\left(\prod_{k=l+1}^{L}\bm T_{k}\bm \Sigma_k^T\right) \bm \Psi \left(\prod_{k =1}^{l-1}\bm \Sigma_k^T \bm T_{k+1}\right)\\
     &~=\sum_{j=l}^{L-1} \left( \prod_{k=l+1}^{j}\bm T_{k}\bm \Sigma_k^T \right)(\bm H_{j+1}-\bm T_{j+1})\bm \Sigma_{j+1}^T\left(\prod_{k=j+2}^{L}\bm H_{k}\bm \Sigma_k^T \right) \bm \Psi \left(\prod_{k=1}^{l-1}\bm \Sigma_k^T \bm H_{k+1} \right)\\
     &\quad +\sum_{j=2}^l \left( \prod_{k=l+1}^{L}\bm T_{k}\bm \Sigma_k^T \right) \bm \Psi\left(\prod_{k=1}^{j-2}\bm \Sigma_k^T \bm T_{k+1} \right)\bm \Sigma_{j-1}^T(\bm H_j - \bm T_j)\left(\prod_{k=j}^{l-1}\bm \Sigma_k^T \bm H_{k+1}\right). 
 \end{align*}
Here, we remark that in the above summations, if the lower index exceeds the upper index, the corresponding product is interpreted as the identity matrix. Then substituting \eqref{eq:SVD Wl} into \eqref{eq:coro1} yields 
\begin{align*}
c_1\|\nabla G(\bm W)\|_F & \ge \left\|(\bm\Sigma_l\bm\Sigma_l^T)^{L-1}\bm\Sigma_l-\sqrt{\lambda}\left(\prod_{k=l+1}^{L}\bm H_{k}\bm \Sigma_k^T \right)\bm \Psi \left(\prod_{k =1}^{l-1}\bm \Sigma_k^T\bm H_{k+1}\right) + \lambda\bm\Sigma_l \right\|_F \\
& \ge \left\| (\bm\Sigma_l\bm\Sigma_l^T)^{L-1}\bm\Sigma_l-\sqrt{\lambda}\left(\prod_{k=l+1}^{L}\bm T_{k}\bm \Sigma_k^T\right) \bm \Psi \left(\prod_{k =1}^{l-1}\bm \Sigma_k^T \bm T_{k+1}\right) + \lambda\bm\Sigma_l\right\|_F - \sqrt{\lambda}\|\bm \Delta\|_F, 
\end{align*} 
where the second inequality uses the triangle inequality. 
Using \eqref{eq:sigmalowerbound} and \eqref{eq0:prop singular 2}, we obtain 
$\|\bm\Delta\|_F \le \left( \frac{3}{2}\sigma_{\max}^* \right)^{L}\frac{3(L-1)y_1}{ 2\delta_{\sigma}\lambda\sigma^{*}_{\min}}  \|\nabla G(\bm W)\|_F.$ This, together with the above inequality, yields 
\begin{align}
 \left\| (\bm\Sigma_l\bm\Sigma_l^T)^{L-1}\bm\Sigma_l-\sqrt{\lambda}\left(\prod_{k=l+1}^{L}\bm T_{k}\bm \Sigma_k^T\right) \bm \Psi \left(\prod_{k =1}^{l-1}\bm \Sigma_k^T \bm T_{k+1}\right) + \lambda\bm\Sigma_l \right\|_F \le \eta_2 \|\nabla G(\bm W)\|_F,
\label{eq:eta_2def}
\end{align}
where $\eta_2 := c_1 +  \left( \frac{3}{2}\sigma_{\max}^* \right)^{L}\frac{3(L-1)y_1}{ 2\delta_{\sigma}\sqrt{\lambda} \sigma^{*}_{\min}}$. Using this inequality, the block structures of $\bm \Sigma_l$ and $\bm T_l$, 
$l \in \argmax\{\sigma_{r_{\sigma}+1}(\bm W_k):{k \in [L]}\}$, and the fact that the Frobenius norm of a matrix is larger than or equal to its spectral norm, we obtain 
\begin{align}\label{eq1:prop singular 2}
\eta_2 \|\nabla G(\bm W)\|_F \ge &\ \left\| \left(\bm\Sigma_{l}^{(p+1)}\bm \Sigma_{l}^{(p+1)^T}\right)^{L-1}\bm \Sigma_{l}^{(p+1)} \right. \notag 
 \\
& \left. -\sqrt{\lambda}\left(\prod_{k=l+1}^{L} \bm T_{k}^{(p+1)} \bm \Sigma_{k}^{(p+1)^T}\right) 
\bm \Psi_{r_{\sigma}+1:d_L,r_{\sigma}+1:d_0}\left(\prod_{k=1}^{l-1} \bm \Sigma_{k}^{(p+1)^T} \bm T_{k+1}^{(p+1)}\right) + \lambda\bm \Sigma_{l}^{(p+1)}\right\| \notag \\
\ge &\  \left\| \left(\bm\Sigma_{l}^{(p+1)}\bm \Sigma_{l}^{(p+1)^T}\right)^{L-1}\bm \Sigma_{l}^{(p+1)} + \lambda\bm \Sigma_{l}^{(p+1)} \right\| - \sqrt{\lambda}\sigma_{r_{ \sigma}+1}^{L-1}(\bm W_l)\left\|\bm \Psi_{r_{\sigma}+1:d_L,r_{\sigma}+1:d_0}\right\| \notag \\
\ge &\ \lambda\sigma_{r_{ \sigma}+1}(\bm W_l)-\sqrt{\lambda}y_1\sigma_{r_{ \sigma}+1}^{L-1}(\bm W_l) 
\ge   \sqrt{\lambda}\left(\sqrt{\lambda}-y_{1}\sigma_{r_{\sigma}+1}^{L-2}(\bm W_l)\right) \sigma_{r_{\sigma}+1}(\bm W_l), 
\end{align}
where \(\bm \Psi_{r_{\sigma}+1:d_L, r_{\sigma}+1:d_0}\) denotes the submatrix of \(\bm \Psi\) with the rows indexed from \(r_{\sigma}+1\) to \(d_L\) and the columns indexed from \(r_{\sigma}+1\) to \(d_0\), and the third inequality is due to $\|\bm \Psi_{r_{\sigma}+1:d_L, r_{\sigma}+1:d_0}\| \le \|\bm{Y}\| = y_1$. Using Weyl's inequality and \eqref{eq:dist 5}, we have
$\sigma_{r_{\sigma}+1}(\bm W_l) \le \mathrm{dist}\left(\bm W, \mathcal{W}_{\bm \sigma^*} \right) \le \left( \frac{\sqrt{\lambda}}{2y_1} \right)^{1/(L-2)}.$
This, together with \eqref{eq1:prop singular 2}, directly yields 
$
\sigma_{r_{\sigma}+1}(\bm W_l) \le  {2\eta_2}\|\nabla G(\bm W)\|_F/{\lambda}$, which implies \eqref{eq:singular 2}.
\end{proof}

The above proposition provides bounds for the $\min\{d_l,d_{l-1}\} - r_{\sigma}$ smallest singular values of $\bm W_l$. It remains to bound the top $r_{\sigma}$ singular values and all singular vectors. For convenience, we define an auxiliary function $\varphi : \mathbb{R}_{++} \;\to\; \mathbb{R}$:
\begin{align}\label{eq:phi}
\varphi(x) := \frac{x^{2L-1} + \lambda x}{\sqrt{\lambda}\, x^{L-1}}.
\end{align} 
The following proposition establishes two key results. First, by the Davis-Kahan theorem, we show that the left singular matrix $\bm{U}_L$ of $\bm{W}_L$ (resp., the right singular matrix $\bm{V}_1$ of $\bm{W}_1$) can be approximated by a product of block orthogonal matrices. Second, we employ Weyl's inequality to bound the gap between the singular values of $\bm{W}$ and $\bm{\sigma}^*$.

\begin{prop}\label{prop:singular control}
Let $\bm \sigma^* \in \mathcal{A}_{\rm sort}\setminus\{\bm 0 \}$ be arbitrary and $\bm{W}$ be arbitrary such that 
\begin{align}
    &\mathrm{dist}(\bm{W}, \mathcal{W}_{\bm{\sigma}^*}) \le \min\left\{\delta_1, \delta_2,\frac{\delta_{\sigma}}{6}\right\}\quad \text{and}\quad \|\nabla G(\bm{W})\|_F \leq \frac{\delta_y\sqrt{\lambda} (\sigma^{*}_{\min})^{L-1}}{3 \cdot 2^{L-1} \sqrt{\eta_3^2 + \eta_4^2}}, \label{eq:gradient_condition}
\end{align}      
where
\begin{align}
        \delta_1 &   := \frac{\delta_y}{3L\left(\frac{4\sigma_{\max}^*}{3}\right)^{L-1}\frac{1}{\sqrt{\lambda}}+3(L-2)\sqrt{\lambda}\left(\frac{2\sigma_{\min}^*}{3}\right)^{1-L}}, \label{eq:delta1}\\
        \delta_2 & := \frac{\min_{i\in [r_{\sigma}]}|\varphi'(\sigma_i^*)|}{\frac{2L(L-1)}{\sqrt{\lambda}}\left(\frac{4\sigma_{\max}^*}{3}\right)^{L-2}+2\sqrt{\lambda}(L-2)(L-1)\left(\frac{2\sigma_{\min}^*}{3}\right)^{-L}},\label{eq:delta2}\\
        \ \eta_3& := c_2 +  c_3\sqrt{d_{\max}} \left(\left( \frac{3\sigma_{\max}^*}{2} \right)^{2(L-1)}\!\!\!\! + \sqrt{\lambda}y_1 \left( \frac{3\sigma_{\max}^*}{2} \right)^{L-2} \!\!\!\!+ \lambda \right),\label{eq:eta3}\\ 
        \eta_4 & := \eta_3 
        + \frac{p \eta_1 (2L - 1)L}{\sigma^*_{\min}} 
            \left( \frac{3 \sigma^*_{\max}}{2} \right)^{2L - 2}\! \!\!\! 
        + \frac{\lambda p \eta_1 L}{\sigma^*_{\min}}. \label{eq:eta4}
 \end{align}        
Suppose in addition Assumptions~\ref{AS:1} and  \ref{AS:2} hold,  \eqref{eq:dist grad 1} holds for $L=2$, and \eqref{eq:dist 5} holds for $L \ge 3$. \\
(i) There exist orthogonal matrices $\hat{\bm U}_{L}^{(i)} \in \mathcal{O}^{h_i}$ for $i \in [p_Y]$, $\hat{\bm U}_{L}^{(p_Y+1)} \in \mathcal{O}^{d_L-r_Y}$, $\hat{\bm V}_{1}^{(p_Y+1)} \in \mathcal{O}^{d_0-r_Y}$, $\bm T_{l}^{(i)} \in \mathcal{O}^{g_i}$ for $i \in [p]$, $\bm P \in \mathcal{O}^{d_L - r_{\sigma}}$, $\bm Q \in \mathcal{O}^{d_0 - r_{\sigma}}$, and a permutation matrix $\bm \Pi \in \mathcal{P}^{d_{\min}}$ with $(\bm \sigma^*, \bm \Pi^T) \in \mathcal{B}$ such that 
\begin{align}\label{eq:singular vectors}
  \|\tilde{\bm U}_L - \bm U_L \|_F \le c_4\|\nabla G(\bm W)\|_F\quad\text{ and } \quad \|\tilde{\bm V}_1 - \bm V_1\|_F \le c_4\|\nabla G(\bm W)\|_F,
\end{align}
where 
\begin{align} 
&\eta_5 := \frac{2^{L+1}(6y_1+\delta_y)(p_Y+1)\sqrt{\eta_3^2+\eta_4^2}}{3\sqrt{\lambda}y_{s_{p_Y}}\delta_y \sigma^{* L-1}_{\min}},\qquad
    c_4 := \eta_5 + \frac{1}{y_{s_{p_Y}}}\left(\frac{2^{L}\sqrt{\eta_3^2+\eta_4^2}}{\sqrt{\lambda} \sigma^{*L-1}_{\min}}+2y_1\eta_5\right) \label{eq:upper_v}\\
& \tilde{\bm U}_L := \mathrm{BlkD}\left(\hat{\bm U}_{L}^{(1)},\dots,\hat{\bm U}_{L}^{(p_Y)},\hat{\bm U}_{L}^{(p_Y+1)} \right)\mathrm{BlkD}\left(\bm \Pi, \bm I_{d_L-d_{\min}}\right)\mathrm{BlkD}\left(\bm I_{r_{\sigma}}, \bm P^T \right), \label{eq:UL}\\
& \tilde{\bm V}_1 := \mathrm{BlkD}\left(\hat{\bm U}_{L}^{(1)},\dots,\hat{\bm U}_{L}^{(p_Y)},\hat{\bm V}_{L}^{(p_Y+1)} \right)\mathrm{BlkD}\left(\bm \Pi, \bm I_{d_0-d_{\min}}\right)\mathrm{BlkD}\left(\bm I_{r_{\sigma}}, \bm Q^T \right) \hat{\bm T}{^T}, \label{eq:V1}\\
&\hat{\bm T} := \mathrm{BlkD}\left(\prod_{l=2}^L\bm T_{l}^{(1)},\cdots,\prod_{l=2}^L\bm T_{l}^{(p)},\bm I_{d_0-r_{\sigma}}\right).\label{eq:T}
\end{align}
(ii) For each $l \in [L]$, there exists a constant $c_5 > 0$  such that 
\begin{align}\label{eq:singular 1}
| \sigma_i\left( \bm W_l) - \sigma^*_{i} \right| \le c_5\|\nabla G(\bm W)\|_F,\ \forall i \in [r_{\sigma}],
\end{align}
where
\begin{align}\label{eq:c5}
    c_5 := \frac{2^{L}\sqrt{ \eta_3^2 + \eta_4^2 }}{\sqrt{\lambda}\sigma_{\min}^{* {L-1}}\min_{i \in [r_{\sigma}] } \left|\varphi^\prime(\sigma_i^*)\right| }  + \frac{3\sqrt{2}(L-1)\sigma_{\max}^*}{4\lambda\sigma_{\min}^*}. 
\end{align}
\end{prop}
\begin{proof} 
According to \eqref{eq:dist 7} and \eqref{eq:dist 5}, we have $\mathrm{dist}(\bm W, \mathcal{W}_{\bm{\sigma}^*}) \le \delta_{\sigma}/{6} < \sigma_{\min}^*/2$. This, together with \Cref{lem:pre} and \Cref{coro1}, implies that \eqref{eq:sigmalowerbound}-\eqref{eq:wl+1} and \eqref{eq:coro1} hold.  According to \Cref{prop:Hlapproximate} and \Cref{lem:twoinequality} with $\mathrm{dist}(\bm W, \mathcal{W}_{\bm{\sigma}^*}) \le \delta_{\sigma}/{6}$, there exist matrices $\bm T_{l}^{(i)} \in \mathcal{O}^{g_i}$ for each $i \in [p]$ and $\bm T_{l}^{(p+1)} \in \mathcal{O}^{d_{l-1} - r_{\sigma}}$ such that \eqref{eq:twoinequality2} and \eqref{eq:twoinequality1} hold, where $\bm A_{i}$ and $\bm B_i$ for $i \in [p+1]$ are defined in \eqref{eq:A} and \eqref{eq:B}. Based on \eqref{eq:dist grad 1} (resp., \eqref{eq:dist 5}) and \Cref{prop:singular value 2}, we have \eqref{eq:prop singular} (resp., \eqref{eq:singular 2}). Recall that \eqref{eq:SVD Wl} denotes an SVD of $\bm W_l$ for $l \in [L]$. Let 
\begin{align}\label{eq0:prop sing}
\bm \Psi := \bm U_L^T\bm Y\bm V_1,\qquad  \hat{\bm \Sigma}_l := \mathrm{BlkD}\left(\bm \Sigma_{l}^{(1)},\dots,\bm \Sigma_{l}^{(p)}, \bm 0_{(d_l-r_{\sigma})\times(d_{l-1} - r_{\sigma})} \right),\ \forall l \in [L].
\end{align}  
(i) We first prove the first inequality in \eqref{eq:singular vectors}. We claim that the following two inequalities hold:
\begin{align}
\sum_{i=1}^{r_{\sigma}} \left\| \varphi(\sigma_i(\bm W_L))\bm e_i - \bm \Psi \hat{\bm T} \bm e_i\right\|^2 &\le  \frac{4^{L-1}\eta_3^2}{\lambda(\sigma_{\min}^{*})^ {2L-2}} \|\nabla G(\bm W)\|_F^2,\label{eqi11:prop sing} \\
\sum_{i=1}^{r_{\sigma}} \left\| \varphi(\sigma_i(\bm W_1))\bm e_i^T - \bm e_i^T\bm \Psi \hat{\bm T} \right\|^2 
& \le  \frac{4^{L-1}\eta_4^2}{\lambda(\sigma_{\min}^*)^{{2L-2}}}   \|\nabla G(\bm W)\|_F^2.\label{eqi12:prop sing}
\end{align} 
To maintain the flow of the main proof, the proof of this claim is deferred to Appendix~\ref{sec:comlementary1}. Now, let $\bm \alpha :=  \left(\varphi(\sigma_1(\bm W_L)),\dots,\varphi(\sigma_{r_{\sigma}}(\bm W_L)) \right) \in \R^{r_{\sigma}}$ and  
$\bm Z := \mathrm{BlkD}\left(\mathrm{diag}(\bm \alpha), (\bm \Psi \hat{\bm T})_{r_{\sigma}+1:d_L,r_{\sigma}+1:d_0}\right)$. 
We have
\begin{align}\label{eq7:prop sing}
\left\|\bm \Psi \hat{\bm T} - \bm Z \right\|_F^2 & = \sum_{i=1}^{r_{\sigma}} \left\|( \bm \Psi \hat{\bm T}  - \bm Z)\bm e_i \right\|^2 + \sum_{i=r_{\sigma}+1}^{d_0} \|(\bm \Psi \hat{\bm T} - \bm Z )\bm e_i \|^2 \notag \\
& = \sum_{i=1}^{r_{\sigma}} \left\| \bm \Psi \hat{\bm T} \bm e_i - \varphi(\sigma_i(\bm W_L))\bm e_i \right\|^2 +  \left\| \left(\bm \Psi \hat{\bm T}\right)_{1:r_{\sigma},r_{\sigma}+1:d_0} \right\|_F^2 \notag\\
& \le \sum_{i=1}^{r_{\sigma}} \left\| \bm \Psi \hat{\bm T} \bm e_i - \varphi(\sigma_i(\bm W_L))\bm e_i \right\|^2 + \sum_{i=1}^{r_{\sigma}} \left\|  \bm e_i^T\bm \Psi \hat{\bm T}- \varphi(\sigma_i(\bm W_1))\bm e_i^T \right\|^2 \notag \\
& \overset{(\ref{eqi11:prop sing}, \ref{eqi12:prop sing})}{\le}  \frac{4^{L-1}\left( \eta_3^2 + \eta_4^2 \right)}{\lambda(\sigma_{\min}^*)^{{2L-2}}}  \|\nabla G(\bm W)\|_F^2,
\end{align}
where the second equality uses the structure of $\bm Z$ and the first inequality is due to
$
    \left\|  \bm e_i^T\bm \Psi \hat{\bm T}- \varphi(\sigma_i(\bm W_1))\bm e_i^T \right\|  \ge \| (\bm \Psi \hat{\bm T})_{i,r_{\sigma}+1:d_0} \|
$ for all $i \in [r_{\sigma}]$. 

Let 
$ 
 (\bm \Psi \hat{\bm T})_{r_{\sigma}+1:d_L,r_{\sigma}+1:d_0}  = \bm P \bm \Lambda \bm Q^T
$ 
be an SVD of $(\bm \Psi \hat{\bm T})_{r_{\sigma}+1:d_L,r_{\sigma}+1:d_0} $, where 
$\bm \Lambda \in \mathbb{R}^{(d_L - r_{\sigma}) \times (d_0 - r_{\sigma})}$ is diagonal with 
entries $\gamma_1 \ge \dots \ge \gamma_{d_{\min} - r_{\sigma}}$ being the top $d_{\min} - r_{\sigma}$ singular values.  To proceed, we define 
\begin{align}\label{eq:c}
\bm \beta := \left(\varphi(\sigma_1(\bm W_L)),\dots,\varphi(\sigma_{r_{\sigma}}(\bm W_L)),\gamma_1,\dots, \gamma_{d_{\min}- r_{\sigma}}\right) \in \R^{d_{\min}}.
\end{align}
Let $\bm \Pi \in \mathcal{P}^{d_{\min}}$ be a permutation matrix such that the entries $\bm \Pi\bm \beta$ are in a decreasing order and 
\begin{equation}\label{eq:foldef}
\begin{aligned}
& \bar{\bm U}_L := \bm U_L\mathrm{BlkD}\left(\bm I_{r_{\sigma}}, \bm P\right)\mathrm{BlkD}\left(\bm \Pi^T, \bm I_{d_L-d_{\min}}\right),\ \bar{\bm V}_1 := \bm V_1 \hat{\bm T} \mathrm{BlkD}\left(\bm I_{r_{\sigma}}, \bm Q\right)\mathrm{BlkD}\left(\bm \Pi^T,\bm I_{d_0-d_{\min}}\right),\\
&\bar{\bm Y} := \mathrm{BlkD}\left(\bm\Pi \mathrm{diag}\left(\bm \beta\right)\bm \Pi^T,\bm 0_{(d_L-d_{\min})\times (d_0 - d_{\min}) }\right).
\end{aligned}
\end{equation} 
Since both $\bm{Y}$ and $\bar{\bm{Y}}$ are diagonal matrices with elements in decreasing order, by using Mirsky's inequality (see \Cref{lem:mirsky} in Appendix~\ref{app:sec auxi}), we have
\begin{align}\label{eq8:prop sing}
\|\bm Y - \bar{\bm Y}\|_F \le \|\bar{\bm U}_L^T \bm Y \bar{\bm V}_1 - \bar{\bm Y}\|_F = \| \bm \Psi \hat{\bm T} - \bm Z\|_F\overset{\eqref{eq7:prop sing}}{\le} \frac{2^{L-1}}{\sqrt{\lambda}\sigma_{\min}^{* {L-1}}} \sqrt{ \eta_3^2 + \eta_4^2 } \|\nabla G(\bm W)\|_F,
\end{align}
where the equality follows from \eqref{eq0:prop sing}, \eqref{eq:c}, and \eqref{eq:foldef}. This, together with \eqref{eq:gradient_condition}, implies  
\begin{align}\label{eq9:prop sing}
\|\bm Y - \bar{\bm Y}\|_F \le \frac{\delta_y}{3},\qquad \|\bar{\bm Y}\| \le \|\bar{\bm Y} - \bm Y\|  + \|\bm Y\|  \le y_1 + \frac{\delta_y}{3}. 
\end{align}\smallskip

Using the triangle inequality, we obtain  
\begin{align}
    \|\bar{\bm U}_L^T \bm Y \bm Y^T \bar{\bm U}_L - \bar{\bm Y} \bar{\bm Y}^T\|_F &~~~ \le~~~~ \|(\bar{\bm U}_L^T \bm Y \bar{\bm V}_1 - \bar{\bm Y}) \bar{\bm V_1}^T \bm Y^T \bar{\bm U}_L\|_F + \|\bar{\bm Y} (\bar{\bm U}_L^T \bm Y \bar{\bm V}_1 - \bar{\bm Y})^T\|_F \notag\\
    & \overset{(\ref{eq8:prop sing}, \ref{eq9:prop sing})}{\le} \left( 2y_1 + \frac{\delta_y}{3} \right) \frac{2^{L-1}\sqrt{\eta_3^2+\eta_4^2}}{\sqrt{\lambda}\sigma^{* {L-1}}_{\min}} \|\nabla G(\bm W)\|_F.\label{eq:control theorem}
\end{align} 
Next, we write \(\bar{\bm{U}}_L\) and \(\bm{I}_{d_L}\) as follows:  
\begin{align*}
&\bar{\bm{U}}_L = \begin{bmatrix}
   (\bar{\bm{U}}_{L}^{(1)})^T & \cdots & (\bar{\bm{U}}_{L}^{(p_Y)})^T &( \bar{\bm{U}}_{L}^{(p_Y+1)})^T
\end{bmatrix}^{\!T},\qquad
\bm{I}_{d_L} = \begin{bmatrix}
   (\bm{E}^{(1)} )^T& \cdots & (\bm{E}^{(p_Y)})^T & (\bm{E}^{(p_Y+1)})^T
\end{bmatrix}^{\!T},\\
&\bm{E}^{(1)} = [\, \bm I_{h_1}\ \ \bm 0_{h_1 \times (d_L - s_1)} \,],\quad
\bm{E}^{(i)} = [\, \bm 0_{h_i \times s_{i-1}}\ \ \bm I_{h_i}\ \ \bm 0_{h_i \times (d_L - s_i)} \,],\quad \bm{E}^{(p_Y+1)} = [\, \bm 0_{(d_L -r_Y) \times r_Y}\ \ \bm I_{d_L - r_Y} \,],
\end{align*}
where \(\bar{\bm{U}}_{L}^{(i)},\ \bm{E}^{(i)} \in \mathbb{R}^{h_i \times d_L }\) for all \(i \in [p_Y]\), and 
\(\bar{\bm{U}}_{L}^{(p_Y+1)},\ \bm{E}^{(p_Y+1)} \in \mathbb{R}^{(d_L - r_Y) \times d_L}\). Applying the Davis-Kahan theorem (see \Cref{lem:daviskahan} in Appendix~\ref{app:sec auxi}) to the matrices 
$\bar{\bm{U}}_L^T \bm{Y} \bm{Y}^T \bar{\bm{U}}_L$ (considered as $\bm{A}$), 
$\bar{\bm{Y}} \bar{\bm{Y}}^T$ (considered as $\hat{\bm{A}}$), 
$\bar{\bm{U}}_L^{(i)^T}$ (considered as $\bm{V}$) , 
and $\bm{E}^{(i)^T}$ (considered as $\hat{\bm{V}}$) for each $i \in [p_Y+1]$, there exist orthogonal matrices 
\(\hat{\bm{U}}_{L}^{(i)} \in \mathcal{O}^{h_i}\) for each \(i \in [p_Y]\) 
and \(\hat{\bm{U}}_{L}^{(p_Y+1)} \in \mathcal{O}^{d_L - r_Y}\) such that  
\begin{align}\label{eq:bddaviskahan}
    \left\| \bar{\bm{U}}_{L}^{(i)} - \hat{\bm{U}}_{L}^{(i)}\bm{E}^{(i)}  \right\|_F 
    \leq  \frac{4\|\bar{\bm U}_L^T \bm Y \bm Y^T \bar{\bm U}_L - \bar{\bm Y} \bar{\bm Y}^T\|_F}{\min\{\lambda_{i-1}-\lambda_{i},\lambda_i-\lambda_{i+1} \}},\ \forall i\in[p_{Y}+1], 
\end{align}
where $\lambda_i$ is the $i$-th largest non-repeated eigenvalue of $\bar{\bm{U}}_L^T \bm{Y} \bm{Y}^T \bar{\bm{U}}_L$ and $\lambda_0=\infty$ and $\lambda_{p_Y+2}=-\infty$ by convention.
We compute $\lambda_i=y_{s_i}^2$ for $i\in[p_Y]$ and $\lambda_{p_Y+1}=0$ .
Using \eqref{eq:control theorem} and the following inequalities
\begin{align*} 
    &y_{s_j}^2 - y_{s_j+1}^2= (y_{s_j} - y_{s_j+1})(y_{s_j} + y_{s_j+1})\ge \delta_y y_{s_{p_Y}},\ \forall j \in [p_Y],\\
    &\min\left\{y_{s_i}^2 - y_{s_i+1}^2, y_{s_{i+1}}^2 - y_{s_{i+1}+1}^2 \right\}\ge \delta_y y_{s_{p_Y}},\ \forall i \in [p_Y-1], 
\end{align*}
we have from \eqref{eq:bddaviskahan} that
$ 
     \| \bar{\bm{U}}_{L}^{(i)} - \hat{\bm{U}}_{L}^{(i)}\bm{E}^{(i)} \|_F 
    \leq  \frac{2^{L+1}(6y_1+\delta_y)\sqrt{\eta_3^2+\eta_4^2}}{3\sqrt{\lambda}y_{s_{p_Y}}\delta_y \sigma^{* L-1}_{\min}} \|\nabla G(\bm{W})\|_F.
$ 
Therefore, we have
\begin{align}\label{eq10:prop sing}
 \left\|\bar{\bm U}_L - \mathrm{BlkD}\left(\hat{\bm U}_{L}^{(1)},\dots,\hat{\bm U}_{L}^{(p_Y+1)} \right) \right\|_F & \le \sum_{i=1}^{p_Y+1} \left\| \bar{\bm U}_{L}^{(i)} -\hat{\bm U}_{L}^{(i)}\bm E^{(i)}\right\|_F \le \eta_5  \|\nabla G(\bm W)\|_F.
\end{align}
According to definition of $\tilde{\bm U}_L$ in \eqref{eq:UL}, we have
\begin{align*}
 \left\| \bm U_L - \tilde{\bm U}_L \right\|_F 
  \overset{\eqref{eq:foldef}}{=} \left\| \bar{\bm U}_L - \mathrm{BlkD}\left(\hat{\bm U}_{L}^{(1)},\dots,\hat{\bm U}_{L}^{(p_Y)},\hat{\bm U}_{L}^{(p_Y+1)} \right) \right\|_F  \overset{\eqref{eq10:prop sing}}{\le} \eta_5  \|\nabla G(\bm W)\|_F\le c_4  \|\nabla G(\bm W)\|_F.
\end{align*}\smallskip

Next we prove the second inequality in \eqref{eq:singular vectors}.
Using the same argument in \eqref{eq:control theorem}-\eqref{eq10:prop sing} to $\bar{\bm V}_1^T \bm Y^T \bm Y \bar{\bm V}_1$ and $\bar{\bm Y}^T \bar{\bm Y}$, we conclude that there exist orthogonal matrices $\hat{\bm V}_{1}^{(i)} \in \mathcal{O}^{h_i}~\forall i \in [p_Y]$ and $\hat{\bm V}_{1}^{(p+1)} \in \mathcal{O}^{d_0-r_Y}$ such that
\begin{align}
    \left\|\bar{\bm V}_1 - \mathrm{BlkD}\left(\hat{\bm V}_{1}^{(1)},\dots,\hat{\bm V}_{1}^{(p_Y)},\hat{\bm V}_{1}^{(p_Y+1)} \right) \right\|_F \le \eta_5\|\nabla G(\bm W)\|_F.  
    \label{eq:boundv}
\end{align} 
Using the block structure of $\bm Y$ in \eqref{eq:SY1}, we have
\begin{align}
&\ \left\|\mathrm{BlkD}\left(\hat{\bm U}_{L}^{(1)},\dots,\hat{\bm U}_{L}^{(p_Y)} \right)-\mathrm{BlkD}\left(\hat{\bm V}_{1}^{(1)},\dots,\hat{\bm V}_{1}^{(p_Y)} \right) \right\|_F\notag \\
\le &\ \frac{1}{y_{s_{p_Y}}} \left\|\mathrm{BlkD}\left(y_{s_1}\hat{\bm U}_{L}^{(1)}\hat{\bm V}_{1}^{(1)T},\dots,y_{s_{p_Y}}\hat{\bm U}_{L}^{p_Y}\hat{\bm V}_{1}^{(p_Y)T}\right)- \mathrm{BlkD}\left(y_{s_1}\bm I_{h_1}, \dots, y_{s_{p_Y}} \bm I_{h_{p_Y}} \right) \right\|_F\notag \\
=&\ \frac{1}{y_{s_{p_Y}}} \left\|\mathrm{BlkD}\left(\hat{\bm U}_{L}^{(1)},\dots,\hat{\bm U}_{L}^{(p_Y)},\hat{\bm U}_{L}^{(p_Y+1)} \right)\bm Y\mathrm{BlkD}\left(\hat{\bm V}_{1}^{(1)},\dots,\hat{\bm V}_{1}^{(p_Y)},\hat{\bm V}_{1}^{(p_Y+1)} \right)^T - \bm Y\right\|_F \notag\\
  \le &\ \frac{1}{y_{s_{p_Y}}}\left( \left\|\left(\mathrm{BlkD}\left(\hat{\bm U}_{L}^{(1)},\dots,\hat{\bm U}_{L}^{(p_Y)},\hat{\bm U}_{L}^{(p_Y+1)} \right) -\bar{\bm U}_L\right)\bm Y \mathrm{BlkD}\left(\hat{\bm V}_{1}^{(1)},\dots,\hat{\bm V}_{1}^{(p_Y)},\hat{\bm V}_{1}^{(p_Y+1)} \right)^T \right\|_F \right.\notag\\ 
&\quad~ \left. +\left\|\bar{\bm U}_L \bm Y\left(\mathrm{BlkD}\left(\hat{\bm V}_{1}^{(1)},\dots,\hat{\bm V}_{1}^{(p_Y)},\hat{\bm V}_{1}^{(p_Y+1)} \right) - \bar{\bm V}_1 \right)^T \right\|_F  + \|\bar{\bm U}_L \bm Y \bar{\bm V}_1^T - \bar{\bm Y}\|_F +\|\bar{\bm Y} - \bm Y\|_F \right)\notag\\
\le &\ \frac{1}{y_{s_{p_Y}}}\left(\frac{2^{L}\sqrt{\eta_3^2+\eta_4^2}}{\sqrt{\lambda} \sigma^{*L-1}_{\min}}+2y_1\eta_5\right)\|\nabla G(\bm W)\|_F. 
\label{eq:diffhatUV}
\end{align}
where the last inequality is due to \eqref{eq8:prop sing}, \eqref{eq10:prop sing}, and \eqref{eq:boundv}.
Using the definition of $\tilde{\bm V}_1$, we compute 
\begin{align*}
&~~~~\|\bm V_1- \tilde{\bm V}_1\|_F  = \left\| \bar{\bm V}_1 - \mathrm{BlkD}\left(\hat{\bm U}_{L}^{(1)},\dots,\hat{\bm U}_{L}^{(p_Y)},\hat{\bm V}_{1}^{(p_Y+1)} \right) \right\|_F  \notag\\
&\le \left\| \bar{\bm V}_1 - \mathrm{BlkD}\left(\hat{\bm V}_{1}^{(1)},\dots,\hat{\bm V}_{1}^{(p_Y)},\hat{\bm V}_{1}^{(p_Y+1)} \right) \right\|_F +  \left\|\mathrm{BlkD}\left(\hat{\bm V}_{1}^{(1)},\dots,\hat{\bm V}_{1}^{(p_Y)} \right) - \mathrm{BlkD}\left(\hat{\bm U}_{L}^{(1)},\dots,\hat{\bm U}_{L}^{(p_Y)} \right) \right\|_F \notag\\
& \overset{(\ref{eq:boundv},\ref{eq:diffhatUV})}{\le} c_4\|\nabla G(\bm W)\|_F.
\end{align*}
That is, the second inequality in \eqref{eq:singular vectors} holds.\medskip

The rest of the proof is devoted to showing that $\bm \Pi$ satisfies $(\bm \sigma^*, \bm \Pi^T) \in \cal B$. According to \eqref{eq9:prop sing} and the definition of $\bar{\bm Y}$ in \eqref{eq:foldef}, we have $\left\| \bm \beta -  \bm \Pi^T\bm y \right\| \le {\delta_y}/{3},$ where $\bm y := (y_1,y_2,\dots,y_{d_{\min}})$. This implies that there exists a permutation $\pi:[d_{\min}] \to [d_{\min}]$ such that
\begin{align}\label{eq12:prop sing}
\left|\beta_{i} - y_{\pi^{-1}(i)}\right| \le \frac{\delta_y}{3},\ \forall i \in [d_{\min}].
\end{align} 
Using Weyl's inequality and 
\eqref{eq:gradient_condition}, 
we have
\begin{align*}
\left| \sigma_i(\bm W_L) - \sigma_i^* \right| \le \mathrm{dist}(\bm W, \mathcal{W}_{\bm \sigma^*}) \le  \min\left\{\delta_1,\frac{\delta_\sigma}{{6}}\right\}, 
\end{align*}
This, together with \eqref{eq:c} and \Cref{lem:derivative bounds} in Appendix~\ref{app:lem deri}, implies 
\begin{align}\label{eq13:prop sing2}
\left| \beta_i - \varphi(\sigma_i^*) \right|  =  \left| \varphi(\sigma_i(\bm W_L) ) - \varphi(\sigma_i^*) \right| \le \frac{\delta_y}{3},\ \forall i \in [r_{\sigma}].
\end{align}
According to \eqref{set:A}, there exists $k_i \in [d_{\min}]$ such that $\varphi(\sigma_i^*) = y_{k_i}$ 
for all $i \in [r_{\sigma}]$. For all $i \in [r_{\sigma}]$, we have
\begin{align*}
\left| y_{k_i} - y_{\pi^{-1}(i)} \right| = \left| \varphi(\sigma_i^*) - y_{\pi^{-1}(i)} \right| \le \left| \varphi(\sigma_i^*) - \beta_i \right| + \left| \beta_i - y_{\pi^{-1}(i)} \right| \overset{(\ref{eq12:prop sing}, \ref{eq13:prop sing2})}{\le} \frac{2\delta_y}{3}.
\end{align*}
Combining this with the definition of $\delta_y$ in \eqref{eq:delta y} yields $\varphi(\sigma_i^*) = y_{k_i} = y_{\pi^{-1}(i)}$ for all $i \in [r_{\sigma}]$.
Thus, we have
\begin{align*}
    (\sigma^*_i)^{2L-1} + \lambda \sigma^*_i - \sqrt{\lambda} y_{\pi^{-1}(i)} (\sigma^*_i)^{{L-1}}  = 0,\ \forall i \in [r_{\sigma}]. 
\end{align*}
For each $i \in \{r_{\sigma}+1,\dots, d_{\min}\}$, we have $\sigma^*_i = 0$, and it is trivial to see that the above equation holds. Thus,
we have $(\sigma_{\pi(i)}^*)^{2L - 1} - \sqrt{\lambda}\, y_{i}\, (\sigma_{\pi(i)}^*)^{L - 1} + \lambda\, \sigma^*_{\pi(i)} = 0,\ \sigma_{\pi(i)}^* \geq 0,\ \forall i \in [d_{\min}].$ This, together with the definition of $\mathcal{A}$ in \eqref{set:A}, yields $\bm{a} := \bm{\Pi}  \bm{\sigma}^* = \left(\sigma_{\pi(1)}^*,\ldots,\sigma_{\pi(d_{\min})}^*\right)\in \mathcal{A}$. 
    It follows from $\bm{\sigma}^*=\bm{\Pi}^T \bm{a}$ and the definition of $\mathcal{B}$ in \eqref{set:B} that $(\bm{\sigma}^*, \bm{\Pi}^T) \in \mathcal{B}$.\medskip

(ii) Using \eqref{eq:c}  and the fact that $y_{\pi^{-1}(i)} = \varphi(\sigma_i^*)$ for each $i \in [r_{\sigma}]$, we have 
\begin{align}\label{eq14:prop sing}
  \left| \varphi(\sigma_i(\bm W_L)) - \varphi(\sigma_i^*)  \right| & = \left| \beta_i - y_{\pi^{-1}(i)}\right| \le \|\bm \beta - \bm \Pi^T \bm y\| = \|\bm \Pi\bm \beta - \bm y\| \notag \\
  & = \|\bar{\bm Y} - \bm Y\|_F \overset{\eqref{eq8:prop sing}}{\le} \frac{2^{L-1}}{\sqrt{\lambda}\sigma_{\min}^{* {L-1}}} \sqrt{ \eta_3^2 + \eta_4^2 } \|\nabla G(\bm W)\|_F.
\end{align}
Under Assumption~\ref{AS:2}, by \Cref{lem:phi} in Appendix~\ref{subsec:A2}, we have 
$\varphi'(\sigma_i^*) \neq 0$ for each $i \in [r_\sigma]$, so that $\min_{i \in [r_{\sigma}] } \left|\varphi^\prime(\sigma_i^*)\right| >0$. 
This, together with \Cref{lem:derivative bounds} (see Appendix~\ref{app:lem deri}), yields that for all $i \in [r_\sigma]$,
\begin{align}\label{eq15:prop sing}
\left|\varphi^\prime(x) - \varphi^\prime(\sigma_i^*) \right| 
\le \frac{|\varphi^\prime(\sigma_i^*)|}{2}, 
\quad \text{if } |x-\sigma_i^*|\le \min\left\{\delta_1,\delta_2,\frac{\delta_\sigma}{3}\right\}. 
\end{align}
By Weyl's inequality, we have
\begin{align}\label{eq16:prop sing}
\left| \sigma_i(\bm W_L) - \sigma_i^* \right| \le \mathrm{dist}(\bm W, \mathcal{W}_{\bm \sigma^*}) \overset{\eqref{eq:gradient_condition}}{\le} \delta_2. 
\end{align}
Applying the mean value theorem to $\varphi(\cdot)$, there exists $x$ satisfying $\sigma_i(\bm W_L)\le x\le \sigma_i^* $  such that 
\begin{align}
  \left|\varphi(\sigma_i(\bm W_L)) - \varphi(\sigma_i^*) \right| = \left|\varphi^\prime(x)\right| \left| \sigma_i(\bm W_L) - \sigma_i^* \right| \ge \frac{\left|\varphi^\prime(\sigma_i^*)\right|}{2} \left| \sigma_i(\bm W_L) - \sigma_i^* \right|, 
  \notag
\end{align}
where the inequality follows from  \eqref{eq15:prop sing} and $|x-\sigma_i^*| \le \left| \sigma_i(\bm W_L) - \sigma_i^* \right| \le \delta_2$ due to \eqref{eq16:prop sing}. This, together with \eqref{eq14:prop sing}, yields for all $i \in [r_{\sigma}]$, 
\begin{align*}
\left| \sigma_i(\bm W_L) - \sigma_i^* \right| \le \frac{2^{L}\sqrt{ \eta_3^2 + \eta_4^2 }}{\sqrt{\lambda}\sigma_{\min}^{* {L-1}}\min_{i \in [r_{\sigma}] } \left|\varphi^\prime(\sigma_i^*)\right| }  \|\nabla G(\bm W)\|_F. 
\end{align*} 
Using this and \eqref{eq:wl+1} in \Cref{lem:pre}, we obtain the desired result.
\end{proof}

Now we are ready to prove the error bound for Problem~\eqref{eq:G}.
\begin{thm}\label{thm:eb G}
Suppose that Assumptions~\ref{AS:1}  and ~\ref{AS:2} hold. Let $\bm \sigma^* \in \mathcal{A}_{\rm sort}\setminus\{\bm 0 \}$ be arbitrary and $\bm W$ be arbitrary such that $\mathrm{dist}(\bmw, \mathcal{W}_{\bm \sigma^*}) \le \epsilon_{\bm \sigma^*}$. Then we have
\begin{align}
\label{eq:eb G}
    \mathrm{dist}(\bmw , \mathcal{W}_{\bm \sigma^*})\le \kappa_{\bm \sigma^*}\|\nabla G(\bmw)\|_F,
\end{align}
where
\begin{align}
\epsilon_{\bm\sigma^*} &:=
\begin{cases}
\displaystyle
\min\!\left\{
\begin{array}{@{}l@{}}
\frac{\delta_{\sigma}}{{6}},\; \delta_1,\; \delta_2,\;
\frac{\delta_y \sqrt{\lambda}\, (\sigma^{*}_{\min})^{L-1}}
         {3 \cdot 2^{L-1} L_G \sqrt{\eta_3^2 + \eta_4^2}}, \\[1.0ex]
\quad \frac{\sqrt{\lambda}}{\sqrt{3(\sqrt{\lambda}+y_1)}}
\Biggl(\min\!\bigl\{\,\min_{i \in [s_{p_Y}]}\lvert\sqrt{\lambda}-y_i\rvert,\; \sqrt{\lambda}\,\bigr\}\Biggr)^{\!1/2}, \\[1.2ex]
\quad \frac{\sqrt{2\lambda}\,
  \min\!\bigl\{\,\min_{i \in [s_{p_Y}]}\lvert\sqrt{\lambda}-y_i\rvert,\; \sqrt{\lambda}\,\bigr\}\,
  \sigma^*_{\min}}
     {12\,c_2\,L_G}
\end{array}
\right\}, &\text{if}\ L = 2, \\[10ex]
\displaystyle
\min\!\Biggl\{ 
   \frac{\delta_{\sigma}}{{6}},\; \delta_1,\; \delta_2,\;
   \frac{\delta_y \sqrt{\lambda}\, (\sigma^{*}_{\min})^{L-1}}
         {3 \cdot 2^{L-1} L_G \sqrt{\eta_3^2 + \eta_4^2}},\;
   \left( \frac{\sqrt{\lambda}}{2y_1} \right)^{\!\frac{1}{L-2}}
\Biggr\}, &\text{if}\ L \ge 3, 
\end{cases} \label{def:epsilon_kappa_2} \\[1.5ex]
\kappa_{\bm \sigma^*} &:= 
\sqrt{L}\left(
   \frac{9\sigma_{\max}^{*2}}{4\delta_{\sigma}\lambda\sigma^{*}_{\min}}
   + c_3\sqrt{d_{\max} - r_{\sigma}}
   + c_4\sigma_{\max}^*
   + c_5\sqrt{r_{\sigma}}
\right).\label{def:kappa_2}
\end{align}
All the explicit expressions of the constants above can be found in Table~\ref{tab:notations}.
\end{thm} 
\begin{proof}
 By the fact that $\mathrm{dist}(\bmw,\mathcal{W}_{\bm \sigma^*}) \le \epsilon_{\bm \sigma^*}$, 
together with \Cref{prop:Hlapproximate}, we obtain
that there exist matrices $\bm T_{l}^{(i)} \in \mathcal{O}^{g_i}$ for each $i \in [p]$ and $\bm T_{l}^{(p+1)} \in \mathcal{O}^{d_{l-1} - r_{\sigma}}$ such that \eqref{eq:Hlapproximate} holds. Combining this with \Cref{app:b4}, we have that \eqref{eq:gradient_condition} and \eqref{eq:dist grad 1} (resp., \eqref{eq:dist 5}) hold when $L=2$ (resp., $L \ge 3$). Moreover, it follows from \Cref{prop:singular value 2} and \Cref{prop:singular control} that \eqref{eq:singular 2}, \eqref{eq:singular vectors}, \eqref{eq:singular 1} hold, where $\tilde{\bm U}_L$, $\tilde{\bm V}_1$, and $\hat{\bm T}$ are respectively defined in \eqref{eq:UL}, \eqref{eq:V1}, and \eqref{eq:T}. 
Recall that $\bm U_l \in \mathcal{O}^{d_l}$ and $\bm V_l \in \mathcal{O}^{d_{l-1}}$ for all $l \in [L]$ are introduced in \eqref{eq:SVD Wl}. For ease of exposition, we define $\bm \Sigma_l^* := \mathrm{BlkD}(\mathrm{diag}(\bm \sigma^*), \bm 0) \in \R^{d_l\times d_{l-1}}$ and 
\begin{align}\label{eq0:thm eb}
\hat{\bm W}_1 := \bm U_1\bm \Sigma_1^* \tilde{\bm V}_1^T,\qquad \hat{\bm W}_l := \bm U_l\bm \Sigma_l^*\bm V_l^T,\ l =2,\dots,L-1,\qquad \hat{\bm W}_L := \tilde{\bm U}_L\bm \Sigma_L^*\bm V_L^T. 
\end{align}
Now, we compute
\begin{align}\label{eq2:thm eb}
\|\hat{\bm W}_1 - \bm W_1\|_F & = \left\| \bm U_1\bm \Sigma_1^* \tilde{\bm V}_1^T -  \bm U_1\bm \Sigma_1 \bm V_1^T \right\|_F\notag \\
& \le  \left\| \bm U_1\bm \Sigma_1^* \tilde{\bm V}_1^T -  \bm U_1\bm \Sigma_1^* \bm V_1^T \right\|_F +  \left\|   \bm U_1\bm \Sigma_1^* \bm V_1^T -  \bm U_1\bm \Sigma_1 \bm V_1^T  \right\|_F\\
& \le \sigma_{\max}^* \|\tilde{\bm V}_1 - \bm V_1\|_F + \|  \bm \Sigma_1^* - \bm \Sigma_1\|_F \notag\\
& \le \left( c_4\sigma_{\max}^* + c_3\sqrt{d_{\max} - r_{\sigma}} + c_5\sqrt{r_{\sigma}} \right) \|\nabla G(\bm W)\|_F, 
\end{align}
where the last inequality uses  \eqref{eq:singular vectors},  \eqref{eq:singular 2}, and \eqref{eq:singular 1}. Using the same argument, we obtain 
\begin{align}\label{eq3:thm eb}
\|\hat{\bm W}_L - \bm W_L\|_F \le \left( c_4\sigma_{\max}^* + c_3\sqrt{d_{\max} - r_{\sigma}} + c_5\sqrt{r_{\sigma}} \right) \|\nabla G(\bm W)\|_F. 
\end{align} 
For all $l=2,\dots,L-1$, we compute
\begin{align*}
\|\hat{\bm W}_l - \bm W_l\|_F &  =  \left\|   \bm \Sigma_l^* - \bm \Sigma_l \right\|_F \overset{(\ref{eq:singular 2},\ref{eq:singular 1})}{\le} \left( c_3\sqrt{d_{\max} - r_{\sigma}} + c_5\sqrt{r_{\sigma}} \right) \|\nabla G(\bm W)\|_F.
\end{align*}
This, together with \eqref{eq2:thm eb} and \eqref{eq3:thm eb}, yields
\begin{align}\label{eq4:thm eb}
\|\bm W - \hat{\bm W}\|_F \le \sqrt{L} \left( c_4\sigma_{\max}^* + c_3\sqrt{d_{\max} - r_{\sigma}} + c_5\sqrt{r_{\sigma}} \right) \|\nabla G(\bm W)\|_F. 
\end{align}
Next, we further define 
\begin{align}
& \bm T_l := \mathrm{BlkD}(\bm T_{l}^{(1)},\dots,\bm T_{l}^{(p)},\bm T_{l}^{(p+1)}),\ l=2,\dots,L, \label{eq6:thm eb}\\
& \bm W_1^* := \bm Q_2 \bm \Sigma_1^* \hat{\bm T}^T \tilde{\bm V}_1^T,\ \bm W_l^* := \bm Q_{l+1} \bm \Sigma_l^*\bm Q_l^T,\ l =2,\dots,L-1,\ \bm W_L^*:= \tilde{\bm U}_L\bm \Sigma_L^*\bm Q_L^T, \label{eq7:thm eb}\\
&\bm Q_l  := \bm V_l\mathrm{BlkD}\left(\left(\prod_{k=l+1}^L\bm T_{k}^{(1)}\right),\dots,\left(\prod_{k=l+1}^L\bm T_{k}^{(p)}\right) , \bm I_{d_{l-1}-r_{\sigma}}\right),\ l = 2,\dots,L-1,\ \bm Q_L := \bm V_L.\label{eq:defQl}
\end{align}
Using \Cref{prop:opti G}, one can verify that $\bm W^* = (\bm W_1^*,\bm W_2^*,\dots,\bm W_L^*) \in \mathcal{W}_{\bm \sigma^*}.$
For each $l=2,\dots,L-1$, we compute
\begin{align}\label{eq8:thm eb}
\|\hat{\bm W}_l - \bm W_l^*\|_F & = \left\| \bm U_l\bm \Sigma_l^*\bm V_l^T - \bm Q_{l+1} \bm \Sigma_l^*\bm Q_l^T \right\|_F \notag \\
& \overset{\eqref{eq:partsigma}}{=} \left\| \bm V_{l+1}^T \bm U_l\bm \Sigma_l^* - \mathrm{BlkD}\left( \sigma_{t_1}^*(\bm T_{l+1}^{(1)})^T,\dots,\sigma_{t_p}^*(\bm T_{l+1}^{(p)})^T, \bm 0 \right) \right\|_F \notag\\
& = \left\| (\bm V_{l+1}^T \bm U_l - \bm T_{l+1}^T) \bm \Sigma_l^* \right\|_F \overset{\eqref{eq:Hlapproximate}}{\le} \frac{9\sigma_{\max}^{*2}}{ 4\delta_{\sigma}\lambda\sigma^{*}_{\min}} \|\nabla G(\bm W)\|_F,
\end{align}
where 
the last equality uses the diagonal block forms of $\bm T_{l}$ and $\bm \Sigma_l^*$. Using \eqref{eq0:thm eb} and \eqref{eq7:thm eb}, we compute 
\begin{align*}
	\|\hat{\bm W}_1 - \bm W_1^*\|_F 
	&= \|\bm U_1 \bm \Sigma_1^* \tilde{\bm V}_1^T - \bm Q_2 \bm \Sigma_1^*\hat{\bm T}^T \tilde{\bm V}_1^T\|_F 
	= \|\bm V_2^T \bm U_1 \bm \Sigma_1^* - \bm V_2^T \bm Q_2 \bm \Sigma_1^*\hat{\bm T}^T  \|_F \notag \\
	&\leq \|\bm V_2^T\bm U_1 \bm \Sigma_1^* - \bm T_2^T \bm \Sigma_1^*\|_F 
	+ \|\bm T_2^T \bm \Sigma_1^* - \bm V_2^T \bm Q_2 \bm \Sigma_1^* \hat{\bm T}^T \|_F \notag \\
	&= \|(\bm U_1^T \bm V_2 - \bm T_2)^T \bm \Sigma_1^*\|_F 
	+ \|\bm \Sigma_1^* - \bm T_2 \bm V_2^T \bm Q_2 \bm \Sigma_1^* \hat{\bm T}^T \|_F \notag \\
    & = \|(\bm U_1^T \bm V_2 - \bm T_2)^T \bm \Sigma_1^*\|_F \overset{\eqref{eq:Hlapproximate}}{\le} \frac{9\sigma_{\max}^{*2}}{ 4\delta_{\sigma}\lambda\sigma^{*}_{\min}}\|\nabla G(\bm W)\|_F,
\end{align*} 
where the last equality follows from the forms of $\bm T_l$ in \eqref{eq6:thm eb}, 
$\bm Q_2$ in \eqref{eq:defQl}, and $\hat{\bm T}$ in \eqref{eq:T}. This, together with \eqref{eq8:thm eb} and the fact that $\bm W_L^* = \hat{\bm W}_L$ due to $\bm Q_L = \bm V_L$, yields 
\begin{align}\label{eq:hatstardiff} \|\hat{\bm W} - \bm W^*\|_F \le \frac{9\sigma_{\max}^{*2}\sqrt{L-1}}{ 4\delta_{\sigma}\lambda\sigma^{*}_{\min}}\|\nabla G(\bm W)\|_F.\end{align}
Since $\bmw^*\in \mathcal{W}_{\bm\sigma^*}$, we have 
$\mathrm{dist}(\bm W, \mathcal{W}_{\bm \sigma^*})  \le \|\bm W - \bm W^*\|_F \le \|\bm W - \hat{\bm W}\|_F + \|\hat{\bm W} - \bm W^*\|_F.$ 
This, together with \eqref{eq4:thm eb} and \eqref{eq:hatstardiff}, implies \eqref{eq:eb G}. 
\end{proof}

Using \Cref{thm:eb G} and the connection of Problems \eqref{eq:F} and \eqref{eq:G}, we are now ready to prove \Cref{thm:eb}. 
\begin{proof}[Proof of \Cref{thm:eb}] 
According to \Cref{thm:eb G}, we conclude that for any $\bm \sigma^* \in \mathcal{A}_{\rm sort} \setminus \{\bm 0\}$, 
\eqref{eq:eb G} holds for $\mathcal{W}_{\bm \sigma^*}$, where the constants $\epsilon_{\bm \sigma^*}$ and $\kappa_{\bm \sigma^*}$ (see \eqref{def:epsilon_kappa_2} and \eqref{def:kappa_2}) depend on $\bm \sigma^*$. Recalling the definition of $\mathcal{A}$ in \eqref{set:A}, and noting that a polynomial has only finitely many roots, it follows that $\mathcal{A}_{\rm sort}$ defined in \eqref{set:W sigma} is finite. Consequently, by \eqref{set:W sigma}, we can derive a corresponding result for $\mathcal{W}_G$ that does not depend on $\bm \sigma^*$. In addition, we take the case $\bm\sigma^*=\bm0$ in \Cref{prop:eb zero} into consideration, where the constants $\epsilon_{\bm0}$ and $\kappa_{\bm0}$ are defined in \eqref{def: epsilon_kappa1}. Specifically, by letting
\begin{align}
    \label{eq:def_kappa_epsilon}
    \kappa := \max_{\bm \sigma^*\in \mathcal{A}_{\rm sort}} \kappa_{\bm \sigma^*}\quad \text{and}\quad \epsilon:= \min_{\bm\sigma^*\in \mathcal{A}_{\rm sort}} \epsilon_{\bm \sigma^*},
\end{align} 
we have 
$
    \mathrm{dist}(\bm W,\mathcal{W}_G)\le \kappa\|\nabla G(\bm W)\|_F
$ for all $\mathrm{dist}(\bmw , \mathcal{W}_G)\le \epsilon.$  Using this and \Cref{lem:equi FG}{(iii)}, \eqref{eq:eb F} holds with 
\begin{align}
\label{eq:def_epsilon_kappa_1}
    \kappa_1 =  {\kappa\lambda}/{\lambda_{\min}}\quad \text{and}\quad\epsilon_1 =  {\epsilon}/{\sqrt{\lambda_{\max}}}.
\end{align}  
\end{proof}
\section{Experimental Results}\label{sec:expe}

In this section, we conduct experiments under different settings to validate our theoretical results. Specifically, we employ GD to train deep networks using the \texttt{PyTorch} library.
We terminate the algorithm at the first iteration \(k\ge1\) such that the squared gradient norm satisfies $\|\nabla F(\bm{W}^k)\|_F^2 \leq 10^{-6}$ and the function value change satisfies $|F(\bm{W}^k) - F(\bm{W}^{k-1})| \leq 10^{-7}$, where $\bm W^k$ is the $k$-th iterate.

\subsection{Linear Convergence to Critical Points}\label{subsec:exp 1}

In this subsection, we investigate the convergence behavior of GD to different critical points of Problem \eqref{eq:F}. In our experiments, we set $d_L=20$, $d_0=10$, and $d_l=32$ for each $l=1,\dots,L-1$. Then, we i.i.d. sample each entry of $\bm Y \in \R^{d_L\times d_0}$ from the standard Gaussian distribution, i.e., $y_{ij} \overset{i.i.d.}{\sim} \mathcal{N}(0,1)$. Moreover, we set the regularization parameters $\lambda_l = 10^{-4}$ for all $l \in [L]$ and the learning rate $4.5 \times 10^{-4}$. To ensure convergence to different critical points, we initialize the weight matrices in the neighborhood of two different critical points of Problem \eqref{eq:F}. Specifically, we apply the SVD to $\bm Y$ and solve the equation \eqref{eq:sigma W} to get its roots. Using these results and \Cref{thm:opti}, we can respectively construct an optimal solution, denoted by $\bm W_{\rm opt}^*$, and a non-optimal critical point, denoted by $\bm W_{\rm crit}^*$, of Problem \eqref{eq:F}.  
Then, we set $\bm W^0 = \bm W^* + 0.01 \bm \Delta$, where $\bm W^*$ is $\bm W_{\rm opt}^*$ or $\bm W_{\rm crit}^*$ and each entry of $\bm \Delta$ is i.i.d. sampled from the standard Gaussian distribution. For each initialization, we run GD for solving Problem \eqref{eq:F} with different depths $L\in \{2,4,6\}$.

\begin{figure}[t]
    \centering
    \begin{minipage}{0.48\textwidth}
        \centering
        \includegraphics[width=0.7\textwidth]{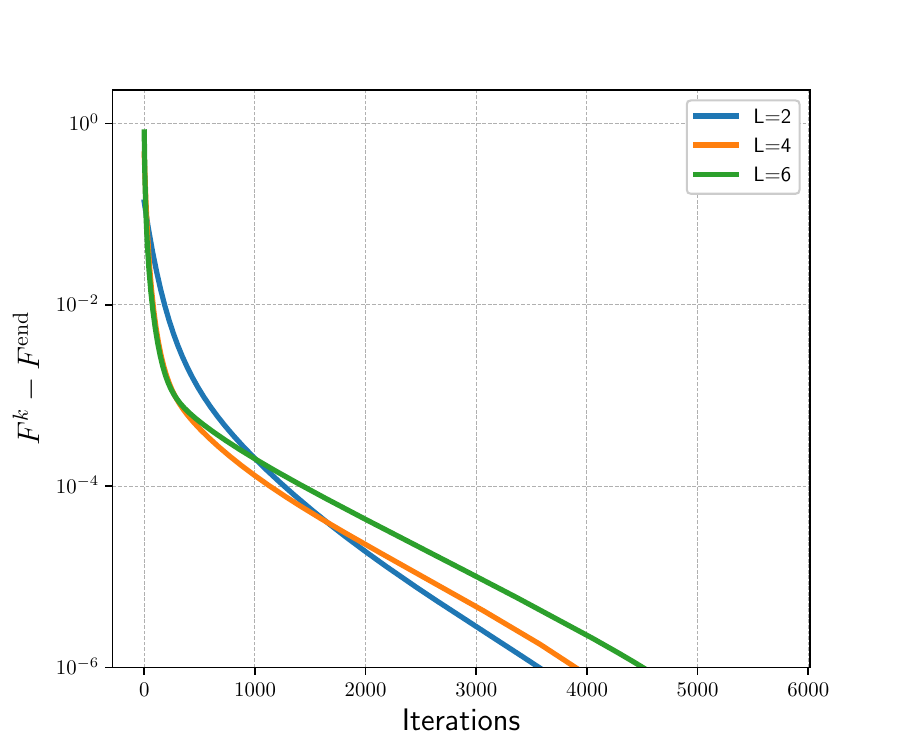} 
        \subfigure{\small (a) Training loss }
\end{minipage}\quad 
    \begin{minipage}{0.48\textwidth}
        \centering
        \includegraphics[width=0.7\textwidth]{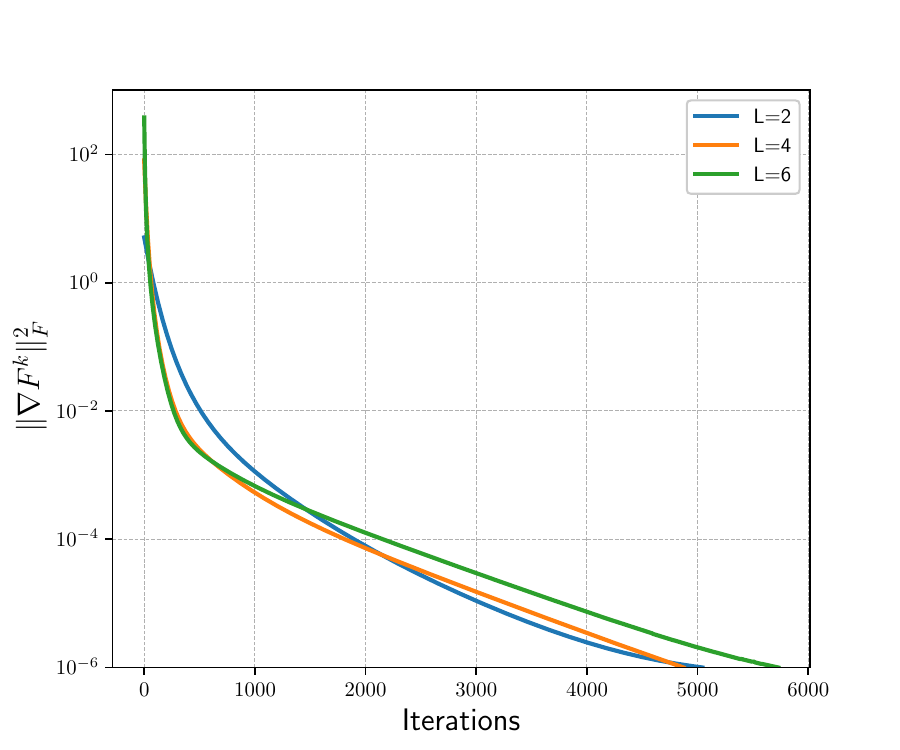} 
        \subfigure{\small (b) Squared gradient norm }
    \end{minipage}\\
    \vspace{-0.05in}
    \begin{minipage}{0.48\textwidth}
        \centering
        \includegraphics[width=0.7\textwidth]{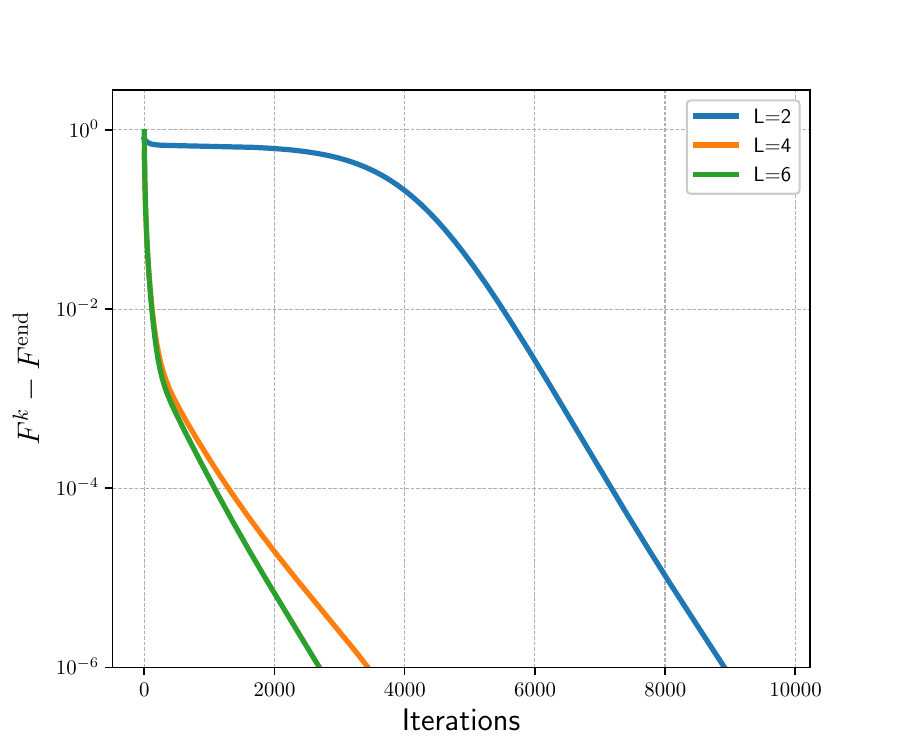} 
        \subfigure{\small (c) Training loss }
    \end{minipage}\quad 
    \begin{minipage}{0.48\textwidth}
        \centering
        \includegraphics[width=0.7\textwidth]{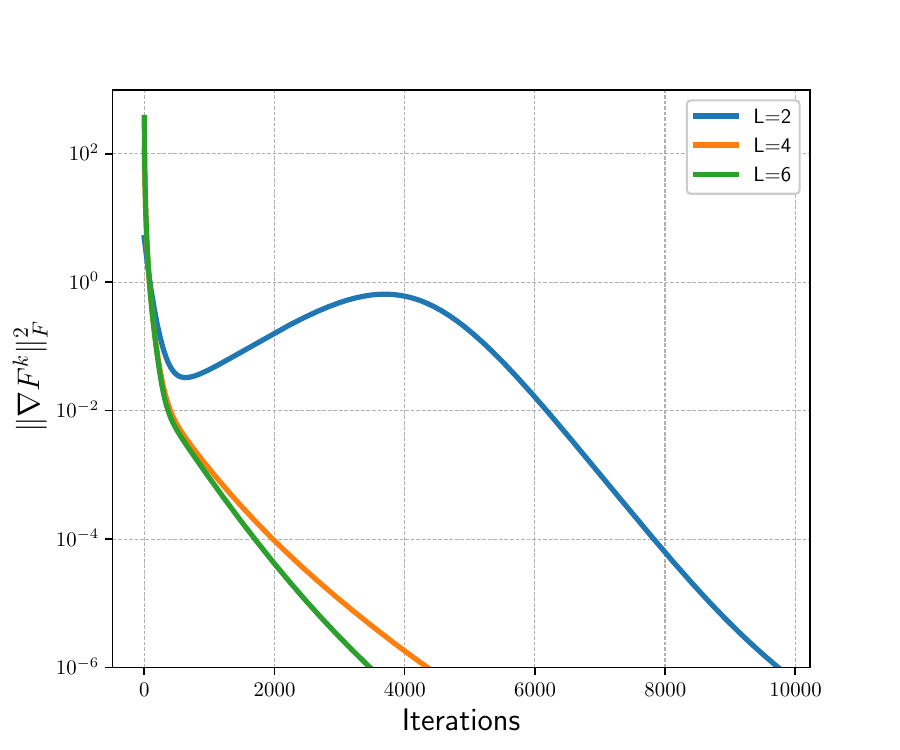} 
        \subfigure{\small (d) Squared gradient norm }
    \end{minipage} 
    \caption{Linear convergence of GD to a critical point.
     In (a) and (b), $\{\bm W^k\}$ converges to an optimal solution when initialized near an optimal solution; in (c) and (d), \(\{\bm W^k\}\) converges to a non-optimal critical point for \(L = 4, 6\), and to an optimal solution for \(L = 2\) when initialized near a non-optimal critical point.
Here $F^{\rm end}$ denotes the training loss when the stopping criterion is met.
    } 
    \label{fig:convergence_rates}
\end{figure}
\begin{table}[!htbp]
\centering
\caption{\centering The function values at the final iterate $\bm W^{\rm end}$ and the critical point $\bm W^*$.}\label{tab:loss_table}
\begin{tabular}{c|c|c|c|c}
\hline
 Initialization& Objective values  & $L = 2$ & $L = 4$ & $L = 6$ \\ \hline
\multirow{2}{*}{\shortstack{Initialization near\\ global optimal solution}} 
  & $F(\bm W_{\rm opt}^*)$ & $9.2900\times 10^{-3}$ & $8.1723 \times 10^{-3}$ & $9.5264\times 10^{-3}$ \\
  & $F(\bm W^{\rm end})$ & $9.2985 \times 10^{-3}$ & $8.2020 \times 10^{-3}$ & $9.5741 \times 10^{-3}$ \\ \hline
\multirow{2}{*}{\shortstack{Initialization near\\ non-optimal critical point}} 
  & $F(\bm W^*_{\rm crit})$ & $6.8038 \times 10^{-1}$ & $6.7906 \times 10^{-1}$ & $6.8022 \times 10^{-1}$ \\ 
  & $F(\bm W^{\rm end})$ & $9.2980 \times 10^{-3}$ & $6.7909 \times 10^{-1}$ & $6.8027 \times 10^{-1}$ \\ \hline
\end{tabular}
\end{table}
\medskip 

We respectively plot the function value gap $F(\bm W^k) - F(\bm W^{\rm end})$ and the squared gradient norm $\|\nabla F(\bm W^k)\|_F^2$ against the iteration number in \Cref{fig:convergence_rates}(a) and (b) (resp., \Cref{fig:convergence_rates}(c) and (d)), where $ \bm W^{\rm end}$ denotes the last iterate of the GD.
We also report the function values of the final iterate, in different settings in Table~\ref{tab:loss_table}. As observed from \Cref{fig:convergence_rates} and Table~\ref{tab:loss_table}, GD converges to an optimal solution at a linear rate for solving Problem \eqref{eq:F} with different network depths, and similarly, it converges to a non-optimal critical point at a linear rate for $L=4,6$. We should point out that Problem \eqref{eq:F} only has global optima and strict saddle points when $L=2$, as shown in \cite{chen2025complete}. Therefore, even if initialized in the vicinity of a non-optimal critical point, i.e., strict saddle point, the algorithm will ultimately converge to a global optimal point almost surely \cite{lee2019first}. This aligns with our linear convergence analysis for all critical points of Problem \eqref{eq:F} in \Cref{prop:threecondition} in Appendix~\ref{app:sec C}, which leverages the error bound of Problem \eqref{eq:F}, and thereby supports \Cref{thm:eb}.

\subsection{Convergence Behavior in General Settings}

In this subsection, we investigate the convergence behavior of GD in more general setups extending beyond linear networks. Specifically, the network depth is fixed as $L=4$ and network widths at different layers are set as follows: $d_L=16$, $d_0=10$, and $d_l=32$ for each $l=2,\dots,L-1$. The regularization parameter is set as $\lambda_l = 5\times10^{-5}$ for all $l \in [L]$ and the learning rate of GD is set as $10^{-3}$. We use the default initialization scheme in \texttt{PyTorch} to initialize the weights for GD. The data matrix $\bm{X}$ is generated according to a uniform distribution using the \texttt{xavier\_uniform\_} function in \texttt{PyTorch}, while the target matrix $\bm Z$ is generated using the same approach as described in \Cref{subsec:exp 1}.  Below, we outline the different setups used in our experiments. 

{\bf General data input.} 
We first consider general data inputs $\bm X_1,\bm X_2,\bm X_3$ instead of orthogonal inputs $\bm X$. Then, we apply GD to optimize Problem \eqref{eq:F} with a linear network $\bm W_L\ldots\bm W_1\bm X$ for each of these different data matrices. 

 {\bf Linear networks with bias.} We next study the regularized loss of deep linear networks with bias terms: 
\[ 
\min\limits_{\{\bm W_l,\bm b_l\}}\; \sum_{i=1}^N\left\| \bm{W}_L(\bm W_{L-1} \cdots (\bm {W}_1 \bm x_i+ \bm{b}_1)+\bm b_{L-1}) + \bm b_L - \bm z_i \right\|^2 + \sum_{l=1}^{L} \lambda_l \left(\left\| \bm{W}_l \right\|_F^2 + \left\| \bm{b}_l \right\|^2 \right),
\]
 where $\bm x_i$ and $\bm z_i$ respectively denote the $i$-th column of $\bm X$ and $\bm Z$. We apply GD to solve the above problem when the input is either an identity matrix or general input data.

 {\bf Deep nonlinear networks.} 
Finally, we study the performance of GD for solving the regularized loss of deep nonlinear networks with different activation functions:
\begin{align*}
   \min\limits_{\{\bm W_l,\bm b_l\}}\ \sum_{i=1}^N \left\|\bm W_L \sigma \left( \bm W_{L-1}\cdots\sigma\left(\bm W_1 \bm x_i + \bm b_1 \right) + \bm b_{L-1} \right) + \bm b_L - \bm z_i \right\|^2 + \sum_{l=1}^{L}  \lambda_l \left(\left\| \bm{W}_l \right\|_F^2 +  \left\| \bm{b}_l \right\|^2 \right), 
\end{align*}
where $\sigma(\cdot)$ denotes an activation function. We set $\bm{X}$ as the identity matrix and use GD for solving the above problem when the activation functions are chosen as ReLU, Leaky ReLU, and tanh, respectively.

\begin{figure}[t]
    \centering
\begin{minipage}{0.33\textwidth}
\centering  \includegraphics[width=\textwidth]{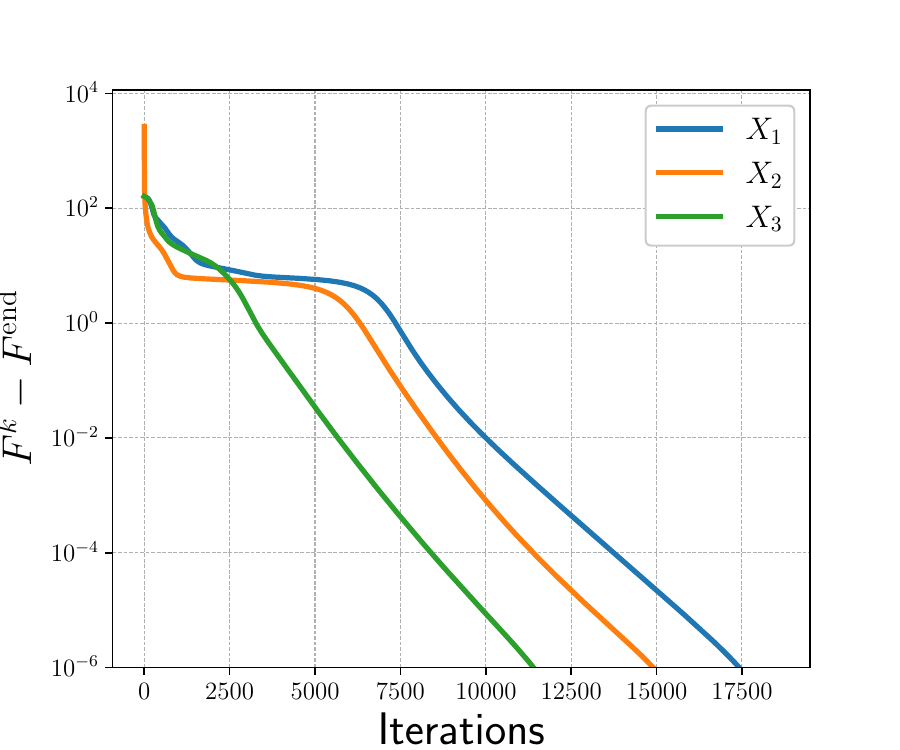}\vspace{-0.05in}
\subfigure{(a) General data input}
\end{minipage}%
\begin{minipage}{0.33\textwidth}
\centering
\includegraphics[width=\textwidth]{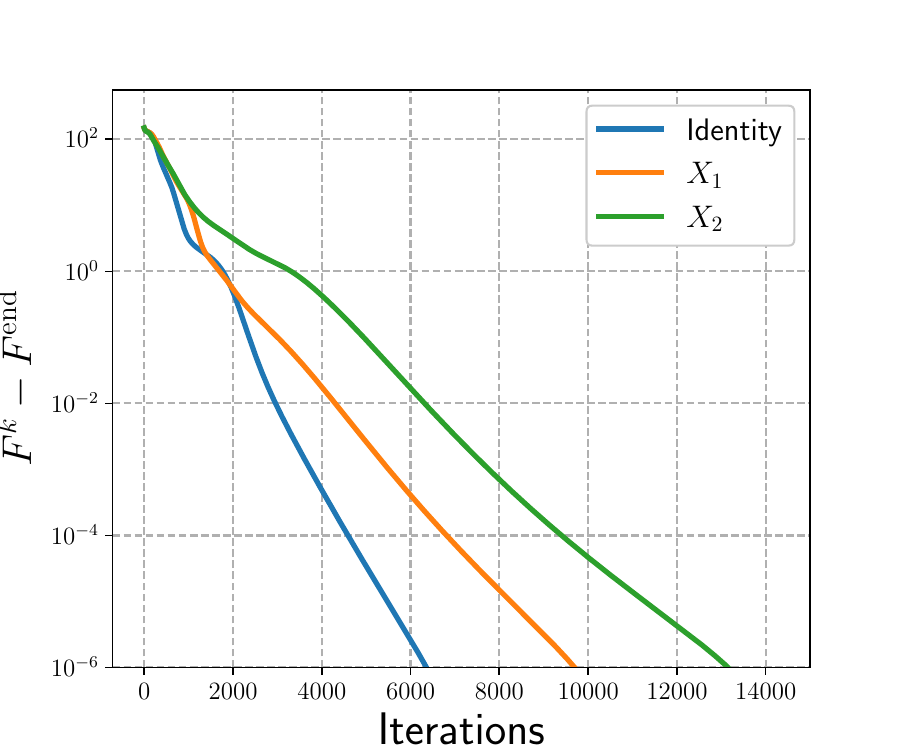}\vspace{-0.05in}
\subfigure{(b) Linear networks with bias}
    \end{minipage}%
\begin{minipage}{0.33\textwidth}
\centering
\includegraphics[width=\textwidth]{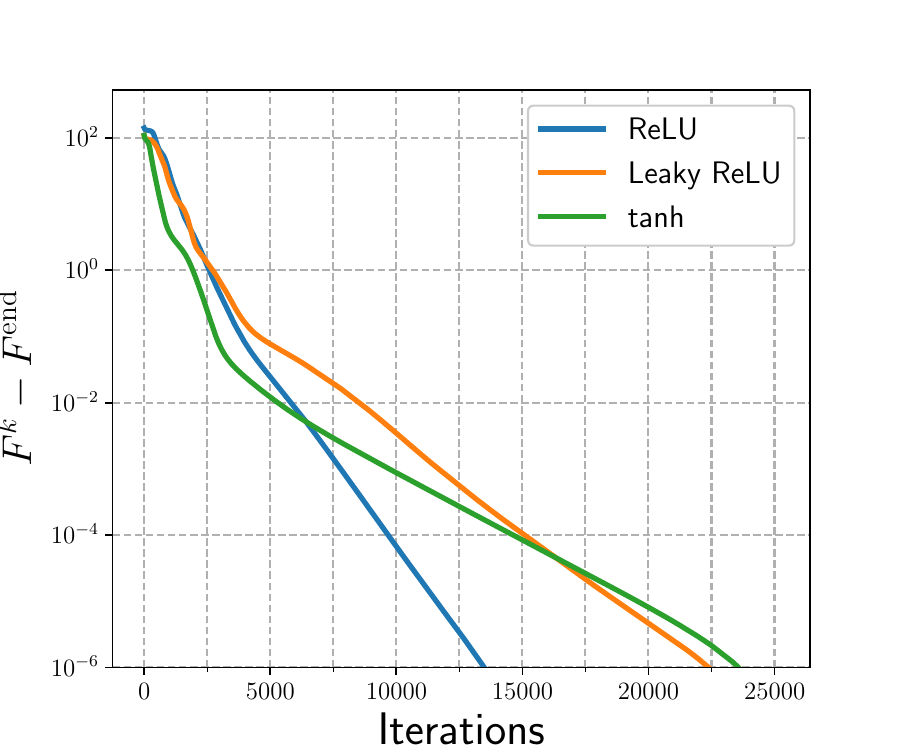}\vspace{-0.05in}
\subfigure{(c) Deep nonlinear networks}
\end{minipage} \vspace{-0.05in} 
\caption{Linear convergence of GD under different settings:
(a) Three different matrices $\bm X_1$, $\bm X_2$, and $\bm X_3$, with condition numbers of 43.43, 36.04, and 16.36, respectively. (b) Identity matrix, $\bm X_1$, and $\bm X_2$, with condition numbers of 1, 4.82, and 14.62, respectively. (c) Identity matrix as data input with ReLU, Leaky ReLU, and tanh as activation functions. $F^{\rm end}$ denotes the training loss at the final step when the stop condition is triggered.} 
\label{fig:extent experiments}
\end{figure}

For the above three different settings, we plot the function value gap $F(\bm W^k) - F(\bm W^\text{end})$ against the iteration number in \Cref{fig:extent experiments}. It is observed from \Cref{fig:extent experiments} that GD converges to a solution at a linear rate across these settings. This consistent behavior leads us to conjecture that the error-bound condition may hold for deep networks in more general scenarios. Additionally, we observe that the number of iterations required to meet the stopping criterion generally increases as the condition number of the input data becomes larger. Exploring this phenomenon presents an interesting direction for future research. 

\section{Conclusions}\label{sec:con}

In this paper, we studied the regularized squared loss of deep linear networks and proved its error bound, a regularity condition that characterizes local geometry around the critical point set. This result is not only theoretically significant but also lays down the foundation for establishing strong convergence guarantees for various methods for solving Problem \eqref{eq:F}. To establish the error bound, we specialized an existing critical-point characterization under our width assumption and developed new analytic techniques to prove the error bound, which may be of independent interest. Our numerical results across different settings provide strong support for our theoretical findings. One future direction is to extend our analysis to deep linear networks with more general data input $\bm X$ and loss functions. Another interesting direction is to investigate the regularized loss of deep nonlinear networks.

\section*{\bf Acknowledgments.}
We thank both anonymous reviewers for their careful reading and constructive comments, which helped substantially improve the paper. We are especially grateful to one reviewer for raising the question of whether the P\L\ inequality is equivalent to a local error bound for general \(C^2\) functions on finite-dimensional spaces, and for providing valuable insights into the proofs.



\bibliographystyle{abbrvnat}
\bibliography{reference,NC}

\appendix
\setcounter{section}{0}
\setcounter{lemma}{0}
\setcounter{defi}{0}
\setcounter{page}{1}
\pagenumbering{roman} 
\renewcommand\thesection{\Alph{section}}
\renewcommand{\thelemma}{\Alph{section}.\arabic{lemma}}
\renewcommand{\thedefi}{\Alph{section}.\arabic{definition}} 

\clearpage 
\begin{appendix}
\begin{center}
{\Large \bf Auxiliary Proofs and Results}
\end{center} 


\section{Showing the Necessity of Assumption~\ref{AS:2}}\label{app sec:counter} 

We claim that the error bound of $F$ (see Problem \eqref{eq:F}) holds if and only if the error bound of $G$ (see Problem \eqref{eq:G}) holds. Indeed, it follows from (iii) in \Cref{lem:equi FG} that the ``if'' direction holds. Now, it remains to show the ``only if'' direction. Suppose that the error bound  of Problem \eqref{eq:F} holds (see \eqref{eq:eb F}). 
Using the proof setup in \Cref{lem:equi FG}, we have
\[ 
    \frac{ \lambda  }{ \lambda_{\max} } \, \mathrm{dist}(\hat{\bm{W}}, \mathcal{W}_G) \overset{\eqref{eq:distconnect}}{\leq} \frac{ \lambda }{ \sqrt{ \lambda_{\max} } } \, \mathrm{dist}(\bm{W}, \mathcal{W}_F) \overset{\eqref{eq:eb F}}{\leq} \frac{ \lambda \kappa_1 }{ \sqrt{ \lambda_{\max} } } \, \| \nabla F(\bm{W}) \|_F \overset{\eqref{eq:gradientnorm}}{\leq} \kappa_1 \, \| \nabla G(\hat{\bm{W}}) \|_F.
\]
Therefore, we prove the claim. Based on the error-bound equivalence between $F$ and $G$, establishing the necessity of Assumption~\ref{AS:2} for $F$ is equivalent to doing so for $G$.

\subsection{The Case $L=2$}

For ease of exposition, we denote $\widehat{\mathcal{W}}_{\bm \sigma^*} := \mathcal{W}_{\rm sort(\bm \sigma^*)}$ for each $\bm \sigma^*\in \mathcal{A}$, where $\mathrm{sort}(\cdot)$ is a sorting function that arranges the elements of a vector in decreasing order. 

\begin{lemma}\label{lem:counter 1}
Suppose that $L=2$ and \eqref{eq:AS L=2} does not hold. For any $\bm \sigma^* \in \cal A$, the error bound for the critical point set $\widehat{\mathcal{W}}_{\bm \sigma^*}$ of Problem \eqref{eq:G} fails to hold.
\end{lemma}
\begin{proof}
Since \eqref{eq:AS L=2} does not hold, there exists $i \in [r_Y]$ such that $y_i = \sqrt{\lambda}$. Now, we define $\bm W(t) := \left(\bm W_1(t),\bm W_2(t) \right)$ as follows: 
\begin{align*}
\begin{cases}
        &\bm W_1(t) = \bm Q_2\bm \Sigma_1(t) \mathrm{BlkD}\left( \bm O_1,\dots,\bm O_{p_{Y}},\bm O_{p_{Y}+1} \right),\ \bm W_2(t) =  \mathrm{BlkD}\left( \bm O_1^T,\dots,\bm O_{p_{Y}}^T,\widehat{\bm O}_{p_{Y}+1}^T \right)\bm \Sigma_2(t) \bm Q_2^T, \\
        &\bm \Sigma_l(t) = \mathrm{BlkD}\left(\mathrm{diag}(\bm \sigma^* +t \bm e_i), \bm 0\right) \in \R^{d_l\times d_{l-1}}, \forall l = 1,2,\
        \bm O_i \in \mathcal{O}^{h_i}, \forall i \in [p_{Y}], \bm O_{p_{Y}+1} \in \mathcal{O}^{d_0 - r_{Y}}, \widehat{\bm O}_{p_{Y}+1} \in \mathcal{O}^{d_2 - r_{Y}}. 
\end{cases}
\end{align*}
According to \Cref{prop:opti G}, one can verify that $\bm W(0) \in \widehat{\mathcal{W}}_{\bm \sigma^*}$. Using $y_i = \sqrt{\lambda}$, we obtain  
$\sigma^*_i \in \{x: x^3 -\sqrt{\lambda}y_i x+\lambda x=0,\ x\ge 0\} = \{0\}$.
Then, we compute 
\begin{align}\label{eq1:lem counter 1}
 \left\|\nabla_{l} G(\bmw(t))\right\|_F \overset{\eqref{eq:grad G}}{=} 2|t^3- \sqrt{\lambda}y_i t+\lambda t| = 2|t^3|,\  l=1,2,
\end{align}
where the second equality follows from $y_i = \sqrt{\lambda}$. 
Moreover, we have
\begin{align} \mathrm{dist}^2(\bmw(t),\widehat{\mathcal{W}}_{\bm \sigma^*})
    & =\min_{(\bm W_1,\bm W_2) \in \widehat{\mathcal{W}}_{\bm \sigma^*}} \|\bmw_1(t)-\bmw_1 \|_F^2+\|\bmw_L(t)-\bmw_2 \|_F^2 \notag\\
    &   \ge \|\bm \Sigma_1(t) - \bm \Sigma_1(0)\|_F^2 +  \|\bm \Sigma_2(t) - \bm \Sigma_2(0)\|_F^2  = 2t^2, \label{eq:final count1} 
\end{align}
where the first inequality uses Mirsky's inequality (see \Cref{lem:mirsky}). 
When $|t| \le  {\delta_\sigma}/{3}$, we have 
$ 
\mathrm{dist}(\bm W(t), \widehat{\mathcal{W}}_{\bm \sigma^*})\le \|\bm W(t) - \bm W(0)\|_F  \le  {\sqrt{2}\delta_\sigma}/{3}.
$ 
Using this  and \Cref{prop:set W}(ii), we have for each $  \bar{\bm \sigma}^* \in \mathcal{A}$ satisfying $ \widehat{\mathcal{W}}_{\bm \sigma^*}\neq \widehat{\mathcal{W}}_{\bar{\bm \sigma}^*}$, 
\[
    \mathrm{dist}(\bm W(t), \widehat{\mathcal{W}}_{\bar{\bm \sigma}^*}) \ge \mathrm{dist}( \widehat{\mathcal{W}}_{\bm \sigma^*},\widehat{\mathcal{W}}_{\bar{\bm \sigma}^*}) - \mathrm{dist}(\bm W(t), \widehat{\mathcal{W}}_{\bm \sigma^*}) 
    \ge \left(1-\frac{\sqrt{2}}{3}\right)\delta_\sigma > \mathrm{dist}(\bm W(t), \widehat{\mathcal{W}}_{\bm \sigma^*}).
\]
Consequently, we obtain $\mathrm{dist}(\bmw(t),\mathcal{W}_{G}) = \mathrm{dist}(\bmw(t),\widehat{\mathcal{W}}_{\bm \sigma^*})$.
This, together with \eqref{eq1:lem counter 1} and \eqref{eq:final count1}, implies
$ 
\|\nabla G(\bm W(t))\|_F =	 2\sqrt{2}|t|^{3} \le  2^{3/4} \mathrm{dist}^{\frac{3}{2}}(\bmw(t),\widehat{\mathcal{W}}_{\bm \sigma^*})=  2^{3/4} \mathrm{dist}^{\frac{3}{2}}(\bmw(t),\mathcal{W}_{G}). 
$ 
Obviously, the error bound does not hold when $t \to 0$.  
\end{proof}

\subsection{The Case $L \ge 3$}
\label{subsec:A2}
We first present the following lemma to derive the equivalent condition to \eqref{eq:AS y}.

\begin{lemma}\label{lem:phi}
When $L \ge 3$, \eqref{eq:AS y} holds if and only if   $\varphi'(\sigma^*) \neq 0$ for all $ \sigma^* \in \mathcal{Y}\setminus\{0\}$, where  $\cal Y$ and $\varphi(\cdot)$ are defined in \eqref{set:Y} and  \eqref{eq:phi}, respectively.
When $L=2$, we have $\varphi'(\sigma^*) \neq 0$.
\end{lemma} 
\begin{proof}
First, consider the case where $L\ge 3$.
Suppose that there exists $ \sigma^* \in \mathcal{Y}\setminus\{0\}$ such that $\varphi'(\sigma^*) = 0$. Note that $\varphi'(\sigma^*) = 0$ is equivalent to
\(\varphi'(\sigma^*)= L\sigma^{*(L-1)}/\sqrt{\lambda} + \sqrt{\lambda}(2-L)\sigma^{*(1-L)} = 0.\)
Using $\sigma^* \in \mathcal{Y}$, there exists $j \in [d_{\min}]$ such that $  \sigma^{*(2L-1)} + \lambda \sigma^* = \sqrt{\lambda} \sigma^{*(L-1)}y_j$. 
    Combining the above two equations, we obtain 
    \begin{align}\label{eq:L3}
      y_j= \left( \left( \frac{L-2}{L} \right)^{\frac{L}{2(L-1)}} + \left( \frac{L}{L-2} \right)^{\frac{L-2}{2(L-1)}} \right)  \lambda^{\frac{1}{2(L-1)}}.
    \end{align}
    This implies that \eqref{eq:AS y} does not hold. 
    
    Conversely, suppose that \eqref{eq:AS y} does not hold. This implies that there exists $j \in [d_{\min}]$ such that
    \begin{align}\label{eq:lambda value}
    \lambda = y_j^{2(L-1)} \left(
    \left(\frac{L-2}{L}\right)^{\frac{L}{2(L-1)}}
    + \left(\frac{L}{L-2}\right)^{\frac{L-2}{2(L-1)}}
    \right)^{-{2(L-1)}}.
    \end{align}
    Since \(\varphi(x)=x^L/\sqrt{\lambda}+\sqrt{\lambda}x^{2-L}\), we have $\varphi'(x)= Lx^{L-1}/{\sqrt{\lambda}}
    +\sqrt{\lambda}(2-L)x^{1-L}.$ Hence \(\varphi'(x)=0\) implies \(Lx^{2L-2}=\lambda(L-2)\), and the positive root is $ x^*=\left(\frac{\lambda(L-2)}{L}\right)^{\frac{1}{2(L-1)}}.$ Substituting this into \(\varphi\) gives
    \[
    \varphi(x^*)=
    \left[
    \left(\frac{L-2}{L}\right)^{\frac{L}{2(L-1)}}
    +
    \left(\frac{L}{L-2}\right)^{\frac{L-2}{2(L-1)}}
    \right]\lambda^{\frac{1}{2(L-1)}}.
    \]
    This, together with \eqref{eq:lambda value}, yields that \(x^{2L-1}-\sqrt{\lambda}y_j x^{L-1}+\lambda x =0\)
    has this positive root \(x^*\) and \(\varphi'(x^*) = 0\).
    Using the definition of $\mathcal{Y}$, we have $x^* \in \mathcal{Y}$. 

    When $ L = 2 $, we compute
    $\varphi'(\sigma^*) = 2 \sigma^*/\sqrt{\lambda} > 0$ for all $\sigma^* \in \mathcal{Y}\setminus\{ 0\}.$ So, $\varphi'(\sigma^*) \neq 0$ holds trivially.
\end{proof}

\begin{lemma}\label{lem:counter 2}
Suppose that $L \ge 3$ and \eqref{eq:AS y} does not hold. There exists $\bm \sigma^* \in \cal A$ such that the error bound for the critical point set $\widehat{\mathcal{W}}_{\bm \sigma^*}$ of Problem \eqref{eq:G} fails to hold.
\end{lemma}
\begin{proof}
According to \Cref{lem:phi} and the fact that \eqref{eq:AS y} does not hold, there exists $\bm \sigma^* \in \cal A$ such that $\varphi^\prime(\sigma_i^*) = 0$ for some $i \in [d_{\min}]$. This implies $ f(\sigma^*_i) := (\sigma^*_i)^{2L-1} - \sqrt{\lambda} y_i (\sigma^*_i)^{L-1} + \lambda \sigma^*_i = 0$ and $\varphi(\sigma^*_i) = y_i$. Note that $f(x) = \sqrt{\lambda} x^{L-1} \varphi(x) - \sqrt{\lambda} y_i x^{L-1}$ and we compute $f^\prime(\sigma_i^*) = \sqrt{\lambda} (\sigma_i^*)^{L-1} \varphi'(\sigma_i^*) = 0$. 
    Now, we define $\bmw(t)=(\bmw_1(t),\cdots,\bmw_L(t))$  as follow:
    \begin{align*}
    \begin{cases}
        &\bm W_1(t) = \bm Q_2\bm \Sigma_1(t) \mathrm{BlkD}\left( \bm O_1,\dots,\bm O_{p_{Y}},\bm O_{p_{Y}+1} \right),\ 
        \bm W_l(t) = \bm Q_{l+1} \bm \Sigma_l(t) \bm Q_l^T,\ l=2,\dots,L-1,  \\
        &\bm W_L(t) =  \mathrm{BlkD}\left( \bm O_1^T,\dots,\bm O_{p_{Y}}^T,\widehat{\bm O}_{p_{Y}+1}^T \right)\bm \Sigma_L(t) \bm Q_L^T,\ \bm \Sigma_l(t) = \mathrm{BlkD}\left(\mathrm{diag}(\bm \sigma^* +t \bm e_i), \bm 0\right) \in \R^{d_l\times d_{l-1}},\ \forall l \in [L], \\
        &\bm O_i \in \mathcal{O}^{h_i}, \forall i \in [p_{Y}],\ \bm O_{p_{Y}+1} \in \mathcal{O}^{d_0 - r_{Y}}, \widehat{\bm O}_{p_{Y}+1} \in \mathcal{O}^{d_L - r_{Y}},\ \bm Q_l \in \mathcal{O}^{d_{l-1}},\ l=2,\dots,L. 
    \end{cases}
    \end{align*}
It follows from \Cref{prop:opti G} that $\bmw(0) \in \mathcal{W}_G$. 
Therefore, we obtain for all $l \in [L]$,  
\[
 \|\nabla_{l} G(\bmw(t))\|_F \overset{\eqref{eq:grad G}}{=} 2\left|(\sigma^*_i+t)^{2L-1}- \sqrt{\lambda}y_i(\sigma^*_i+t)^{L-1}+\lambda (\sigma^*_i+t)\right|= 2\left|f(\sigma^*_i+t)\right|.
\]
Applying the Taylor expansion to $f(\sigma^*_i+t)$ at $\sigma^*_i$, together with $f(\sigma^*_i)=0$ and $ f^\prime(\sigma^*_i)=0$, yields that when $t \to 0$, 
$
    \|\nabla G(\bmw(t))\|_F = O(t^2).
$
We also note that
\(
    \mathrm{dist}^2(\bmw(t),\widehat{\mathcal{W}}_{\bm \sigma^*}) =  \|\bmw_1(t)-\bmw_1^*\|_F^2+\cdots+\|\bmw_L(t)-\bmw_L^*\|_F^2 \ge Lt^2,
\)
where the inequality follows from Weyl's inequality. Using the same argument in \Cref{lem:counter 1}, we conclude that $\mathrm{dist}(\bmw(t),\widehat{\mathcal{W}}_{\bm \sigma^*}) = \mathrm{dist}(\bmw(t),\mathcal{W}_{G})$ when $t$ is sufficiently small. Then we have $\left\|\nabla G\left(\bm W(t)\right)\right\|_F =O\left(\mathrm{dist}^2(\bmw(t),\mathcal{W}_{G })\right)$, which implies that the error bound fails to hold.
\end{proof}

\begin{remark}
    Under Assumption~\ref{AS:1}, it follows from \Cref{thm:eb} that Assumption~\ref{AS:2} is a sufficient condition for the error bound to hold for the critical point set $\mathcal{W}_G$. According to \Cref{lem:counter 1} and \Cref{lem:counter 2}, we conclude that Assumption~\ref{AS:2} is also a necessary condition. Therefore, we establish the necessary and sufficient condition under which the error bound holds. 
\end{remark}

\section{Supplementary Results and Proofs in \Cref{sec:main} and \Cref{sec:proof}}\label{app:sec A}  

\subsection{\bf Proof of \Cref{coro:PL QG}}\label{app:PL QG}
First, we establish the separation property (see, e.g., \cite{li2018calculus, luo1993}) for Problem~\eqref{eq:F}, i.e., for any $\bmw^* \in \mathcal{W}_F$, there exists a neighborhood $\mathcal{U}$ of $\bmw^*$ such that $F(\bm{W}) = F(\bmw^*)$ for all $\bm{W} \in \mathcal{W}_F \cap \mathcal{U}$. It is important to emphasize that this property is not restricted to the specific function in \eqref{eq:F}. More generally, it holds for any $\mathcal{C}^2$ function in the neighborhood of a critical point where a local error bound is satisfied. Applying \cite[Corollary~1]{behling2013effect} to $H(\bm{W}):=\nabla F(\bm{W})$ (which is $\mathcal{C}^1$ since $F\in \mathcal{C}^2$), the critical set $\mathcal{W}_F$ is locally a differentiable manifold around $\bmw^*$. After possibly shrinking a neighborhood $\mathcal{U}$ of $\bmw^*$, the set $\mathcal{W}_F \cap \mathcal{U}$ is a connected manifold patch. Since $\nabla F = \bm{0}$ on $\mathcal{W}_F$, the restriction of $F$ to this manifold has vanishing differential; hence, $F$ is constant on $\mathcal{W}_F \cap \mathcal{U}$. In particular, all critical points sufficiently close to $\bmw^*$ share the same function value. 

Choose $\epsilon_2 \in (0, \epsilon_1]$ such that the ball $\mathcal{N} := \{ \bm{W} : \|\bm{W} - \bmw^*\|_F \le 2\epsilon_2 \}$ is contained within $\mathcal{U}$, where $\epsilon_1$ is defined in \Cref{thm:eb}. For any $\bmw$ satisfying $\|\bmw - \bmw^*\|_F \le \epsilon_2$, let $\overline{\bmw}$ be a projection of $\bmw$ onto $\mathcal{W}_F$ such that $\|\bmw - \overline{\bmw}\|_F = \mathrm{dist}(\bmw, \mathcal{W}_F)$. By the triangle inequality and the optimality of $\overline{\bmw}$, we have
\begin{equation*}
    \|\bmw^* - \overline{\bmw}\|_F \le \|\bmw^* - \bmw\|_F + \|\bmw - \overline{\bmw}\|_F \le 2 \|\bmw - \bmw^*\|_F \le 2\epsilon_2.
\end{equation*}
Since $\overline{\bmw} \in \mathcal{W}_F \cap \mathcal{U}$, it follows from the separation property that $F(\bmw^*) = F(\overline{\bmw})$. Because $F \in \mathcal{C}^2$, its gradient $\nabla F$ is Lipschitz continuous on the compact set $\mathcal{N}$ with constant $L_F$. Since both $\bmw$ and $\overline{\bmw}$ lie in $\mathcal{N}$, using the proof in \citep[Lemma 5.7]{beck2017first} and noting $\nabla F(\overline{\bmw}) = \bm{0}$ yields
\begin{equation}\label{eq:descent_result}
    |F(\bmw) - F(\overline{\bmw})| \le   |\langle \nabla F(\overline{\bW}),\bW-\bar \bW\rangle | + \frac{L_F}{2} \|\bmw - \overline{\bmw}\|^2_F = \frac{L_F}{2} \|\bmw - \overline{\bmw}\|^2_F .
\end{equation}
Substituting the error bound from \Cref{thm:eb} into \eqref{eq:descent_result}, we obtain
\begin{equation*}
    |F(\bmw) - F(\bmw^*)| = |F(\bmw) - F(\overline{\bmw})| \le \frac{L_F}{2} \mathrm{dist}^2(\bmw, \mathcal{W}_F) \le \frac{L_F \kappa_1^2}{2} \|\nabla F(\bmw)\|_F^2.
\end{equation*}
This establishes the P\L\ inequality in (i) with $\mu_1 = 2 / (L_F \kappa_1^2)$.

In addition, if $\bm W^*$ is a local minimizer, the quadratic growth condition holds due to (i) and \cite[Proposition 2.2]{rebjock2024fast}.
\hfill$\square$
\subsection{Proof of \Cref{prop:set W}}\label{subsec app:set W} 

(i) Obviously, the ``only if'' direction is trivial since $\bm W \in \mathcal{W}_{\bm \sigma, \bm \Pi}$ and $\bm W^\prime \in \mathcal{W}_{\bm \sigma^\prime, \bm \Pi^\prime}$ share the same non-increasing singular values. It remains to prove the ``if'' direction. Suppose that $(\bm \sigma,\bm \Pi)\in \mathcal{B}$ and  $(\bm \sigma^\prime,\bm \Pi^\prime) \in \mathcal{B}$ satisfy $\bm \sigma = \bm \sigma^\prime$. There exist $\bm a, \bm a^\prime \in \mathcal{A}$ such that
    \begin{align}\label{eq:defahata}
        \bm a = \bm \Pi^T \bm \sigma,\ \bm a^\prime = \bm \Pi^{\prime T} \bm \sigma^\prime.
    \end{align}
    This, together with $\bm \sigma = \bm \sigma^\prime$, implies that $\bm a, \bm a^\prime$ have the same positive elements but in a different order. Consider the following equation from $\mathcal{A}$: 
        $x_i^{2L-1} - \sqrt{\lambda}y_i x_i^{L-1} + \lambda x_i = 0,\ x_i \geq 0,\ \forall i \in [d_{\min}].$
    Note that if $x_i = x_j>0$, we have $y_i=y_j$, which implies that if $y_i \neq y_j$, we have $x_i \neq x_j$ when $x_i, x_j>0$. 
    Since $\bm \sigma = \bm \sigma'$ and each element of $\bm a$ and $\bm a'$ is obtained by solving the above equation, along with the fact that different $y_i$'s correspond to equations with no common positive root and the partition of $(y_1, \ldots, y_{d_{\min}})$ in \eqref{eq:SY1}, yields that the elements of $\bm a$ and $\bm a'$ in each partition of the form $\bm Y$ differ only in their order. Therefore,there exist $\bm P_i \in \mathcal{P}^{h_i}$ for all $i \in [p_Y]$ and $\bm P_{p_Y+1} \in \mathcal{P}^{d_{\min} -r_{Y}}$ such that $\mathrm{diag}(\bm a) = \mathrm{BlkD}(\bm P_1,\ldots,\bm P_{p_Y},\bm P_{p_Y+1})\mathrm{diag}(\bm a^\prime)\mathrm{BlkD}(\bm P_1^T,\ldots,\bm P_{p_Y}^T,\bm P_{p_Y+1}^T). $ 
    Substituting \eqref{eq:defahata} into the above equality, together with $\bm \sigma = \bm \sigma^\prime$, yields 
        \begin{align}\label{eq1:prop set W}
        \bm \Pi^T\mathrm{diag}(\bm \sigma)\bm \Pi 
         =   \mathrm{BlkD}(\bm P_1,\ldots,\bm P_{p_Y},\bm P_{p_Y+1})\bm \Pi^{\prime T}\mathrm{diag}(\bm \sigma )\bm \Pi^\prime\mathrm{BlkD}(\bm P_1^T,\ldots,\bm P_{p_Y}	^T,\bm P_{p_Y+1}^T).
    \end{align}
    Let $(\bm W_1,\cdots,\bm W_L)\in \mathcal{W}_{\bm \sigma,\bm \Pi}$ be arbitrary. For ease of exposition, let 
    \begin{align}
    \label{eq:defql}
        \tilde{\bm Q}_l :=\bm Q_l\mathrm{BlkD}(\bm \Pi,\bm I)
         \mathrm{BlkD}\left(\bm P_1,\ldots,\bm P_{p_Y},\bm P_{p_Y+1},\bm I \right)\mathrm{BlkD}(\bm \Pi^{\prime T},\bm I) \in \mathcal{O}^{d_{l-1}},\ \forall l \in \{2,\dots,L\}. 
    \end{align}
    Moreover, we have $\bm \Sigma_l \mathrm{BlkD}(\bm \Pi, \bm I) = \mathrm{BlkD}(\mathrm{diag}(\bm \sigma)\bm \Pi, \bm 0),\ \forall l \in [L]. $
    Then, we have 
    \begin{align*}
         \bm W_1 &= \bm Q_2\bm \Sigma_1 \mathrm{BlkD}\left(\bm \Pi, \bm I  \right)\mathrm{BlkD}\left( \bm O_1,\dots,\bm O_{p_Y},\bm O_{p_Y+1} \right)\\
        &= \tilde{\bm Q}_2 \bm \Sigma_1 \mathrm{BlkD}\left( \bm \Pi^\prime, \bm I  \right)\mathrm{BlkD}(\bm P_1^T\bm O_1,\ldots,\bm P_{p_Y}^T\bm O_{p_{Y}},\mathrm{BlkD}(\bm P_{p_Y+1}^T,\bm I_{d_0-d_{\min}})\bm O_{p_{Y}+1}),  
    \end{align*}
    where the second equality uses \eqref{eq1:prop set W} and \eqref{eq:defql}. Using the same argument, we have
    \begin{align*}
    & \bm W_L = \mathrm{BlkD}\left( \bm O_1^T\bm P_1,\dots,\bm O_{p_{Y}}^T\bm P_Y,\widehat{\bm O}_{p_{Y}+1}^T\mathrm{BlkD}(\bm P_{p_Y+1}, \bm I) \right)\mathrm{BlkD}\left(\bm \Pi^{\prime T}, \bm I \right)\bm \Sigma_L \tilde{\bm Q}_{L}^T,\ \bm W_l = \tilde{\bm Q}_{l+1} \bm \Sigma_l \tilde{\bm Q}_l^T,\ l =2,\dots, L-1. 
    \end{align*}
Therefore, we obtain that $(\bm W_1,\ldots,\bm W_L)\in \mathcal{W}_{\bm \sigma^\prime, \bm \Pi^\prime}$ and thus $\mathcal{W}_{\bm\sigma,\bm\Pi} \subseteq  \mathcal{W}_{\bm\sigma^\prime, \bm\Pi^\prime} $. Applying the same argument, we also have $\mathcal{W}_{\bm\sigma^\prime, \bm\Pi^\prime} \subseteq \mathcal{W}_{\bm\sigma,\bm\Pi}$. Therefore, we have $\mathcal{W}_{\bm\sigma,\bm\Pi} = \mathcal{W}_{\bm\sigma^\prime, \bm\Pi^\prime}$.\\
(ii)  Using Mirsky's inequality (see \Cref{lem:mirsky}),  for any $\bm W \in \mathcal{W}_{\bm \sigma, \bm \Pi}$ and $\bm W^\prime \in \mathcal{W}_{\bm \sigma^\prime, \bm \Pi^\prime}$, we have
$ \|\bm W- \bm W^\prime\|_F \geq \|\bm \sigma -\bm \sigma^\prime\|_2 \overset{\eqref{set:Y}}{\geq} \delta_{\sigma}.$
	 \hfill$\square$

\subsection{Auxiliary Lemma for Proving \Cref{lem:twoinequality}}\label{app:lem commute}

\begin{lemma}\label{lem:commute}
    Let \(a > 0\) be a constant, \(\bm \Sigma \in \R^{n\times n} \) be a diagonal matrix, and \(\bm Q \in \mathcal{O}^n\) be an orthogonal matrix. Then, if \(\|\bm \Sigma -  a \bm I\| \leq a/2\), we have
    $
    \|\bm Q \bm \Sigma^2 - \bm \Sigma^2\bm Q\|_F \geq a \|\bm Q \bm \Sigma - \bm \Sigma\bm Q\|_F.
    $
\end{lemma}
\begin{proof}
For ease of exposition, let $\bm \Delta := \bm \Sigma - a\bm I$. We compute
\begin{align*}
\|\bm Q \bm \Sigma^2 - \bm \Sigma^2\bm Q\|_F & = \|\bm Q \left(\bm \Delta + a \bm I \right)^2 - \left(\bm \Delta + a \bm I \right)^2\bm Q\|_F = \left\| \bm Q \bm \Delta^2 - \bm \Delta^2 \bm Q + 2a\left(\bm Q\bm \Delta - \bm \Delta\bm Q \right) \right\|_F \\
& \ge 2a \left\|  \bm Q\bm \Delta - \bm \Delta\bm Q  \right\|_F - \left\| \bm Q \bm \Delta^2 - \bm \Delta^2 \bm Q \right\|_F \ge 2\left( a - \|\bm \Delta\|\right) \left\|  \bm Q\bm \Sigma - \bm \Sigma \bm Q  \right\|_F  \ge a\left\|  \bm Q\bm \Sigma - \bm \Sigma \bm Q  \right\|_F,
\end{align*}
where the second inequality follows from $\bm Q\bm \Delta - \bm \Delta\bm Q = \bm Q\bm \Sigma - \bm \Sigma \bm Q$ and $\| \bm Q \bm \Delta^2 - \bm \Delta^2 \bm Q \|_F = \|\bm \Delta(\bm \Delta\bm Q - \bm Q\bm \Delta) + (\bm \Delta\bm Q - \bm Q\bm \Delta)\bm \Delta \|_F \le 2\|\bm \Delta\| \|\bm Q\bm \Sigma - \bm \Sigma \bm Q\|_F$, and the last inequality uses $\|\bm \Delta\| \le a/2$.   
\end{proof}

\subsection{Auxiliary Lemma and Proof for Proving \Cref{prop:singular control}}\label{app:lem deri}

\begin{lemma}\label{lem:derivative bounds}
Let $\bm{\sigma}^* \in \mathcal{A}_{\mathrm{sort}} \setminus \{\bm{0}\}$ and $i \in [r_\sigma]$ be arbitrary. Suppose that  
$
|x - \sigma_i^*| \le \min\left\{\delta_1,\delta_2,\frac{\delta_\sigma}{3}\right\}, 
$
where $\delta_1$ and $\delta_2$ are defined in \eqref{eq:delta1} and \eqref{eq:delta2}, respectively. Then we have
\begin{align}
\left| \varphi(x) - \varphi(\sigma_i^*) \right| \le \frac{\delta_y}{3} \text{ and }\ |\varphi'(x)-\varphi'(\sigma^*_i)|\le \frac{|\varphi'(\sigma_i^*)|}{2}.
\label{eq:result_phi}
\end{align}
\end{lemma}
\begin{proof}
According to the definition of $\delta_\sigma$ in \eqref{set:Y}, if $|x - \sigma_i^*|\le \frac{\delta_\sigma}{3}$, then we have $   \frac{2\sigma_{\min}^*}{3}\le |x|\le \frac{4\sigma_{\max}^*}{3}.$
Consequently, we obtain that
\begin{align*}
    |\varphi'(x)|
    &=\Bigl|\frac{L}{\sqrt{\lambda}}x^{L-1}+\sqrt{\lambda}(2-L)x^{1-L}\Bigr| \le \frac{L}{\sqrt{\lambda}}\left(\frac{4\sigma_{\max}^*}{3}\right)^{L-1}
      +(L-2)\sqrt{\lambda}\left(\frac{2\sigma_{\min}^*}{3}\right)^{1-L},\\[0.8ex]
    |\varphi''(x)|
    &=\Bigl|\frac{L(L-1)}{\sqrt{\lambda}}x^{L-2}+\sqrt{\lambda}(2-L)(1-L)x^{-L}\Bigr| \le \frac{L(L-1)}{\sqrt{\lambda}}\left(\frac{4\sigma_{\max}^*}{3}\right)^{L-2}
      +\sqrt{\lambda}(L-2)(L-1)\left(\frac{2\sigma_{\min}^*}{3}\right)^{-L}.
\end{align*}
Finally, by the mean value theorem, for all $x$ satisfying 
$|x - \sigma_i^*| \le \min\{\delta_1,\delta_2,\frac{\delta_\sigma}{3}\}$,
we obtain \eqref{eq:result_phi}. 
\end{proof}

\subsubsection{Proof of \eqref{eqi11:prop sing} and \eqref{eqi12:prop sing}}
\label{sec:comlementary1}

Using the form of $\bm \Sigma_l$ in \eqref{eq:SVD Wl} and \eqref{eq:prop singular} for $L=2$ (resp., \eqref{eq:singular 2} for $L \ge 3$), we have for each $l \in [L]$, 
\begin{align}
& \left\|\hat{\bm \Sigma}_l - \bm \Sigma_l\right\|_F = \left\| \bm \Sigma_{l}^{(p+1)} \right\|_F \le c_3\sqrt{\min\{d_l,d_{l-1}\}} \|\nabla G(\bm W)\|_F \le c_3\sqrt{d_{\max}}\|\nabla G(\bm W)\|_F,\label{eq1:prop sing} \\
& \left\|(\hat{\bm \Sigma}_l\hat{\bm \Sigma}_l^T)^{L-1}\hat{\bm \Sigma}_l - (\bm \Sigma_l\bm \Sigma_l^T)^{L-1}\bm \Sigma_l \right\|_F  = \left\|\left(\bm \Sigma_{l}^{(p+1)}\bm \Sigma_{l}^{(p+1)^T}\right)^{L-1}\bm \Sigma_{l}^{(p+1)} \right\|_F \notag \\
&\qquad\qquad\qquad\qquad\qquad\qquad\ \ \overset{(\ref{eq:sigmalowerbound}, \ref{eq1:prop sing})}{\le}  c_3\sqrt{d_{\max}}\left( \frac{3\sigma_{\max}^*}{2} \right)^{2(L-1)}\|\nabla G(\bm W)\|_F. \label{eq2:prop sing}
\end{align}
We first bound 
\begin{align}
& \left\|(\hat{\bm \Sigma}_L\hat{\bm \Sigma}_L^T)^{L-1}\hat{\bm \Sigma}_L + \lambda \hat{\bm \Sigma}_L - \sqrt{\lambda}\bm \Psi \mathrm{BlkD}\left(\bm B_1, \dots, \bm B_{p}, \bm 0\right)  \right\|_F \notag\\
\le &  \left\|(\bm \Sigma_L\bm \Sigma_L^T)^{L-1}\bm \Sigma_L + \lambda \bm \Sigma_L - \sqrt{\lambda}\bm \Psi \mathrm{BlkD}\left(\bm B_1, \dots, \bm B_{p}, \bm B_{p+1}\right) \right\|_F + \lambda\|\hat{\bm \Sigma}_L - \bm \Sigma_L\|_F  \notag\\
\ & + \left\|(\hat{\bm \Sigma}_L\hat{\bm \Sigma}_L^T)^{L-1}\hat{\bm \Sigma}_L - (\bm \Sigma_L\bm \Sigma_L^T)^{L-1}\bm \Sigma_L \right\|_F + \sqrt{\lambda}\|\bm \Psi\|\|\bm B_{p+1}\|_F \notag\\
\le& \left( c_2 + \left(\left( \frac{3\sigma_{\max}^*}{2} \right)^{2(L-1)} + \lambda \right) c_3\sqrt{d_{\max}} \right) \left\| \nabla G(\bm W) \right\|_F + \sqrt{\lambda}y_1\|\bm B_{p+1}\|_F  
\le  \eta_3 \left\| \nabla G(\bm W) \right\|_F,  \label{eq3:prop sing}
\end{align}
where the second inequality follows from \eqref{eq:B}, \eqref{eq:twoinequality1}, 
\eqref{eq1:prop sing}, and \eqref{eq2:prop sing}, and the last  uses \eqref{eq:sigmalowerbound} and \eqref{eq:B}.
Using \eqref{eq:B} and \eqref{eq0:prop sing}, we have
\begin{align*}
& \sum_{i=1}^{r_{\sigma}} \left\| \left( \sigma_i^{2L-1}(\bm W_L) + \lambda \sigma_i(\bm W_L)\right)\bm e_i - \sqrt{\lambda}\sigma_i^{L-1}(\bm W_L) \bm \Psi \hat{\bm T}\bm e_i \right\|^2 \notag\\
\le & \left\|(\hat{\bm \Sigma}_L\hat{\bm \Sigma}_L^T)^{L-1}\hat{\bm \Sigma}_L + \lambda \hat{\bm \Sigma}_L - \sqrt{\lambda}\bm \Psi \hat{\bm T}\mathrm{BlkD}\left((\bm \Sigma_L^{(1)})^{L-1}, \dots, (\bm \Sigma_L^{(p)})^{L-1}, \bm 0\right)  \right\|_F^2\notag \\
= &\left\|(\hat{\bm \Sigma}_L\hat{\bm \Sigma}_L^T)^{L-1}\hat{\bm \Sigma}_L + \lambda \hat{\bm \Sigma}_L - \sqrt{\lambda}\bm \Psi \mathrm{BlkD}\left(\bm B_1, \dots, \bm B_{p}, \bm 0\right)  \right\|_F^2 
\end{align*}
Dividing the above inequality by $\sqrt{\lambda}\,\sigma_{r_{\sigma}}^{L-1}(\bm W_L)$ and using \eqref{eq3:prop sing} and the auxiliary function $\varphi$ in \eqref{eq:phi} yields
\begin{align}\label{eq5:prop sing}
\sum_{i=1}^{r_{\sigma}} \left\| \varphi(\sigma_i(\bm W_L))\bm e_i - \bm \Psi \hat{\bm T} \bm e_i\right\|^2 &  \le   \frac{\eta_3^2 }{\lambda\sigma_{r_{\sigma}}^{2(L-1)}(\bm W_L)}\|\nabla G(\bm W)\|_F^2  \overset{\eqref{eq:sigmalowerbound}}{\le}  \frac{4^{L-1}\eta_3^2}{\lambda(\sigma_{\min}^{*})^ {2L-2}} \|\nabla G(\bm W)\|_F^2,
\end{align}
which is just the desired inequality \eqref{eqi11:prop sing}. By applying the same argument used to derive \eqref{eq3:prop sing} using \eqref{eq:twoinequality2}, we have 
\begin{align}\label{eq4:prop sing}
\left\| ( \hat{\bm \Sigma}_1 \hat{\bm \Sigma}_1^T)^{L-1} \hat{\bm \Sigma}_1+\lambda \hat{\bm \Sigma}_1-\sqrt{\lambda}\mathrm{BlkD}\left(\bm A_1,\dots,\bm A_{p},\bm 0\right)\bm \Psi \right\|_F  \le \eta_3 \left\| \nabla G(\bm W) \right\|_F. 
\end{align}
Now, we claim
\begin{align}
    \left\| \left( \hat{\bm \Sigma}_1 \hat{\bm \Sigma}_1^T \right)^{L-1} \hat{\bm \Sigma}_1 
        + \lambda \hat{\bm \Sigma}_1 
        - \sqrt{\lambda}\, \mathrm{BlkD}\left( (\bm \Sigma_{1}^{(1)})^{L-1}, \dots, (\bm \Sigma_{1}^{(p)})^{L-1}, \bm 0 \right) \bm \Psi \hat{\bm T} \right\|_F \le \eta_4 \left\| \nabla G(\bm W) \right\|_F. 
    \label{eq0:prop_sing}
\end{align}
Applying a similar argument as in \eqref{eq5:prop sing} to the rows of the above inequality yields \eqref{eqi12:prop sing}.

In the rest of this proof, we show that \eqref{eq0:prop_sing} holds.
For each $ i \in [p] $, we define
$ \hat{\bm T}_i := \prod_{l=2}^L \bm{T}_l^{(i)} $. For each $i \in [p]$, we have  
\begin{align}
& \quad \left\| \hat{\bm T}_i (\bm \Sigma_{1}^{(i)})^{2L-1} 
- (\bm \Sigma_{1}^{(i)})^{2L-1} \hat{\bm T}_i \right\|_F  = \left\|\left(\prod_{l=2}^L\bm T_{l}^{(i)}\right)(\bm \Sigma_{1}^{(i)})^{2L-1} 
- (\bm \Sigma_{1}^{(i)})^{2L-1}\prod_{l=2}^L\bm T_{l}^{(i)} \right\|_F \notag\\
 & \le \left\|\left(\prod_{l=2}^{L-1}\bm T_{l}^{(i)}\right) \left(\bm T_{L}^{(i)}(\bm\Sigma_{1}^{(i)})^{2L-1}-(\bm \Sigma_{1}^{(i)})^{2L-1}\bm T_{L}^{(i)}\right)\right\|_F \notag\\
 &\quad + \left\|\left(\prod_{l=2}^{L-2}\bm T_{l}^{(i)}\right) \left(\bm T_{L-1}^{(i)}(\bm\Sigma_{1}^{(i)})^{2L-1}-(\bm \Sigma_{1}^{(i)})^{2L-1}\bm T_{L-1}^{(i)} \right) \bm T_{L}^{(i)}\right\|_F \notag \\
 &  \quad  + \cdots +\left\| \left(\bm T_{2}^{(i)}(\bm\Sigma_{1}^{(i)})^{2L-1}-(\bm \Sigma_{1}^{(i)})^{2L-1}\bm T_{2}^{(i)}\right)\left(\prod_{l=3}^{L}\bm T_{l}^{(i)}\right)\right\|_F \notag \\
 & \le \sum_{l=2}^L\sum_{j=1}^{2L-1} \left\|(\bm\Sigma_{1}^{(i)})^{j-1}\left(\bm T_{l}^{(i)}\bm\Sigma_{1}^{(i)}-\bm \Sigma_{1}^{(i)}\bm T_{l}^{(i)}\right)(\bm\Sigma_{1}^{(i)})^{2L-j-1} \right\|_F \notag \\
 & \le \frac{ \eta_1 (2L - 1)L}{\sigma^*_{\min}} 
            \left( \frac{3 \sigma^*_{\max}}{2} \right)^{2L - 2}\|\nabla G(\bm W)\|_F,\label{eq:complementary1}
\end{align}
where the last inequality follows from \eqref{eq:sigmalowerbound} and \eqref{eq4:lem two}. Similarly, we have
\begin{equation}
\left\| \hat{\bm T}_i \bm \Sigma_{1}^{(i)} 
- \bm \Sigma_{1}^{(i)}\hat{\bm T}_i \right\|_F \le \frac{ \eta_1 L}{\sigma^*_{\min}} \|\nabla G(\bm W)\|_F.
\label{eq:complementary2}
\end{equation} Then we have
\begin{align}
   & \quad\left\| \left( \hat{\bm \Sigma}_1 \hat{\bm \Sigma}_1^T \right)^{L-1} \hat{\bm \Sigma}_1 
        + \lambda \hat{\bm \Sigma}_1 
        - \sqrt{\lambda}\, \mathrm{BlkD}\left( (\bm \Sigma_{1}^{(1)})^{L-1}, \dots, (\bm \Sigma_{1}^{(p)})^{L-1}, \bm 0 \right) \bm \Psi \hat{\bm T} \right\|_F \notag \\
    &= \left\| 
        \mathrm{BlkD}\left( \hat{\bm T}_1, \dots, \hat{\bm T}_p, \bm I_{d_1 - r_{\sigma}} \right) 
        \left( \left( \hat{\bm \Sigma}_1 \hat{\bm \Sigma}_1^T \right)^{L-1} \hat{\bm \Sigma}_1 
            + \lambda \hat{\bm \Sigma}_1 \right) \hat{\bm T}^T  
        - \sqrt{\lambda}\, \mathrm{BlkD}\left( \bm A_1, \dots, \bm A_p, \bm 0 \right) \bm \Psi 
    \right\|_F \notag \\
    &\leq \left\| 
        \left( \hat{\bm \Sigma}_1 \hat{\bm \Sigma}_1^T \right)^{L-1} \hat{\bm \Sigma}_1 
        + \lambda \hat{\bm \Sigma}_1 
        - \sqrt{\lambda}\, \mathrm{BlkD}\left( \bm A_1, \dots, \bm A_p, \bm 0 \right) \bm \Psi 
    \right\|_F \notag \\
    & \quad + \sum_{i=1}^p \left\| \hat{\bm T}_i (\bm \Sigma_{1}^{(i)})^{2L-1} \hat{\bm T}_i^T 
        - (\bm \Sigma_{1}^{(i)})^{2L-1} \right\|_F 
     +\lambda \sum_{i=1}^p \left\| \hat{\bm T}_i \bm \Sigma_{1}^{(i)} \hat{\bm T}_i^T 
        - \bm \Sigma_{1}^{(i)} \right\|_F \notag \\
    & \leq \eta_3 \left\| \nabla G(\bm W) \right\|_F 
        + \sum_{i=1}^p \left\| \hat{\bm T}_i (\bm \Sigma_{1}^{(i)})^{2L-1} 
            - (\bm \Sigma_{1}^{(i)})^{2L-1} \hat{\bm T}_i \right\|_F + \lambda\sum_{i=1}^p \left\| \hat{\bm T}_i \bm \Sigma_{1}^{(i)} 
        - \bm \Sigma_{1}^{(i)}\hat{\bm T}_i \right\|_F \notag \\
    & \leq \left( \eta_3 
        + \frac{p \eta_1 (2L - 1)L}{\sigma^*_{\min}} 
            \left( \frac{3 \sigma^*_{\max}}{2} \right)^{2L - 2}  
        + \frac{\lambda p \eta_1 L}{\sigma^*_{\min}} 
    \right) \left\| \nabla G(\bm W) \right\|_F = \eta_4 \left\| \nabla G(\bm W) \right\|_F, 
    \notag
\end{align}
where the first equality follows from the definition of $\bm  A_i$ for each $i \in [p]$ in \eqref{eq:A}, the second inequality follows from \eqref{eq4:prop sing}, and the last inequality follows from \eqref{eq4:lem two}, \eqref{eq:complementary1}, and  \eqref{eq:complementary2}.
 \hfill$\square$ 

\subsection{Auxiliary Lemma for Proving \Cref{thm:eb G}}\label{app:b4}

\begin{lemma}\label{lem:llip}
Let $\bm\sigma^* \in \mathcal{A}_{\rm sort}\setminus\{\bm 0\}$ be arbitrary. 
Then there exists a constant $L_G>0$ such that 
\begin{align*}
    \|\nabla G(\bm W)\|_F 
    \;\le\; L_G \, \mathrm{dist}(\bm W, \mathcal{W}_{\bm\sigma^*}),
    \quad \text{for all } \mathrm{dist}(\bm W, \mathcal{W}_{\bm\sigma^*}) \le \sigma_{\max}^*,
\end{align*}
where 
\begin{align}\label{eq:defLG}
    L_G \;=\; 2\lambda L 
    \;+\; 2^{2L} L^2 (\sigma_{\max}^*)^{2L-2}
    \;+\; \sqrt{\lambda}\, y_1 \, 2^{L-1} L^2 (\sigma_{\max}^*)^{L-2}.
\end{align}
\end{lemma}

\begin{proof}
 Suppose $\|\bmw - \bmw^*\|_F = \mathrm{dist}(\bmw , \mathcal{W}_{\bm \sigma^*})$ with $\bm W^*\in \mathcal{W}_{\bm \sigma^*}$, for each $l \in [L]$, we have 
\begin{align}
&\|\nabla_l G(\bm W)\|_F
= \|\nabla_l G(\bm W)-\nabla_l G(\bm W^*)\|_F
\notag\\
\le{}&
2\lambda\|\bm W_l-\bm W_l^*\|_F
+2\left\|
\bm W_{L:l+1}^T
\left(\bm W_L\cdots\bm W_1-\sqrt{\lambda}\bm Y\right)
\bm W_{l-1:1}^T
\right.
\notag\\
&\left.
-\bm W_{L:l+1}^{*T}
\left(\bm W_L^*\cdots\bm W_1^*-\sqrt{\lambda}\bm Y\right)
\bm W_{l-1:1}^{*T}
\right\|_F
\notag\\
\le{}&
2\lambda\|\bm W_l-\bm W_l^*\|_F
+2\left\|
\bm W_{L:l+1}^T\bm W_{L:1}\bm W_{l-1:1}^T
-\bm W_{L:l+1}^{*T}\bm W_{L:1}^*
\bm W_{l-1:1}^{*T}
\right\|_F
\notag\\
&\quad
+2\sqrt{\lambda}\left\|
\bm W_{L:l+1}^T\bm Y\bm W_{l-1:1}^T
-\bm W_{L:l+1}^{*T}\bm Y\bm W_{l-1:1}^{*T}
\right\|_F.
\label{eq:nable_upbound}
\end{align}
By \Cref{prop:opti G}, we have 
$\|\bmw_{l}^*\|\le {\sigma}_{\max}^*
\ \text{and}\ \|\bmw_l\|\le  \|\bmw_l^*\|+\|\bmw -\bmw^*\|_F\le 2{\sigma}_{\max}^*,\ \forall l \in [L].
$
Using the triangle inequality and the matrix norm inequality $\|\bm A\bm B\|_F\le \|\bm A\|\|\bm B\|_F$, we have
\begin{align}
&\|\bmw_{L:l+1}^T\bmw_{L:1}\bmw_{l-1:1}^T-\bmw_{L:l+1}^{*T}\bmw_{L:1}^*\bmw_{l-1:1}^{*T}\|_F\notag\\
   \le& \|(\bmw_{l+1}-\bmw_{l+1}^*)^T\bmw_{L:l+2}^T\bmw_{L:1}\bmw_{l-1:1}^T\|_F+\|\bmw_{l+1}^{*T}(\bmw_{l+2}-\bmw_{l+2}^*)^T\bmw_{L:l+3}^T\bmw_{l-1:1}^T\|_F+\cdots\notag\\
   &+\|\bmw_{L:l+1}^{*T}\bmw_{L:1}^*\bmw_{l-2:1}^{*T}(\bmw_{l-1}-\bmw^*_{l-1})\|_F\notag\\
   \le& L2^{2L-1}{\sigma}_{\max}^{*2L-2}\|\bmw - \bmw^*\|_F.\label{eq:get_bound1}
\end{align}
Using a similar telescope inequality, we have 
\begin{align}
\sqrt{\lambda}\|\bmw_{L:l+1}^T\bm Y\bmw_{l-1:1}-\bmw_{L:l+1}^{*T}\bm Y\bm W_{l-1:1}^{*T}\|_F\le \sqrt{\lambda}L y_1 2^{L-2}{\sigma}_{\max}^{*L-2}\|\bmw -\bmw^*\|_F.\label{eq:get_bound2}
\end{align}
Plugging \eqref{eq:get_bound1} and \eqref{eq:get_bound2} into \eqref{eq:nable_upbound}, we obtain
\begin{align*}
    \|\nabla G(\bmw)\|_F \le \sum_{l=1}^L\|\nabla G_l(\bmw)\|_F \le 2 \left(\lambda L+ 2^{2L-1}L^2{\sigma}_{\max}^{*2L-2}+ \sqrt{\lambda} y_1 2^{L-2}L^2{\sigma}_{\max}^{*L-2}\right)\|\bmw -\bmw^*\|_F. 
\end{align*}
\end{proof}

\section{Convergence Rate Analysis under the Error Bound Condition}\label{app:sec C}

\begin{prop}[\cite{attouch2013convergence,attouch2009convergence,attouch2010proximal,bolte2014proximal}]\label{prop:threecondition} Suppose that Assumptions~\ref{AS:1} and ~\ref{AS:2} hold, and the sequence $\{\bm{W}^k\}_{k \ge 1}$ satisfies the following conditions:
\begin{itemize}
    \item[(i)] (\textsl{Sufficient Decrease}) There exists $\kappa_1 > 0$ with
    $
        F(\bmw^{k+1}) - F(\bmw^k) \leq -\kappa_1 \|\bmw^{k+1} - \bmw^k\|_F^2.
    $
    \item[(ii)] (\textsl{Relative Error}) There exists a constant $\kappa_2 > 0$ such that
    $$
        \|\nabla F(\bmw^{k+1})\|_F \leq \kappa_2 \|\bmw^{k+1} - \bmw^k\|_F.
    $$
\end{itemize}
Then, there exists $k_1>0$ such that for all $k\ge k_1$, $\{\bmw^k\}_{k \ge k_1}$ converges R-linearly to $\bmw^*$ for some $\bmw^* \in \mathcal{W}_F$, and $\{F(\bmw^k)\}_{k \ge k_1}$ converges Q-linearly to $F(\bmw^*)$.
\end{prop}
\begin{proof}
We observe that the regularization term renders the objective function $F(\bmw)$ level-bounded. Coupled with condition (i), which implies that $\{F(\bmw^k)\}$ is non-increasing and thus $\{\bmw^k\}$ remains within a compact sublevel set, we conclude that the sequence $\{\bmw^k\}$ is bounded. Consequently, there exists an accumulation point $\bmw^* \in \mathcal W_F$, since conditions (i) and (ii) imply $\|\nabla F(\bmw^k)\|_F \to 0$. By \Cref{thm:eb} and \Cref{coro:PL QG}, the P\L\ inequality holds near $\bmw^*$, implying the K\L\ property with exponent $\theta = 1/2$. Invoking \cite[Theorem 2.9]{attouch2013convergence}, the whole sequence converges to $\bmw^*$ with finite length, i.e., $\sum_{k=0}^\infty \|\bmw^{k+1} - \bmw^k\|_F < \infty$. Hence, there exists $k_1>0$ such that for all $k\ge k_1$, $\bmw^k$ stays within the neighborhood where the P\L\ inequality holds.

The $R$-linear convergence of $\{\bmw^k\}_{k\ge 0}$ to $\bmw^*$ follows from the analysis in \cite[Theorem 2]{attouch2009convergence}. Although their result is stated for proximal algorithms, the proof relies solely on conditions (i) and (ii), and thus can be directly applied in our setting.

The $Q$-linear convergence of $\{F(\bmw^k)\}_{k\ge k_1}$ follows by combining the P\L\ inequality (with modulus $\mu_1$), the relative error condition (ii), and the sufficient decrease condition (i). Specifically, we have
\[
    F(\bmw^{k+1}) - F(\bmw^*)\le\frac{1}{\mu_1} \|\nabla F(\bmw^{k+1})\|_F^2 \le \frac{\kappa_2^2}{\mu_1} \|\bmw^{k+1} - \bmw^k\|_F^2 \le \frac{\kappa_2^2}{\mu_1 \kappa_1} \left( F(\bmw^k) - F(\bmw^{k+1}) \right).
\]
Rearranging the terms yields
\[
    F(\bmw^{k+1}) - F(\bmw^*) \le \frac{\kappa_2^2}{\mu_1 \kappa_1 + \kappa_2^2} \left( F(\bmw^k) - F(\bmw^*) \right).
\]
Since $\frac{\kappa_2^2}{\mu_1 \kappa_1 + \kappa_2^2} < 1$, the sequence $\{F(\bmw^k)\}_{k\ge k_1}$ converges $Q$-linearly to $F(\bmw^*)$.
\end{proof}

\begin{remark}[Connection to Second-Order Methods]\label{rmk:second}
It is worth noting that the error bound condition also has significant applications in the convergence analysis of second-order algorithms. In fact, it is equivalent to the metric subregularity of the gradient mapping with respect to $\mathcal W_F$ (i.e., with exponent $p=1$) in \cite{li2024generalized}.  
Let $\bm \Theta$ be the set of second order stationary points. We claim that the metric subregularity of the gradient mapping with respect to $\bm \Theta$ also holds, i.e.,
for any $\bm{W}^* \in \bm{\Theta}$, there exists a neighborhood $\mathcal N$ of $\bm{W}^*$ such that for all $\bW\in \mathcal N$ we have 
 $\mathrm{dist}(\bm{W}, \bm{\Theta}) \le \kappa_1 \|\nabla F(\bm{W})\|_F.$
Since the objective function $F$ is smooth and the regularization terms ensure that any sublevel set is bounded, the Hessian of $F$ is Lipschitz continuous on some convex compact set which include the sublevel set about the initial point, which implies that it is Hölder continuous with exponent $q=1$ in the sense of \cite{li2024generalized}.
Therefore, by  \cite[Corollary 7.8]{li2024generalized}, the high-order regularization method with momentum in \cite[Algorithm 1]{li2024generalized} achieves a quadratic convergence rate to second-order stationary points, which complements the linear convergence of the first-order methods derived in this work.

It remains to prove the claim. Indeed, under Assumption \ref{AS:2}, it follows from \cite{chen2025complete} (specifically Corollary 2.3 and Proposition A.1) that every second-order critical point of $F(\bW)$ is a local minimizer. Our structural analysis of the critical point set (Proposition \ref{prop:set W}) reveals that it consists of a finite number of well-separated connected components, on each of which $F(\bW)$ is constant. This implies that the standard separation property and thus the weak separation property in equation (6.1) in \cite{li2024generalized} holds for Problem \eqref{eq:F}. 
With the above two facts, by \cite[Proposition~6.3]{li2024generalized}, for any $\bm{W}^* \in \bm{\Theta}$, there exists a neighborhood $\mathcal N$ of $\bm{W}^*$ such that for all $\bW\in \mathcal N$ we have $\mathrm{dist}(\bm{W}, \mathcal{W}_F) = \mathrm{dist}(\bm{W}, \bm{\Theta})$. Combined with the error bound in Theorem \ref{thm:eb}, the claim is proved. 
\end{remark}

\begingroup
\section{\texorpdfstring{Equivalence Between the P\L\ Inequality and the Error Bound}{Equivalence Between the PL Inequality and the Error Bound}}
\label{app:abs_PL_MSR}

This appendix completes the equivalence, for finite-dimensional $C^2$
functions around critical points, by proving that the
P\L\ inequality implies the local error
bound.
We first prove a simple auxiliary lemma
showing that a $C^2$ function that is flat up to second order at the origin
cannot satisfy a P\L\ inequality unless it is locally
constant.
\begin{lemma}\label{lem:flat_abs_PL}
Let \(E\) be a finite-dimensional Euclidean space, and let \(\rho>0\). Suppose that
\(g:B_\rho(\bm 0)\subseteq E\to\R\) is a \(C^2\) function satisfying 
\begin{equation}
\label{eq:g_zero}
    g(\bm 0)= 0,\quad \nabla g(\bm 0)=\bm 0,\quad \nabla^2 g(\bm 0)=\bm 0.
\end{equation}
Suppose further that the P\L\ inequality holds locally around \(\bm{\bm 0}\), i.e., there exists $\mu>0$ such that
\begin{equation}\label{eq:plnu}
    2\mu |g(\bm v)|\le \|\nabla g(\bm v)\|^2,
    \quad \forall \bm v\in B_\rho(\bm 0).    
\end{equation}
Then there exists \(\rho_0\in(0,\rho)\) such that \(g(\bm v)=0\) for all
\(\bm v\in B_{\rho_0}(\bm 0)\).
\end{lemma}
\begin{proof}
Suppose, for contradiction, that \(g\) is not identically zero in any neighborhood of the origin. For \(r\in(0,\rho)\), define
\begin{align}\label{eq:Mr}
    M(r):=\max_{\|\bm v\|\le r}|g(\bm v)|.    
\end{align}
Then \(M(r)>0\) for any sufficiently small \(r>0\). Moreover, Taylor's theorem, together with \eqref{eq:g_zero}, yields $M(r)=o(r^2)$ as $r \downarrow 0$.

We next show that \(M\) is locally Lipschitz on \((0,\rho)\). Indeed, for any
compact interval \([a,b]\subset(0,\rho)\), the function \(g\) is Lipschitz on
\(\overline B_b(\bm 0)\), say with constant \(L_b\). Hence, for any \(r_1,r_2\in[a,b]\) with
\(r_1\le r_2\), choosing \(\bm v_{r_2}\in \overline B_{r_2}(\bm 0)\) such that
\(|g(\bm v_{r_2})|=M(r_2)\), we have
\[
M(r_2)-M(r_1)
\overset{\eqref{eq:Mr}}{\le} |g(\bm v_{r_2})|-\left|g\left(\frac{r_1}{r_2}\bm v_{r_2}\right)\right| \le \left|g(\bm v_{r_2})
   -g\left(\frac{r_1}{r_2}\bm v_{r_2}\right)\right|
\le L_b\Bigl(1-\frac{r_1}{r_2}\Bigr)\|\bm v_{r_2}\|
\le L_b(r_2-r_1).
\]
Together with the monotonicity of \(M\), this shows that \(M\) is Lipschitz on
\([a,b]\). Therefore, by Rademacher's theorem, \(M\) is differentiable for
almost every \(r\in(0,\rho)\).

Fix a sufficiently small \(r>0\) at which \(M\) is differentiable. Let
\(\bm v_r\in\overline B_r(\bm 0)\) be such that \(|g(\bm v_r)|=M(r)\). Since \(M(r)>0\), we have \(g(\bm v_r)\neq0\). We claim that \(\|\bm v_r\|=r\). Suppose to the contrary that \(\bm v_r\) is an interior point of \(\overline B_r(\bm 0)\), i.e., \(\|\bm v_r\|<r\). Let \(s_r:=\operatorname{sign}(g(\bm v_r))\). Since
\[
s_rg(\bm v)\le |g(\bm v)|\le M(r)=s_rg(\bm v_r),\quad \forall\ \bm v\in \overline B_r(\bm 0),
\]
the function \(s_rg\) has a local maximum at the interior point \(\bm v_r\). Hence \(\nabla g(\bm v_r)=\bm 0\), and \eqref{eq:plnu} implies \(g(\bm v_r)=0\), contradicting \(|g(\bm v_r)| \neq 0\). Therefore, \(\|\bm v_r\|=r\).

Since \(\|\bm v_r\|=r\),  the KKT condition for 
\[
\max_{\bm v}\ s_rg(\bm v)\qquad\text{ s.t.}\quad \|\bm v\|^2-r^2\le0
\]
yields a multiplier \(\lambda_r\ge0\) such that
$
-s_r\nabla g(\bm v_r)+2\lambda_r \bm v_r=0.
$
Equivalently, with \(\alpha_r:=2\lambda_r r\ge0\), we have
$s_r\nabla g (\bm v_r)=\alpha_r \bm v_r/r.$ Since \(s_r=\pm1\) and \(\|\bm v_r/r\|=1\), it follows that $\alpha_r = \|\nabla g(\bm v_r)\|.$ Using \eqref{eq:plnu}, we obtain
\begin{align}\label{eq1:lem PL}
\alpha_r\ge \sqrt{2\mu |g(\bm v_r)|}=\sqrt{2\mu M(r)}.    
\end{align}
For \(t>0\) sufficiently small, we have
\[
\begin{aligned}
    M(r+t)  \ge \left|g\left(\bm v_r+t\frac{\bm v_r}{r}\right)\right| \ge s_rg\left(\bm v_r+t\frac{\bm v_r}{r}\right) =M(r)+t\left\langle\nabla(s_rg)(\bm v_r),\frac{\bm v_r}{r}\right\rangle+o(t) =M(r)+\alpha_rt+o(t),
\end{aligned}
\]
where the first inequality follows from $\|\bm v_r+t {\bm v_r}/{r}\| = r+t$ and \eqref{eq:Mr}, the first equality uses Taylor's expansion, and the last equality is due to $\nabla(s_rg)(\bm v_r) = s_r\nabla g(\bm v_r) = \alpha_r\bm v_r/r$.
Thus, at every sufficiently small point \(r\) at which \(M\) is differentiable,
dividing the inequality by \(t>0\) and letting \(t\downarrow0\) yields $M'(r)\ge \alpha_r\ge \sqrt{2\mu M(r)},$ where the last inequality uses \eqref{eq1:lem PL}. Since \(M(r)>0\) for all sufficiently small \(r>0\), it follows that, for almost every
sufficiently small \(r>0\),
\[
    \frac{d}{dr}\sqrt{M(r)}
    =
    \frac{M'(r)}{2\sqrt{M(r)}}
    \ge
    \sqrt{\frac{\mu}{2}}.
\]
Integrating this inequality from \(\tau\) to \(r\), and then letting
\(\tau\downarrow0\), yields $\sqrt{M(r)}\ge \sqrt{{\mu}/{2}}\,r$ for all sufficiently small \(r>0\). Hence \(M(r)\ge \mu r^2/2\), which contradicts
\(M(r)=o(r^2)\). Therefore, \(g\) must be identically zero in a neighborhood of the
origin.
\end{proof}

Using the preceding lemma, we show that the P\L\ inequality implies
the EB around an arbitrary critical point of a finite-dimensional $C^2$
function. The proof also yields the local Morse--Bott (MB)
property of the critical
set.
\begin{prop}[P\L\ implies EB and MB]
\label{prop:absolute_PL_implies_MSR}
Let \(E\) be a finite-dimensional Euclidean space, let
\(f:\mathcal U\to\R\) be a \(C^2\) function on an open set \(\mathcal U\subseteq E\),
and let \(\bm{\bar x}\in\mathcal U\) satisfy \(\nabla f(\bm{\bar x})=0\). Suppose that
 the P\L\ inequality holds locally around \(\bm{\bar x}\), i.e., there
exist constants \(\mu>0\) and \(\rho>0\) such that \(B_\rho(\bm{\bar x})\subseteq\mathcal U\)
and
\begin{equation}\label{eq:PL-critical}
    2\mu |f(\bm x)-f(\bm{\bar x})|\le \|\nabla f(\bm x)\|^2,
    \quad \forall \bm x\in B_\rho(\bm{\bar x}).
\end{equation}
Let \(\mathcal C:=\{\bm x\in\mathcal U:\nabla f(\bm x)=0\}\).  The following statements hold: \\
(i) {\bf Error bound:} There exist constants \(\epsilon>0\) and \(\kappa>0\) such that $\mathrm{dist}(\bm x,\mathcal C)\le \kappa\|\nabla f(\bm x)\|$ for all $\bm x\in B_\epsilon(\bm{\bar x}).$ \\
(ii) {\bf MB property:} After possibly shrinking \(\epsilon\), 
\(\mathcal C\cap B_\epsilon(\bm{\bar x})\) is a \(C^1\) submanifold and $T_{\bm z}\mathcal C=\ker\nabla^2 f(\bm z)$ for all $\bm z\in \mathcal C\cap B_\epsilon(\bm{\bar x}).$
\end{prop}
\begin{proof}
Without loss of generality, assume that \(\bm {\bar x}=\bm 0\) and \(f(\bm{\bar x}) = 0\), after translating the variable and subtracting the constant
\(f(\bm{\bar x})\). Let
\begin{equation}
    H:=\nabla^2 f(\bm 0),\quad
    \mathcal R:=\mathrm{Range}(H),\quad
    \mathcal K:=\ker(H).
\label{eq:PL-EB-decomp}
\end{equation}
Since \(H\) is self-adjoint, \(E=\mathcal R\oplus\mathcal K\) is an orthogonal
decomposition. Let \(\Pi_{\mathcal R}\) and \(\Pi_{\mathcal K}\) denote the orthogonal
projections onto \(\mathcal R\) and \(\mathcal K\), respectively. 

We first use an implicit-function reduction along the kernel and range of the Hessian. Every point \(\bm x\) near the origin can
be uniquely written as \(\bm x=\bm u+\bm v\), where \(\bm u\in\mathcal R\) and \(\bm v\in\mathcal K\).
Define
\begin{equation}
    \Phi(\bm u,\bm v):=\Pi_{\mathcal R}\nabla f(\bm u+\bm v),
    \quad (\bm u,\bm v)\in\mathcal R\times\mathcal K.
\label{eq:PL-EB-Phi}
\end{equation}
By \eqref{eq:PL-EB-Phi}, \(D_{\bm u}\Phi(\bm 0,\bm 0)=H|_{\mathcal R}\), which is nonsingular.
By the implicit function theorem, there exist neighborhoods \(U_{\mathcal R}\subseteq\mathcal R\)
and \(U_{\mathcal K}\subseteq\mathcal K\) of the origin, and a \(C^1\) mapping
\(h:U_{\mathcal K}\to U_{\mathcal R}\), with \(h(\bm 0)=\bm 0\), such that
\begin{equation}
    \Pi_{\mathcal R}\nabla f(\bm u+\bm v)=\bm 0
    \quad\Longleftrightarrow\quad
    \bm u=h(\bm v)
\label{eq:PL-EB-IFT}
\end{equation}
for all \((\bm u,\bm v)\in U_{\mathcal R}\times U_{\mathcal K}\).
To justify \(Dh(\bm 0)=\bm 0\), note first that, by \eqref{eq:PL-EB-Phi} and
\eqref{eq:PL-EB-decomp}, \(D_{\bm v}\Phi(\bm 0,\bm 0)=\Pi_{\mathcal R}H|_{\mathcal K}=\bm 0\).
Differentiating the identity \(\Phi(h(\bm v),\bm v)=\bm 0\) at \(\bm v=\bm 0\)
gives $    D_{\bm u}\Phi(\bm 0,\bm 0)\,Dh(\bm 0)\bm\xi
    +
    D_{\bm v}\Phi(\bm 0,\bm 0)\bm\xi
    =
    H|_{\mathcal R}\,Dh(\bm 0)\bm\xi
    =
    \bm 0$ for every \(\bm\xi\in\mathcal K\). Since \(H|_{\mathcal R}\) is nonsingular, \(Dh(\bm 0)\bm\xi=\bm 0\) for all
\(\bm\xi\in\mathcal K\), and hence \(Dh(\bm 0)=\bm 0\).
Define the local graph
\begin{equation}
    \mathcal M:=\{h(\bm v)+\bm v:\bm v\in U_{\mathcal K}\}.
\label{eq:PL-EB-M}
\end{equation}
By \eqref{eq:PL-EB-IFT}, every critical point sufficiently close to the origin belongs to
\(\mathcal M\). Thus, there exists a sufficiently small neighborhood \(V\) of the origin such that \(\mathcal C\cap V\subseteq\mathcal M\).

It remains to prove the reverse inclusion, namely that
\(\mathcal M\) consists entirely of critical points locally. Define the reduced function
\begin{equation}
    g(\bm v):=f(h(\bm v)+\bm v),\qquad \bm v\in U_{\mathcal K}.
\label{eq:PL-EB-g}
\end{equation}
For every \(\bm \xi\in\mathcal K\), using \eqref{eq:PL-EB-IFT} with
\(\bm u=h(\bm v)\) and
\(Dh(\bm v)\bm \xi\in\mathcal R\), we obtain
\[
\begin{aligned}
    Dg(\bm v)[\bm \xi]
    &=
    \left\langle \nabla f(h(\bm v)+\bm v),Dh(\bm v)\bm \xi+\bm \xi\right\rangle =
    \left\langle \Pi_{\mathcal K}\nabla f(h(\bm v)+\bm v),\bm \xi\right\rangle.
\end{aligned}
\]
Thus,
\begin{equation}
    \nabla g(\bm v)=\Pi_{\mathcal K}\nabla f(h(\bm v)+\bm v).
\label{eq:PL-EB-gradg}
\end{equation}
Since \(h\) is \(C^1\) and \(\nabla f\) is \(C^1\), the right-hand side of
\eqref{eq:PL-EB-gradg} is \(C^1\).
Therefore, \(g\) is \(C^2\). Moreover, \(g(\bm 0)= 0\) and \(\nabla g(\bm 0)=\bm 0\).
Differentiating \eqref{eq:PL-EB-gradg} at \(\bm v=\bm 0\) and applying
the chain rule gives, for all \(\bm\xi\in\mathcal K\),
\[
    \nabla^2g(\bm 0)[\bm \xi]
    =
    \Pi_{\mathcal K}H\big(Dh(\bm 0)\bm \xi+\bm \xi\big)
    =
    \bm 0,
\]
where the last equality follows from \(Dh(\bm 0)=\bm 0\).
Hence \(\nabla^2g(\bm 0)=\bm 0\). 
For \(\bm v\) sufficiently close to the origin, we obtain
\begin{equation}
\begin{aligned}
    2\mu |g(\bm v)|
    \overset{\eqref{eq:PL-EB-g}}{=}
    2\mu |f(h(\bm v)+\bm v)|
    \overset{\eqref{eq:PL-critical}}{\le}
    \|\nabla f(h(\bm v)+\bm v)\|^2
    \overset{\eqref{eq:PL-EB-IFT}}{=}
    \|\Pi_{\mathcal K}\nabla f(h(\bm v)+\bm v)\|^2
    \overset{\eqref{eq:PL-EB-gradg}}{=}
    \|\nabla g(\bm v)\|^2.
\end{aligned}
\label{eq:PL-EB-reduced-PL}
\end{equation}
By Lemma~\ref{lem:flat_abs_PL} applied with \eqref{eq:PL-EB-reduced-PL}, after shrinking \(U_{\mathcal K}\) if necessary,
we have \(g(\bm v)=\bm 0\) for all \(\bm v\in U_{\mathcal K}\). Therefore \(\nabla g(\bm v)=\bm 0\)
for all \(\bm v\in U_{\mathcal K}\). Using \eqref{eq:PL-EB-gradg}, we get
\(\Pi_{\mathcal K}\nabla f(h(\bm v)+\bm v)=\bm 0\) for all \(\bm v\in U_{\mathcal K}\).
Together with \eqref{eq:PL-EB-IFT}, this yields
\(\nabla f(h(\bm v)+\bm v)=\bm 0\) for all \(\bm v\in U_{\mathcal K}\).
Consequently, after shrinking the neighborhood \(V\) if necessary, we have
\begin{equation}
\label{eq:PL-EB-critical-set}
    \mathcal C\cap V=\mathcal M\cap V.
\end{equation}
Thus, the local critical point set is a \(C^1\) submanifold.

We now prove the error bound. Let \(\bm x\) be sufficiently close to the origin and
write \(\bm x=\bm u+\bm v\), where \(\bm u\in\mathcal R\) and \(\bm v\in\mathcal K\). Since
\(h(\bm v)+\bm v\in\mathcal C\cap V\) by \eqref{eq:PL-EB-M}
and \eqref{eq:PL-EB-critical-set}, we have
\begin{equation}
    \mathrm{dist}(\bm x,\mathcal C)\le \|\bm x-(h(\bm v)+\bm v)\|=\|\bm u-h(\bm v)\|.
\label{eq:PL-EB-dist}
\end{equation}
Using \eqref{eq:PL-EB-IFT} with \(\bm u=h(\bm v)\) and the fundamental theorem of calculus,
we obtain
\begin{equation}
\begin{aligned}
    \Pi_{\mathcal R}\nabla f(\bm u+\bm v)
    &=
    \Pi_{\mathcal R}\nabla f(\bm u+\bm v)
    -
    \Pi_{\mathcal R}\nabla f(h(\bm v)+\bm v) =
    A(\bm u,\bm v)(\bm u-h(\bm v)),
\end{aligned}
\label{eq:PL-EB-ftc}
\end{equation}
where $ A(\bm u,\bm v):=
    \int_0^1
    \Pi_{\mathcal R}\nabla^2 f\big(h(\bm v)+\bm v+t(\bm u-h(\bm v))\big)\big|_{\mathcal R}\,dt.$
The operator \(A(\bm u,\bm v)\) in \eqref{eq:PL-EB-ftc} converges to
\(H|_{\mathcal R}\) as \(\bm x\to\bm 0\).
Since \(H|_{\mathcal R}\) is nonsingular, after shrinking the neighborhood if
necessary, there exists \(c>0\) such that
\begin{equation}
    \|\bm u-h(\bm v)\|
    \le c\|\Pi_{\mathcal R}\nabla f(\bm u+\bm v)\|
    \le c\|\nabla f(\bm u+\bm v)\|.
\label{eq:PL-EB-normal-bound}
\end{equation}
Combining \eqref{eq:PL-EB-dist} and \eqref{eq:PL-EB-normal-bound},
\(\mathrm{dist}(\bm x,\mathcal C)\le c\|\nabla f(\bm x)\|\) for all \(\bm x\) sufficiently
close to the origin. This proves the error bound.

It remains to prove the MB property in statement (ii), namely
\(T_{\bm z}\mathcal C=\ker\nabla^2 f(\bm z)\) for all nearby
\(\bm z\in\mathcal C\). Define the parametrization \(\Psi:U_{\mathcal K}\to E\) by
\(\Psi(\bm v):=h(\bm v)+\bm v\). Then \(\Pi_{\mathcal K}\circ\Psi=\mathrm{id}_{U_{\mathcal K}}\).
Since \(\mathcal M\) is the graph of the \(C^1\) map
\(h:U_{\mathcal K}\to\mathcal R\), it is a \(C^1\) embedded submanifold
parametrized by \(\Psi\).
For \(\bm z=\Psi(\bm v)=h(\bm v)+\bm v\), we have
\begin{equation}
    T_{\bm z}\mathcal C
    =
    T_{\bm z}\mathcal M
    =
    D\Psi(\bm v)[\mathcal K]
    =
    \{Dh(\bm v)\bm \xi+\bm \xi:\bm \xi\in\mathcal K\}.
\label{eq:PL-EB-tangent}
\end{equation}
The map \(D\Psi(\bm v):\mathcal K\to E\) appearing in \eqref{eq:PL-EB-tangent} is injective.
Indeed, if
\(Dh(\bm v)\bm \xi+\bm\xi=0\), then applying \(\Pi_{\mathcal K}\) gives \(\bm \xi=\bm 0\), because
\(Dh(\bm v)\bm\xi\in\mathcal R\) and \(\bm \xi\in\mathcal K\). Therefore,
\begin{equation}
    \dim T_{\bm z}\mathcal C=\dim\mathcal K.
\label{eq:PL-EB-tangent-dim}
\end{equation}
Since every point \(h(\bm v)+\bm v\) on \(\mathcal M\) is critical by
\eqref{eq:PL-EB-critical-set}, we have
\(\nabla f(h(\bm v)+\bm v)=\bm 0\). Differentiating this identity in the
direction \(\bm \xi\in\mathcal K\) gives
\(\nabla^2 f(\bm z)[Dh(\bm v)\bm \xi+\bm \xi]=\bm 0\).
Thus
\begin{equation}
    T_{\bm z}\mathcal C\subseteq \ker\nabla^2 f(\bm z).
\label{eq:PL-EB-tangent-kernel-incl}
\end{equation}
On the other hand, by \eqref{eq:PL-EB-decomp} and continuity, after shrinking
the neighborhood if necessary, the operator
\(\Pi_{\mathcal R}\nabla^2 f(\bm z)\big|_{\mathcal R}:\mathcal R\to\mathcal R\)
remains nonsingular for all nearby \(\bm z\in\mathcal C\). Hence
\(\mathrm{rank}\,\nabla^2 f(\bm z)\ge \dim\mathcal R\). By the rank-nullity theorem,
\begin{equation}
    \dim\ker\nabla^2 f(\bm z)\le \dim E-\dim\mathcal R=\dim\mathcal K.
\label{eq:PL-EB-kernel-dim}
\end{equation}
Combining \eqref{eq:PL-EB-tangent-dim}, \eqref{eq:PL-EB-tangent-kernel-incl}, 
and \eqref{eq:PL-EB-kernel-dim} yields that
\(T_{\bm z}\mathcal C=\ker\nabla^2 f(\bm z)\)
for all nearby \(\bm z\in\mathcal C\). 
\end{proof}

\begin{remark}
The local MB property obtained above is a critical-set counterpart to
the MB viewpoint in \cite[Section~2.3]{rebjock2024fast}. There, the
relevant set is the local-minimizer set, whereas here it is the local
critical set \(\mathcal C=\{\bm x:\nabla f(\bm x)=0\}\) around an arbitrary critical point.
This distinction is important near saddle points, where
\(f(\bm x)-f(\bm{\bar x})\) need not be nonnegative. The absolute value in
\eqref{eq:PL-critical} is what allows the argument to cover such points. The
proof is self-contained. Its guiding idea is inspired by this MB viewpoint.
More precisely, we use the P\L\ inequality to reveal local MB-type geometry for
the critical set and then derive EB from normal nondegeneracy.
\end{remark}

We next record the corresponding equivalence between the EB and this
critical-set MB property.

\begin{prop}[Equivalence between EB and MB]
\label{prop:EB_MB_equiv}
Let \(E\) be a finite-dimensional Euclidean space and let
\(f:\mathcal U\to\mathbb R\) be \(C^2\) on an open set \(\mathcal U\subseteq E\).
Then, around any critical point \(\bm{\bar x}\), the EB is equivalent to the
critical-set MB property.
\end{prop}
\begin{proof}
Let $H(\bm x):=\nabla^2 f(\bm x)$ and
$\mathcal C:=\{\bm x:\nabla f(\bm x)=0\}$. The implication from EB to MB follows
from \cite[Theorem~1 and Corollary~1]{behling2013effect}. Indeed, applying their
result to the mapping $\nabla f$, the EB is exactly the calmness condition for
$(\nabla f)^{-1}$. Hence, after shrinking the neighborhood if necessary,
$\mathcal C$ is a $C^1$ submanifold near $\bar{\bm x}$, $H$ has constant rank on
$\mathcal C$ near $\bar{\bm x}$, and
$\operatorname{codim}\mathcal C=\operatorname{rank}H(\bm z)$ for all
$\bm z\in\mathcal C$ near $\bar{\bm x}$. Since $\nabla f\equiv \bm 0$ on
$\mathcal C$, the argument leading to \Cref{eq:PL-EB-tangent-kernel-incl} gives
$T_{\bm z}\mathcal C\subseteq \ker H(\bm z)$ for all $\bm z\in\mathcal C$ near
$\bar{\bm x}$. Moreover, if the ambient dimension is $n$, then
\[
    \dim T_{\bm z}\mathcal C
    =
    n-\operatorname{codim}\mathcal C
    =
    n-\operatorname{rank}H(\bm z)
    =
    \dim\ker H(\bm z).
\]
Together with the inclusion above, the equality of dimensions yields
$T_{\bm z}\mathcal C=\ker H(\bm z)$.

Conversely, suppose that the MB condition holds near $\bar{\bm x}$; that is,
$\mathcal C$ is a $C^1$ submanifold and
$T_{\bm z}\mathcal C=\ker H(\bm z)$ for all nearby $\bm z\in\mathcal C$. Then,
since \(H\) is continuous and the tangent spaces
\(T_{\bm z}\mathcal C\) vary continuously with \(\bm z\) on the \(C^1\)
submanifold \(\mathcal C\),
$H(\bm z)$ is uniformly nonsingular on the normal spaces
$N_{\bm z}\mathcal C=(T_{\bm z}\mathcal C)^\perp$. Thus, after shrinking the
neighborhood if necessary, there exists $\alpha>0$ such that
\begin{equation}\label{eq:normal_gap_MB}
|H(\bm z)\bm u|
\ge
\alpha|\bm u|,
\qquad
\forall \bm u\in N_{\bm z}\mathcal C,\quad
\forall \bm z\in\mathcal C \text{ near } \bar{\bm x}.
\end{equation}
For any $\bm x$ sufficiently close to $\bar{\bm x}$, after shrinking the
neighborhood if necessary, let \(\bm z\) be a closest point of \(\bm x\) in the
local submanifold \(\mathcal C\) around \(\bar{\bm x}\). Then
$\bm z$ is close to $\bar{\bm x}$ and $\bm x-\bm z\in N_{\bm z}\mathcal C$. Since $\nabla f(\bm z)=0$, Taylor's theorem
and \eqref{eq:normal_gap_MB} yield
\begin{align*}
|\nabla f(\bm x)|=
|\nabla f(\bm x)-\nabla f(\bm z)| \ge
|H(\bm z)(\bm x-\bm z)|
-
o(|\bm x-\bm z|) \ge
\frac{\alpha}{2}|\bm x-\bm z|, \quad
\forall  \bm x \text{ near }\bar{\bm x}.
\end{align*} Since
$|\bm x-\bm z|=\operatorname{dist}(\bm x,\mathcal C)$, we obtain $\operatorname{dist}(\bm x,\mathcal C)\le \frac{2}{\alpha}|\nabla f(\bm x)|,$ which is the EB.
\end{proof}

Together with the implication from the error bound to the P\L\ inequality
proved in Appendix~\ref{app:PL QG}, \Cref{prop:absolute_PL_implies_MSR} shows
that the two conditions are locally equivalent for finite-dimensional $C^2$
functions around arbitrary critical points. Moreover, \Cref{prop:EB_MB_equiv} shows that
these equivalent conditions are also locally equivalent to the critical-set
MB property. Thus, the three conditions are locally equivalent around
arbitrary critical points.

\endgroup

\section{Auxiliary Results}\label{app:sec auxi}

To bound the spectral gap between eigenvectors associated with repeated eigenvalues between two symmetric matrices, we introduce the Davis-Kahan theorem \cite{yu2015useful}. 

\begin{lemma}\label{lem:daviskahan}
 Let $\bm A, \hat{\bm A} \in \mathbb{R}^{p \times p}$ be symmetric with eigenvalues $\lambda_1 \geq \cdots \geq \lambda_p$ and $\hat{\lambda}_1 \geq \cdots \geq \hat{\lambda}_p$ respectively. Fix $1 \leq r \leq s \leq p$ and assume that $\min\{\lambda_{r-1} - \lambda_r, \lambda_s - \lambda_{s+1}\} > 0$, where $\lambda_0 := \infty$ and $\lambda_{p+1} := -\infty$. Let $d := s - r + 1$ and $\bm V = (\bm v_r, \bm v_{r+1}, \dots, \bm v_s) \in \mathbb{R}^{p \times d}$ and $\hat{\bm V} = (\hat{\bm v}_r, \hat{\bm v}_{r+1}, \dots, \hat{\bm v}_s) \in \mathbb{R}^{p \times d}$ have orthonormal columns satisfying $\bm A \bm v_j = \lambda_j \bm v_j$ and $\hat{\bm A} \hat{\bm v}_j = \hat{\lambda}_j \hat{\bm v}_j$ for all $j = r, r+1, \ldots, s$. Then, there exists an orthogonal matrix $\bm O \in \mathcal{O}^{d}$ such that
\[
\|\hat{\bm V} \bm O - \bm V\|_F \leq \frac{4 \|\hat{\bm A} - \bm A\|_F}{\min\{\lambda_{r-1} - \lambda_r, \lambda_s - \lambda_{s+1}\}}.
\]
\end{lemma} 
\begin{lemma}[Mirsky's Inequality \cite{stewart1990matrix}]\label{lem:mirsky}
Let $\bm{X}, \tilde{\bm{X}} \in \mathbb{R}^{m \times n}$ be arbitrary matrices with singular values
\[
\sigma_1 \geq \sigma_2 \geq \cdots \geq \sigma_{l} \quad \text{and} \quad \tilde{\sigma}_1 \geq \tilde{\sigma}_2 \geq \cdots \geq \tilde{\sigma}_{l},
\]
respectively, where $l = \min\{m, n\}$. Then, for any unitarily invariant norm (e.g., the Frobenius norm), we have
\[
\left\|\mathrm{diag}(\tilde{\sigma}_1 - \sigma_1, \dots, \tilde{\sigma}_l - \sigma_l)\right\| \leq \|\tilde{\bm X} - \bm X\|.
\]
\end{lemma}

\end{appendix}
\end{document}